\documentclass[11pt]{article}

\usepackage{epsfig}
\usepackage{amssymb, latexsym}
\usepackage{amscd}
\usepackage[all]{xy}
\usepackage{epsf}
\usepackage[mathscr]{eucal}

\usepackage{amsmath,amsfonts,amscd,amssymb,amsthm}

\usepackage[titletoc]{appendix}
\usepackage[sc]{titlesec}

\long\def\comment#1\endcomment{}

\theoremstyle{plain}

\newtheorem{theorem}{\sc Theorem}[section]
\newtheorem{lemma}[theorem]{\sc Lemma}

\newtheorem{prop}[theorem]{\sc Proposition}
\newtheorem{coroll}[theorem]{\sc Corollary}

\newtheorem{claim}[theorem]{\sc Claim}

\renewcommand{\thesection}{\sc \arabic {section}}
\renewcommand{\thetheorem}{\sc \arabic{theorem}}

\comment

\renewcommand{\thetheorem}{\thesection.{\sc \arabic {theorem}}}
\renewcommand{\thelemma}{\thesection.{\sc \arabic {lemma}}}
\renewcommand{\theconjecture}{\thesection.{\sc \arabic {conjecture}}}
\renewcommand{\theprop}{\thesection.{\sc \arabic {prop}}}
\renewcommand{\thecoroll}{\thesection.{\sc \arabic {coroll}}}
\renewcommand{\theklemma}{\thesection.{\sc \arabic {klemma}}}
\renewcommand{\thesublemma}{\thesection.{\sc \arabic {sublemma}}}
\renewcommand{\themtheorem}{\thesection.{\sc \arabic {mtheorem}}}

\endcomment

\theoremstyle{plain}

\newtheorem{defn}[theorem]{\sc Definition}

\theoremstyle{exercise}
\newtheorem{remark}[theorem]{\sc Remark}
\newtheorem{remarks}[theorem]{\sc Remarks}
\newtheorem{example}[theorem]{\sc Example}

\makeatletter \@addtoreset{equation}{section} \makeatother

\def\eqref#1{\thetag{\ref{#1}}}

\let\latexref=\ref
\def\ref#1{{\normalfont{\latexref{#1}}}}

\newcommand{\ldot}{{\:\raisebox{2,3pt}{\text{\circle*{1.5}}}}}
\newcommand{\udot}{{\:\raisebox{3pt}{\text{\circle*{1.5}}}}}
\newcommand{\mb}{{\:\raisebox{3pt}{\text{\circle*{1.5}}}}}

\def\dlim_#1{{\displaystyle\lim_{#1}}^\hdot}

\newcommand{\cchar}{\operatorname{\sf char}}

\newcommand{\gr}{\operatorname{\sf gr}}

\newcommand{\Ext}{\operatorname{Ext}}

\newcommand{\id}{\operatorname{\rm id}}

\newcommand{\eqto}{\mathrel{\stackrel{\sim}{\to}}}

\newcommand{\Ob}{\mathrm{Ob}}

\newcommand{\Mod}{\mathrm{Mod}}
\newcommand{\opp}{\mathrm{opp}}
\renewcommand{\Bar}{\mathrm{Bar}}
\newcommand{\Cobar}{\mathrm{Cobar}}

\newcommand{\Hom}{\mathrm{Hom}}
\newcommand{\Bimod}{\mathrm{Bimod}}

\newcommand{\RHom}{\mathrm{RHom}}
\newcommand{\Hoch}{\mathrm{Hoch}}
\newcommand{\Lie}{\mathrm{Lie}}

\newcommand{\dg}{\mathrm{dg}}

\newcommand{\Cone}{\mathrm{Cone}}

\newcommand{\Vect}{\mathscr{V}ect}

\newcommand{\Coalg}{{\mathscr{C}oalg}}

\newcommand{\Alg}{{\mathscr{A}lg}}

\newcommand{\Cat}{{\mathscr{C}at}}

\renewcommand{\top}{\mathrm{top}}

\newcommand{\Fun}{{\mathscr{F}un}}

\newcommand{\nilp}{\mathrm{nilp}}

\newcommand{\simpl}{{\Delta^\opp}}
\renewcommand{\k}{\Bbbk}
\newcommand{\hoe}{\mathsf{hoe}}

\newcommand{\alg}{\mathrm{alg}}

\newcommand{\coalg}{\mathrm{coalg}}
\newcommand{\Assoc}{\mathrm{Assoc}}

\renewcommand{\cchar}{\mathrm{char}\ }

\newcommand{\Dr}{{\mathscr{D}r}}
\newcommand{\Res}{\mathrm{Res}}
\renewcommand{\gr}{\mathrm{gr}}
\newcommand{\lhs}{\mathrm{l.h.s.}}
\newcommand{\rhs}{\mathrm{r.h.s.}}
\newcommand{\pprime}{{\prime\prime}}
\newcommand{\EM}{\mathrm{EM}}
\newcommand{\tot}{\mathrm{tot}}

\title{\sc{Graded Leinster monoids and generalized Deligne conjecture for 1-monoidal abelian categories}}

\author{\sc{Boris Shoikhet}}
\date{}

\begin{document}\maketitle
{\footnotesize
\begin{center}{\parbox{4,5in}{{\sc Abstract.}
In our recent paper [Sh1] a version of the ``generalized Deligne conjecture'' for abelian $n$-fold monoidal categories is proven. For $n=1$ this result says that, given an abelian monoidal $\k$-linear category $\mathscr{A}$ with unit $e$, $\cchar\k=0$, $\RHom_{\mathscr{A}}(e,e)$ is the first component of a Leinster 1-monoid in $\Alg(\k)$ (provided a rather mild condition on the monoidal and the abelian structures in $\mathscr{A}$, called homotopy compatibility, is fulfilled).

In the present paper, we introduce a new concept of a {\it graded} Leinster monoid. We show that the Leinster monoid in $\Alg(\k)$, constructed 
by a monoidal $\k$-linear abelian category in [Sh1], is graded.
We construct a functor, assigning an algebra over the chain operad $C_\ldot(E_2,\k)$, to a graded Leinster 1-monoid in $\Alg(\k)$, which respects the weak equivalences.
Consequently, this paper together with [Sh1] provides a complete proof of the ``generalized Deligne conjecture'' for 1-monoidal abelian categories, in the form most accessible for applications to deformation theory (such as Tamarkin's proof of the Kontsevich formality).
}}
\end{center}
}

\section*{\sc Introduction}
In Subsections \ref{subsection01} and \ref{subsection02}, we provide an introduction to (generalized) Deligne conjecture. The reader is referred to the Introduction in [KT], for a historical overview. 

Starting with Subsection \ref{subsection03}, we focus on an overview of the results and methods of this paper.

\subsection{}\label{subsection01}
The ``classical'' Deligne conjecture for Hochschild cochains has now several affirmative solutions, e.g. [T1,2], [MS1], [KoS]. It says that the Hochschild cochain complex $\Hoch^\udot(A)$ of any (dg) associative algebra $A$ over a field $\k$ of characteristic 0 admits an action of the chain operad $C_\ldot(E_2,\mathbb{Q})$ of the topological little discs operad $E_2$. 
\footnote{In its original form, the conjecture was formulated for associative rings (over $\k=\mathbb{Z}$), and was proven in [MS1]. In this paper, we discuss only the case when $\k$ is a field of characteristic 0.
}

The homology of the operad $E_2$ is the operad $\mathsf{e}_2$ of Gerstenhaber algebras (or 2-algebras). It is a Koszul quadratic operad, and it is generated by two binary operations in $\mathsf{e}_2(2)$, of degree 0 and of degree -1. (One can describe easily all binary  operations in the operad $H_\ldot(E_2,\mathbb{Q})$: the topological space $E_2(2)$ is homotopically equivalent to a circle $S^1$, and the two mentioned operations are generators in $H_0(S^1)$ and $H_1(S^1)$, correspondingly).

On the side of algebras over the homology operad $H_\ldot(E_2,\mathbb{Q})$, the generator in $H_0(S^1,\mathbb{Q})$ defines a commutative product $m\colon T\otimes T\to T$ of degree 0, and the generator in $H_1(S^1,\mathbb{Q})$ defines a Lie bracket $[-,-]\colon \Lambda^2(T[1])\to T[1]$ (which is an operation on $T$ of degree -1). The commutativity and the associativity of $m$ and the skew-commutativity and the Jacobi identity for $[-,-]$ are derived as identities in $H_\ldot(E_2(3),\mathbb{Q})$.
In fact, one more identity holds in $H_\ldot(E_2(3),\mathbb{Q})$, which is translated to the odd Leibniz rule: for any $H_\ldot(E_2,\mathbb{Q})$-algebra $T$
\begin{equation}\label{lr}
[a,b\wedge c]=[a,b]\wedge c+(-1)^{(|a|-1)\cdot |b|}b\wedge[a,c]
\end{equation}
where we use the notation $a\wedge b=m(a,b)$, and $a,b,c$ are homogeneous elements in $T$ of degrees $|a|,|b|,|c|$ correspondingly.

\begin{defn}{\rm
A Gerstenhaber algebra, or a 2-algebra $T$ is a graded vector space endowed with a commutative product $m\colon S^2(T)\to T$ and a Lie algebra structure $[-,-]$ on $T[1]$ such that the compatibility \eqref{lr} holds.
}
\end{defn}
This is the original Gerstenhaber's definition [G].

It was proven by F.Cohen in 1976 [C] that this definition is equivalent to the operadic definition of a Gerstenhaber algebra as a $H_\ldot(E_2,\mathbb{Q})$-algebra.

Gerstenhaber proved [G] in 1964 that the Hochschild cohomology (=the cohomology of the Hochschild cochain complex) of any algebra $A$ is a Gerstenhaber algebra.

The operations $m$ and $[-,-]$ can be lifted to the Hochschild cochain complex, as an associative homotopy commutative product on $\Hoch^\udot(A)$, called the cup-product, and as a Lie bracket on $\Hoch^\udot(A)[1]$, called the Gerstenhaber bracket. The game starts with an observation that the identity \eqref{lr} {\it fails} on the level of cochains. The question was how to formulate and to prove the property that \eqref{lr} holds on the Hochschild cochains ``up to homotopy''.

P.Deligne conjectured (unpublished, in a 1993 letter to several mathematicians) that the chain operad $C_\ldot(E_2,\k)$ acts on $\Hoch^\udot(A)$ for any $A$.  An evidence for this conjecture was that Gerstenhaber's [G] and Cohen's [C] results which taken together state that the homology operad of $C_\ldot(E_2,\k)$ acts on the homology complex of $\Hoch^\udot(A)$.

Nowadays it is known [T2], [KoS] that the operad $C_\ldot(E_2,\k)$ is quasi-isomorphic to the operad $\hoe_2(\k)$ of homotopy Gerstenhaber algebras (when $\cchar \k=0$), and a more common reformulation of the classical Deligne conjecture is that $\Hoch^\udot(A)$ is an algebra over the operad $\hoe_2(\k)$, compatible with Gerstenhaber's action of $\mathsf{e}_2(\k)$ on the Hochschild cohomology $HH^\udot(A)$.

There are several known proofs of the classical Deligne conjecture for Hochschild cochains: [T2], [KoS], [MS1,2], [BB], [BF].

\subsection{\sc }\label{subsection02}
The paper is devoted to is the ``generalized Deligne conjecture'' which (in its simplest form) is a statement for monoidal abelian $\k$-linear categories, see Theorem \ref{theorem2intro} below for precise formulation.

The link with the classical Deligne conjecture for Hochschild cochains is as follows.

Let $A$ be an associative algebra over $\k$. Consider the category $\mathscr{A}=\Bimod(A)$ of $A$-bimodules. It is an abelian $\k$-linear category. It is well-known that the Hochschild cohomology of $A$ can be intrinsically defined as
\begin{equation}\label{intro1}
\mathrm{HH}^\udot(A)=\Ext^\udot_{\Bimod(A)}(A,A)
\end{equation}
where $A$ is considered as the ``tautological'' $A$-bimodule. The Hochschild cochain complex itself can be defined as
\begin{equation}\label{intro2}
\Hoch^\udot(A)=\RHom^\udot_{D(\Bimod(A))}(A,A)
\end{equation}
as the derived Hom functor in the derived category of $A$-bimodules.

On the other hand, the category $\mathscr{A}=\Bimod(A)$ is monoidal: for any two $A$-bimodules $M,N$ their monoidal product is given as $M\otimes_A N$. The tautological $A$-bimodule $A$ is the unit for this monoidal product.

That is, what we have in the right-hand side of \eqref{intro2}, can be interpreted as follows. There is an abelian $\k$-linear category $\mathscr{M}$, which is also a monoidal category, with a unit $e$. The monoidal product is not assumed to be an exact bi-functor, but some ``homotopy exactness'' condition should be fulfilled. In particular, in our example with $\mathscr{A}=\Bimod(A)$, the monoidal product is right exact bi-functor. Then the right-hand side of \eqref{intro2} looks like
\begin{equation}\label{intro3}
\RHom^\udot_{\mathscr{M}}(e,e)
\end{equation}
It had been a folklore statement in mathematical circles that \eqref{intro3} is a homotopy 2-algebra over $\k$. We call this statement ``the generalized Deligne conjecture for 1-monoidal categories''.
This paper provides (along  with our previous paper [Sh1]) a precise formulation and a proof of this statement.

\subsection{\sc}\label{subsection03}
Throughout the paper, $\k$ is any field of characteristic 0.

In the recent years, several papers providing generalizations of the Deligne conjecture for Hochschild cochains have appeared, see e.g. [KT], [L2], [Fr], [GTZ], [BM], [BF], [FV], among the others. 

Roughly, there are known two different approaches to the generalized Deligne conjecture.  The first one is based on the theory of $\infty$-categories and $\infty$-operads of A.Joyal and J.Lurie [L1], [L2], [Fr], [GTZ], [GH]. The second one is based on the theory of higher operads of M.Batanin [Ba1,2], [BM], [BB]. We also mention the paper [BFSV] where the concept of a higher-monoidal category was introduced, and the recent paper [FV] providing an account on the additivity theorem for $E_n$-algebras.

Our approach to generalized Deligne conjecture, started in [Sh1], is closely related to the one proposed in [KT], for the cartesian monoidal case. Our initial goal was to make loc.cit. work for the $\k$-linear case as well. We made use of Leinster monoids as a proper replacement of Segal monoids in loc.cit. for the non-cartesian case.  As well, we introduced in [Sh1] a new refinement of the Drinfeld dg quotient [Dr] (having nicer monoidal properties), which substitutes the Dwyer-Kan simplicial localization made use in [KT].

In this paper, we develop our approach in [Sh1] further, by introducing a new concept of a {\it graded} Leinster monoid. 

An advantage of our approach, compared with the $\infty$-categorical one, is that it allows to
construct a {\it honest algebra} structure over a {\it honest operad}, quasi-isomorphic to $C_\ldot(E_2,\k)$, on an explicit dg vector space, quasi-isomorphic to $\RHom^\udot_{\mathscr{A}}(e,e)$, for an abelian $\k$-linear monoidal category $\mathscr{A}$ with unit $e$, in a functorial way. It makes possible, in particular, to mimic Tamarkin's proof [T1] of the Kontsevich formality for Hochschild cochains [Ko1], in possibly more general context of monoidal abelian categories over $\k$ (and, in prospect, in the context of $n$-fold monoidal [BFSV] $\k$-linear abelian categories, for $n>1$). 

The difference between the objects of the real world and the objects of the $\infty$-categorical world  can be illustrated here by the difference between the {\it graded} Leinster monoids in $\Vect(\k)$ and the general Leinster monoids in $\Vect(\k)$, see Section \ref{introlmg} below.

\subsection{\sc }
Recall  the main result of [Sh1].
\footnote{
In [Sh1], a more general statement for $n$-fold monoidal categories, $n\ge 1$, is proven. As we deal here only with $n=1$ case, the result from [Sh1] is quoted below only for this, technically simpler, case.
}

Let $\mathscr{A}$ be an abelian $\k$-linear category, $\k$ a field of characteristic 0. Suppose a monoidal structure $(\otimes,e)$ on the underlying $\k$-linear category is given. Denote by $\mathscr{A}^\dg$ the dg category of bounded from above complexes in $\mathscr{A}$, and let $\mathscr{I}\subset \mathscr{A}^\dg$ be the full dg subcategory of acyclic objects. Recall the definition of the {\it weak compatibility} between the exact and the monoidal structure in $\mathscr{A}$, see [Sh1, Def. 5.1]. The exact and the monoidal structure in $\mathscr{A}$ are said to be {\it weakly compatible} if there is a full additive $\k$-linear category $\mathscr{A}_0\subset \mathscr{A}$ such that the following conditions are fulfilled:
\begin{itemize}
\item[(i)] $\mathscr{A}_0$ is monoidal subcategory of $\mathscr{A}$, $e\in\mathscr{A}$ belongs to $\mathscr{A}_0$ and is a monoidal unit in $\mathscr{A}_0$,
\item[(ii)] denote $\mathscr{J}=\mathscr{A}_0^\dg\cap\mathscr{I}$, then $\mathscr{J}$ is a 2-sided ideal in $\mathscr{A}_0^\dg$ with respect to the monoidal product; that is, for an acyclic $I\in\mathscr{J}$ and any $X\in\mathscr{A}_0^\dg$ both complexes $X\otimes I$ and $I\otimes X$ are acyclic (here $\mathscr{A}_0^\dg$ is the full dg subcategory of $\mathscr{A}^\dg$, whose objects are bounded from above complexes in $\mathscr{A}_0$),
\item[(iii)] the natural embedding of pairs of dg categories
$$
(\mathscr{A}_0^\dg,\mathscr{J})\hookrightarrow (\mathscr{A}^\dg,\mathscr{I})
$$
induces a quasi-equivalence of the dg quotients 
$$
\mathscr{A}_0^\dg/\mathscr{J}\tilde{\rightarrow} \mathscr{A}^\dg/\mathscr{I}
$$
\end{itemize}

The $n=1$ case of our main result in [Sh1] reads:\footnote{The assumption for $\mathscr{A}$ to be essentially small is not necessary for this result, but it is technically simpler. In fact, Theorem \ref{theorem1intro} holds for any Grothendieck category $\mathscr{A}$, with a minor modification of the proof.}
\begin{theorem}\label{theorem1intro}
Let $\k$ be a field of characteristic 0, and let $\mathscr{A}$ be essentially small $\k$-linear monoidal abelian category, $e$ its unit object. Suppose the abelian and the monoidal structures are weakly compatible. Then there is a Leinster monoid $X_L$ in the monoidal category $\Alg(\k)$ of dg algebras over $\k$, whose first component $X_1$ is $\RHom^\udot_{\mathscr{A}}(e,e)$ as dg algebra, with the Yoneda dg algebra structure on the latter complex.
\end{theorem}

We recall the definition of a Leinster monoid in a symmetric monoidal category in Section \ref{lmintro}.

\subsubsection{\sc Examples}
Here we give two examples of monoidal abelian categories and discuss the fulfillment of the weak compatibility condition.
\begin{itemize}
\item[(1)] $A$ an associative dg algebra over $\k$, $\mathscr{A}=\Bimod(A)$, with the monoidal structure defined above. 
The weak compatibility holds if we take for $\mathscr{A}_0$ the full subcategory of $\mathscr{A}$ formed by the bimodules which are left flat and right flat as $A$-modules. Note that we can not take for $\mathscr{A}_0$ the category of projective (or flat) bimodules, as in the latter case the monoidal unit (the tautological bimodule $A$) does not belong to $\mathscr{A}_0$;
\item[(2)] $B$ an associative bialgebra, and $\mathscr{A}=\Mod(B)$ be the category of left modules over the underlying algebra $B$.
For $M,N\in\Mod(B)$, define a left $B$-module structure on the vector space $M\otimes_kN$ (which is in general a $B\otimes_kB$-module) via the coproduct
$$
\Delta\colon B\to B\otimes B
$$
The compatibility axiom in bialgebra is expressed by saying that $\Delta$ is a map of associative algebras, which makes possible to restrict the $B\otimes_kB$-module structure on $M\otimes_kN$ to a $B$-module structure. This monoidal product is exact, the unit is $k$ regarded as a $B$-module via the counit map $\varepsilon \colon B\to k$, which is a map of algebras. Thus, the weak compatibility holds.
.\end{itemize}

\comment
\subsubsection{\sc A more general setting}
In fact the proof of Theorem \ref{theorem1intro} in [Sh1] goes in a more general than abelian $\k$-linear categories setting of exact $\k$-linear categories.

In [Sh1], we pass from an abelian $\k$-linear category $\mathscr{A}$ to the corresponding dg category $\mathscr{A}^\dg$ and consider the pair
$(\mathscr{A}^\dg, \mathscr{I}^\dg)$ where $\mathscr{I}^\dg$ is the full subcategory in $\mathscr{A}^\dg$ of acyclic objects.

Then our assumption of weak compatibility with the monoidal structure is formulated as follows:

There exist a full pre-triangulated dg subcategory $i\colon \mathscr{A}_0^\dg\hookrightarrow\mathscr{A}^\dg$ such that:
\begin{itemize}
\item[(1)] $i$ is a quasi-equivalence of dg categories,
\item[(2)] $\mathscr{A}_0^\dg$ is closed under the monoidal product and contains the unit $e$,
\item[(3)] $\mathscr{I}_0^\dg:=\mathscr{I}^\dg\cap \mathscr{A}_0^\dg$ is the monoidal two-sided ideal in $\mathscr{A}_0^\dg$.
\end{itemize}
In this setting, our proof in [Sh1] gives immediately that $\Hom^\udot_{\mathscr{A}^\dg/\mathscr{I}^\dg}(e,e)$ is a Leinster 1-monoid in $\Alg(\k)$,
where $\mathscr{A}^\dg/\mathscr{I}^\dg$ stands for the dg quotient, see [K1,2], [Dr].

We can start with an exact category $\mathscr{A}$, then pass to $\mathscr{A}^\dg$, and take for $\mathscr{I}^\dg$ the corresponding dg category of the subcategory of acyclic objects in the exact category $\mathscr{A}$.
\begin{claim}{\rm
For an essentially small exact $\k$-linear category $\mathscr{A}$ suppose that the conditions (1)-(3) above are fulfilled for some pre-triangulated full sub-category $\mathscr{A}_0^\dg\subset \mathscr{A}^\dg$. Then $\Hom^\udot_{\mathscr{A}^\dg/\mathscr{I}^\dg}(e,e)$ is a Leinster 1-monoid in $\Alg(\k)$.
}
\end{claim}
The proof is a straightforward rewriting of the proof in [Sh1].
\endcomment

\subsection{\sc Conversion of Leinster monoids in $\Alg(\k)$ to $C(E_2,\k)$-algebras}\label{introlmg}
The main topic of this paper is the question {\it how to convert a Leinster monoid $X_L$ in $\Alg(\k)$ into a $C_\ldot(E_2,\k)$-algebra}, whose underlying complex is quasi-isomorphic to the first component $X_1$ of the Leinster monoid $X_L$.

It can be always done in the context of $\infty$-categories, with the Lurie's definition of $E_n$-algebras [L2].
However, for applications like Tamarkin's proof of the Kontsevich formality, to have a honest (vs. $\infty$-categorical) Deligne conjecture may have its advantages.

For a general Leinster monoid in $\Alg(\k)$, we do not know any explicit construction of an algebra over an operad quasi-isomorphic to $C_\ldot(E_2,\k)$ on a space quasi-isomorphic to $X_1$.
Morally, we think on this situation as on a structure of an $E_2$-algebra, existing in the world of $\infty$-categories, and not having any ``classical'' counter-part.

On the other hand, we introduce here a new concept of a {\it graded} Leinster monoid (in $\Vect(\k)$, $\Alg(\k)$, $\Cat^\dg(\k)$, \dots).
The {\it grading} on a Leinster monoid $X_L$ in $\Vect(\k)$ (see Definition \ref{grl}) is an integral structure, which makes it closer to the standard bar-construction of an associative dg algebra $A$ over $\k$.  It makes possible, in particular, to define a ``bar-complex'' $\Bar(X_L)$, for a graded Leinster monoid $X_L$ in $\Vect(\k)$, which is always a dg coalgebra over $\k$. In the case of a graded Leinster monoid in $\Alg(\k)$, we make use of this bar-complex to define a $B_\infty$ algebra structure on the space $\Cobar(\Bar(X_L))$, quasi-isomorphic to $X_1$. The two main results of this paper (A) show how to do that, and (B) show that the Leinster monoids of Theorem \ref{theorem1intro}, constructed in [Sh1], are in fact graded Leinster monoids.

More precisely, we prove the following results.
\begin{theorem}[proven as Theorem \ref{mtheorem}, is referred to as Main Theorem]\label{theorem2intro}
Let $X_L$ be a graded Leinster monoid in the category $\Alg(\k)$ of dg algebras over $\k$, see Definitions \ref{grl}, \ref{grla}. Then there is a structure of a $B_\infty$-algebra on a space $\mathscr{Y}(X_L)$, quasi-isomorphic to $X_{1}$, such that the underlying homotopy commutative algebra on $\mathscr{Y}(X_L)$ is quasi-isomorphic (as $A_\infty$ algebra) to the underlying associative algebra structure in $X_{1}\in \Alg(\k)$. The correspondence
\begin{equation}
X_L\rightsquigarrow \mathscr{Y}(X_L)
\end{equation}
gives rise to a functor
\begin{equation}
\mathscr{Y}\colon L_1(\Alg(\k))\to B_\infty(\k)
\end{equation}
This functor maps weakly equivalent Leinster $1$-monoids in $\Alg(\k)$ to quasi-isomorphic $B_\infty$-algebras over $\k$.
\end{theorem}
The concept of a $B_\infty$-algebra was introduced in [GJ]. We recall the definition in Section \ref{sectionbraces}. It is well known [MS1,2] that a $B_\infty$-algebra $Y$ is an algebra over $C_\ldot(E_2,\k)$-algebra, $\cchar \k=0$ (we recall it in Section \ref{sms}).

Our next result shows the applicability of Theorem \ref{theorem2intro} for the Leinster monoids in $\Alg(\k)$ constructed in [Sh1].
\begin{theorem}[proven as Theorem \ref{propappf}]\label{theorem33intro}
Let $\mathscr{A}$ be a $\k$-linear abelian monoidal category, for which the weak compatibility is fulfilled, and let $X_L$ be the Leinster monoid in $\Alg(\k)$ from Theorem \ref{theorem1intro}. Then $X_L$ is graded, see Definitions \ref{grl}, \ref{grla}.
\end{theorem}
Theorem \ref{theorem33intro} expresses a nice property of the ``refined Drinfeld dg quotient'', the main new construction of [Sh1], we have not considered earlier. It turns out that this property is essential for making everything in this paper work. 

As a consequence of Theorems \ref{theorem1intro}, \ref{theorem2intro}, \ref{theorem33intro}, we get
\begin{coroll}[generalized Deligne conjecture for monoidal abelian categories]
Let $\k$ be a field of characteristic 0, and let $\mathscr{A}$ be essentially small $\k$-linear monoidal abelian category, $e$ its unit object. Suppose the abelian and the monoidal structures are weakly compatible. Then there is a dg vector space $Y=Y(\mathscr{A})$, quasi-isomorphic to $\RHom_{\mathscr{A}}(e,e)$, with a $B_\infty$-algebra structure on it, whose homotopy commutative algebra structure is quasi-isomorphic, as an $A_\infty$ algebra, to the Yoneda dg associative algebra on $\RHom_{\mathscr{A}}(e,e)$.
\end{coroll}

\begin{remark}{\rm
(1). Any Leinster monoid is ``homotopically graded'', in the sense that its cohomology Leinster monoid is graded, by trivial reasons. (2). Technically, the {\it grading} on $X_L$  is necessary for defining the ``correct'' signs in the chain (bar) differential.
It turns out that the (naive) chain differential of a simplicial vector space (or, rather, of a functor $X_L\colon \Delta_0^\opp\to\Vect(\k)$, underlying of a Leinster monoid) does not respect the signs which appear from the inner grading on the components $X_n$. In particular, if we take $X_n=A^{\otimes n}$, for a dg algebra $A$, the ordinary chain complex of the corresponding functor $\Delta_0^\opp\to\Vect(\k)$ has wrong signs, as it does not remember the $\mathbb{Z}^n$-grading of the elements of $A^{\otimes n}$.}
\end{remark}

\subsection{\sc The techniques}\label{goals}
Our main technical tool is the bar-complex $\Bar(X_L)$ of a {\it graded} Leinster monoid in $\Vect(\k)$, which is a dg coalgebra over $\k$. The grading on $X_L$ is necessary to define the correct signs with the terms of the chain differential, which mimic the bar-complex of a dg algebra. Without the grading, we can only define the naive chain complex, which is not the correct concept.

Then we construct an analogue of the Eilenberg-MacLane map (also known as the Eilenberg-Zilber, or the shuffle, map), providing a lax-monoidal structure on $\Bar(-)$, originally defined for the chain complexes of simplicial abelian groups. We do not know whether it is possible to deduce the standard properties for this map from their classical counter-parts (except for the case of zero grading, for which we have an alternative and more conceptual proof, discussed in Appendix). In any case, we prove here everything we need on the modified Eilenberg-MacLane map ``by hand''.
It results in a lot of computations, what shouldn't make the reader surprised, as the proofs of the corresponding classical results, given in [EM], are also rather computational.

For a graded Leinster monoid $X_L$ in $\Alg(\k)$, the Eilenberg-MacLane map endows $\Bar(X_L)$ with a dg algebra structure over $\k$. We show that in this case the dg coalgebra and the dg algebra structures on $\Bar(X_L)$ are compatible, making it a dg bialgebra over $\k$, see Theorem \ref{theoremcoalg}.

A general construction, we refer to as the Tamarkin-Kadeishvili construction
\footnote{The construction was firstly invented by D.Tamarkin in his approach to quantization of Lie bialgebras [T4], and two years later was independently re-discovered by T.Kadeishvili [Kad].}, endows $\Cobar(B_\coalg)$, for a dg bialgebra $B$, with a $B_\infty$ algebra structure. Here $B_\infty$ is the Getzler-Jones ``braces'' operad [GJ]. On the other hand, $\Cobar(\Bar(X_L))$ is quasi-isomorphic to $X_1$, by a generalization of the bar-cobar duality. Finally, the results of McClure-Smith [MS1,2] link $B_\infty$-algebras and $C_\ldot(E_2,\k)$-algebras. Then we link the associative algebra parts, by constructing an explicit $A_\infty$ morphism. It completes the proof of Theorem \ref{theorem2intro}.

The proof of Theorem \ref{theorem33intro} is, in contrary, conceptual. It reveals the fact that the ``refined Drinfeld dg quotient'' introduced in [Sh1] (which is a homotopy refinement on the Drinfeld dg quotient [Dr], having nicer monoidal properties), does not ``mix up'' the multi-gradings.

\subsection{\sc Organization of the paper}
Our general compositional principle was to put the computational proofs off to the end of the paper. Such are Sections \ref{proeff}, \ref{proofcoalg}, and \ref{proefmy}, which contain the proofs of Theorems \ref{emmnewtheorem}, \ref{theoremcoalg}, and  \ref{mainyon}, correspondingly. It should make the main body of the paper more transparent and readable. 

The contents of the individual Sections is as follows.

In Section 1, we recall basic definitions on Leinster monoids in a monoidal category with weak equivalence, following [Le].

In Section 2, we recall some practical consequences of the Quillen bar-cobar duality. We are mainly focused on a quasi-isomorphism $\Cobar(\Bar(A))\to A$, for which we provide a direct proof. The reason for providing a complete proof here (instead of giving a reference to the well-known result in the closed model categories) is that we later generalize this result for graded Leinster monoids, in Section 3.6, and the proof there follows the same line and relies on some results proven in Section 2.

The first two Sections do not contain any new results and have no claims to originality.

In Section 3, we define graded Leinster monoids in $\Vect(\k)$ and in $\Alg(\k)$. We define the bar-complex of a graded Leinster monoid in $\Vect(\k)$, our main tool in ``rectification'' of graded Leinster monoids, in Section 3.3, and prove the quasi-isomorphism $\Cobar(\Bar(X_L))\to X_1$, for a graded Leinster monoid $X_L$ in $\Vect(\k)$, in Section 3.4.

Section 4 develops the theory of the modified Eilenberg-MacLane map for graded Leinster monoids in $\Vect(\k)$. We remind the classical Eilenberg-MacLane map in Section 4.1, and formulate the main result on the modified Eilenberg-MacLane map in Section 4.2, Theorem \ref{emmnewtheorem}. Another very important result proven here, in Section 4.4, is Theorem \ref{theoremcoalg}, saying that the modified Eilenberg-MacLane map $\nabla^\Bar\colon \Bar(X_L)\otimes \Bar(Y_L)\to\Bar(X_L\otimes Y_L)$, for two graded Leinster monoids $X_L,Y_L$ in $\Vect(\k)$, is a map of dg coalgebras. As an immediate consequence we see that $\Bar(X_L)$, for $X_L$ a graded Leinster monoid in $\Alg(\k)$, is a dg bialgebra (Corollary 4.7).

In Section 5 and Section 6 we assemble all parts of the story and prove Theorem 2 above (as Theorem \ref{mtheorem}). 
We recall the Tamarkin-Kadeishvili construction in Section 5.3, which implies that $\Cobar(\Bar(X_L))$, for a graded Leinster monoid $X_L$ in $\Alg(\k)$, is a $B_\infty$ algebra.
The proof of the statement that the underlying dg associative algebra of the $B_\infty$ algebra $\Cobar(\Bar(X_L))$, for a graded Leinster monoid $X_L$ in $\Alg(\k)$, is $A_\infty$ quasi-isomorphic to $X_1$, is proven by a construction of an explicit $A_\infty$ morphism, in Theorem \ref{mainyon}.

In Section 7 we prove Theorem 3 above (as Theorem \ref{propappf}).

Sections 8, 9, 10 contain the computational proofs of Theorems \ref{emmnewtheorem}, \ref{theoremcoalg}, and \ref{mainyon}, correspondingly.

In Appendix we provide a more conceptual proof of Theorem \ref{emmnewtheorem}, for the particular case of zero grading. In this case, we were able to deduce our result from some results proven in [EM]. Note that it is still a non-trivial statement, as the chain differential of a Leinster monoid is not the same as the simplicial chain differential, but is equal to the latter one with the two extreme terms removed.

\subsubsection*{\sc Acknowledgement}
I am grateful to Michael Batanin, Sasha Beilinson, Pavel Etingof, Boris Feigin, Anton Gerasimov, Vladimir Hinich, Dima Kaledin, Bernhard Keller, Alexander Kuznetzov, Wendy Lowen, Stefan Schwede, Dima Tamarkin, Vadim Vologodsky for many inspiring discussions, along the years and on different occasions, which were helpful for this work. \\[2pt]
The work was partially supported by the Flemish Science Foundation (FWO) Research grant {\it Krediet aan Navorser} 19/6525.

\renewcommand{\thetheorem}{\thesection.{\sc \arabic {theorem}}}
\renewcommand{\thelemma}{\thesection.{\sc \arabic {lemma}}}
\renewcommand{\theconjecture}{\thesection.{\sc \arabic {conjecture}}}
\renewcommand{\theprop}{\thesection.{\sc \arabic {prop}}}
\renewcommand{\thecoroll}{\thesection.{\sc \arabic {coroll}}}
\renewcommand{\theklemma}{\thesection.{\sc \arabic {klemma}}}
\renewcommand{\thesublemma}{\thesection.{\sc \arabic {sublemma}}}
\renewcommand{\themtheorem}{\thesection.{\sc \arabic {mtheorem}}}

\section{\sc Monoidal categories and Leinster monoids}
Throughout the paper, $\k$ denotes a field of characteristic zero.

\subsection{\sc }\label{section11}
In this paper, we deal with $\k$-linear categories. 
Here we recall some basic definitions on enriched category theory, in a greater generality of $\mathscr{V}$-enriched categories, for a (symmetric closed) monoidal category $\mathscr{V}$.

An ordinary (= set-enriched) category $\mathscr{V}$ is called a {\it monoidal category} if there is an associative up to coherent isomorphisms monoidal product given by a bi-functor
$\otimes\colon \mathscr{V}\times \mathscr{V}\to\mathscr{V}$ with a unit object $e$. The associativity up to coherent isomorphisms means that there is an isomorphism of functors
$\phi\colon F_1\to F_2$, where $F_1,F_2\colon \mathscr{V}^{\times 3}\to\mathscr{V}$ are defined on objects as
$F_1(X,Y,Z)=X\otimes (Y\otimes Z)$ and $F_2(X,Y,Z)=(X\otimes Y)\otimes Z$ subject to the pentagon diagram, see [ML, Section XI.1] or [AM, Section 1.1.1]. That is, we have functorial isomorphisms $\phi_{XYZ}\colon X\otimes (Y\otimes Z)\to (X\otimes Y)\otimes Z$, for any three objects $X,Y,Z\in\mathscr{V}$ fulfilling the pentagon commutative diagram for any four objects, $X,Y,Z,W\in\mathscr{V}$, see loc.cit.
The isomorphisms $\phi_{XYZ}$ should be implemented to all our formulas. We suppress these associativity isomorphisms for brevity.
They can be uniquely inserted in a canonical way. By abuse of notations we write
$$
X\otimes (Y\otimes Z)=(X\otimes Y)\otimes Z
$$
instead of
$$
\phi_{XYZ}\colon X\otimes(Y\otimes Z)\simeq (X\otimes Y)\otimes Z
$$

An ordinary monoidal category is called {\it braided} if there are bi-functorial maps $\beta_{a,b}\colon a\otimes b\to b\otimes a$ (that is, $\beta$ is a map of bi-functors $F_1\Rightarrow F_2$, where $F_1(a,b)=a\otimes b$, $F_2(a,b)=b\otimes a$), satisfying the following axioms: 

(1) the compositions
$$
a\eqto e\otimes a\xrightarrow{\beta_{e,a}}a\otimes e\eqto a\text{  and  }
a\eqto a\otimes e\xrightarrow{\beta_{a,e}}e\otimes a\eqto a
$$
are equal to $\id_a$, for any $a\in\mathscr{V}$;

(2) the hexagon axiom, see [ML, XI.1].

A braided monoidal category is called {\it symmetric} if $\beta_{b,a}\circ \beta_{a,b}=\id_{a\otimes b}$, for any $a,b\in \mathscr{V}$.

\hspace{1mm}

For a (not necessarily symmetric) monoidal category $\mathscr{V}$, one defines {\it a $\mathscr{V}$-enriched category} $\mathscr{C}$ as a ``category'' whose morphisms $\mathscr{C}(A,B)\in\mathscr{V}$ and whose ``unit morphisms'' are given by maps $e\to \mathscr{C}(A,A)$ in $\mathscr{V}$ (where $e$ is the unit in $\mathscr{V}$). To define the compositions 
$$
\mathscr{C}(B,C)\times \mathscr{C}(A,B)\to\mathscr{C}(A,C)
$$
one uses the monoidal product in $\mathscr{V}$. See [Ke, Ch 1, 1.2] for more detail.

For any $\mathscr{V}$-enriched category $\mathscr{C}$, there is an ordinary (set-enriched) ``underlying'' category $\mathscr{C}_0$, defined as follows. 
The ordinary category $\mathscr{C}_0$ has the same objects as $\mathscr{C}$, and $\mathscr{C}_0(A,B)=\Hom_{\mathscr{V}}(e,\mathscr{C}(A,B))$. The composition in $\mathscr{C}_0$ is defined by the isomorphism $e\to e\otimes e$ in $\mathscr{V}$, see [Ke, Ch 1, 1.3] for more detail. 

Many definitions from the ordinary category theory are extended to the $\mathscr{V}$-enriched case. 
In particular, one easily defines {\it a $\mathscr{V}$-functor} $F\colon \mathscr{C}\to \mathscr{D}$ between $\mathscr{V}$-enriched categories.
A bit less straightforward is the definition of {\it a natural transformation} $\alpha\colon F\Rightarrow G\to \mathscr{C}\to \mathscr{D}$, 
for two $\mathscr{V}$-functors $F$ and $G$, see [Ke, Ch. 1, 1.2]. 

To define {\it a monoidal structure} on the category of $\mathscr{V}$-enriched categories, one needs to assume additionally that the monoidal category $\mathscr{V}$ is {\it symmetric} monoidal. 
Assume that the category $\mathscr{V}$ is symmetric monoidal, and $\mathscr{C},\mathscr{D}$ are two $\mathscr{V}$-enriched categories.

Define their product $\mathscr{C}\otimes \mathscr{D}$ as a $\mathscr{V}$-enriched category whose objects is the product $\Ob\mathscr{C}\times\Ob\mathscr{D}$, and
$$
(\mathscr{C}\otimes\mathscr{D})(A\times A_1,B\times B_1)=\mathscr{C}(A,B)\otimes \mathscr{D}(A_1,B_1)
$$
To define compositions
$$
(\mathscr{C}\otimes\mathscr{D})(B\times B_1,C\times C_1)\otimes(\mathscr{C}\otimes\mathscr{D})(A\times A_1,B\times B_1)\to
(\mathscr{C}\otimes\mathscr{D})(A\times A_1,C\times C_1)
$$
one needs to exchange the two middle factors in the product of four factors in the l.h.s., and then to apply the compositions in $\mathscr{C}$ and $\mathscr{D}$, correspondingly. One uses the symmetric braiding in $\mathscr{V}$ to exchange the two factors in the middle. See [Ke, Ch. 1, 1.4] for more detail.

It follows that for the case $\mathscr{V}$ is a symmetric monoidal category, one can define {\it a monoidal $\mathscr{V}$-enriched category} $\mathscr{M}$ as a monoid in the monoidal category of $\mathscr{V}$-enriched categories.

To extend the Yoneda lemma and the concept of an adjoint pair of functors for the $\mathscr{V}$-enriched setting, one needs to make one more assumption that $\mathscr{V}$ is {\it closed monoidal}. By definition, a monoidal category $\mathscr{V}$ is {\it closed monoidal}, if each functor $-\otimes Y$ (for fixed $Y$) has a right adjoint $[Y,-]$, so that one has functorial isomorphisms
$$
\mathscr{V}(X\otimes Y,Z)\simeq \mathscr{V}(X,[Y,Z])
$$
It gives rise to the adjunction maps
$$
d\colon X\to [Y,X\otimes Y]\text{   and   }ev\colon [Y,Z]\otimes Y\to Z
$$ 
called {\it the counit} and {\it the evaluation}. One deduces from here a natural isomorphism 
$$
Z\simeq [e,Z]
$$
(where $e$ is the unit in $\mathscr{V}$). One also deduces an extended adjunction
$$
[X\otimes Y,Z]\simeq [X,[Y,Z]]
$$
When $\mathscr{V}$ is symmetric monoidal, one defines a closed symmetric monoidal structure, which is a particular case of a biclosed monoidal category, see [Ke, Ch. 1,1.5].

For a symmetric monoidal closed category $\mathscr{V}$, $\mathscr{V}$ itself can be considered as a $\mathscr{V}$-enriched category. In this case, one defines a representable functor, an adjoint pair, and proves a $\mathscr{V}$-enriched Yoneda lemma, see [Ke, Ch. 1, 1.6-1.11]. 

For the category $\mathscr{V}=\Mod(R)$ of modules over a ring $R$, $\mathscr{V}$ is monoidal if $R$ is commutative, with $\otimes=\otimes_R$ and $e=R$. In this case, $\mathscr{V}$ is also symmetric and closed monoidal.

Our main examples are the category of $\k$-vector spaces $\mathscr{V}=\Vect(\k)$ and the category of complexes of $\k$-vector spaces $\mathscr{V}=\Vect^\udot(\k)$. The monoidal products are $M\otimes_\k N$ and the sum-total complex $\mathrm{Tot}(M\otimes_\k N)$, correspondingly. We say that a category $\mathscr{C}$ is {\it $\k$-linear} if it is enriched over $\mathscr{V}=\Vect(\k)$. (It is clear that a category enriched over $\mathscr{V}=\Vect^\udot(\k)$ can be regardered as a category enriched over 
 $\mathscr{V}=\Vect(\k)$, and thus is $\k$-linear).

In this paper, we often consider functors between monoidal categories which are {\it lax-} or {\it colax-monoidal} functors. Let us recall the definitions, for the case of ordinary categories. 

A functor $F\colon \mathscr{M}\to\mathscr{N}$ between two ordinary monoidal categories is called {\it colax-monoidal} if there is a map of bifunctors
$\beta_{X,Y}\colon F(X\otimes Y)\to F(X)\otimes F(Y)$ and a morphism $\alpha\colon F(e_{\mathscr{M}})\to e_{\mathscr{N}}$ such that the diagrams below commute:
\begin{itemize}
\item[(i)]
for any three $X,Y,Z\in\mathscr{M}$:
\begin{equation}\label{colax1}
\xymatrix{
F(X\otimes Y\otimes Z)\ar[rr]^{\beta_{X\otimes Y,Z}}\ar[d]_{\beta_{X,Y\otimes Z}}&&F(X\otimes Y)\otimes F(Z)\ar[d]^{\beta_{X,Y}\otimes \id}\\
F(X)\otimes F(Y\otimes Z)\ar[rr]^{\id\otimes \beta_{Y,Z}}&&F(X)\otimes F(Y)\otimes F(Z)
}
\end{equation}
\item[(ii)]
for any $X\in\Ob\mathscr{M}$ the following two diagrams are commutative
\begin{equation}\label{colax2}
\begin{aligned}
\ &\xymatrix{
F(e_{\mathscr{M}}\otimes X)\ar[d]\ar[r]^{\beta_{e,X}}& F(e_{\mathscr{M}})\otimes F(X)\ar[d]^{\alpha\otimes\id}\\
F(X)&      e_{\mathscr{N}}\otimes F(X)\ar[l]
} &
\xymatrix{
F(X\otimes e_{\mathscr{M}})\ar[d]\ar[r]^{\beta_{X,e}}& F(X)\otimes F(e_{\mathscr{M}})\ar[d]^{\id\otimes \alpha}\\
F(X)&     F(X)\otimes e_{\mathscr{N}}\ar[l]
}
\end{aligned}
\end{equation}
\end{itemize}

A functor $f\colon \mathscr{M}\to\mathscr{N}$ between two ordinary monoidal categories is called {\it lax-monoidal} if there is a map of bifunctors
$\gamma_{a,b}\colon F(a)\otimes F(b)\to F(a\otimes b)$ and a morphism $\kappa\colon e_{\mathscr{N}}\to F(e_{\mathscr{M}})$ such that the diagrams below commute:
\begin{itemize}
\item[(1)] for any three objects $X,Y,Z\in\Ob(\mathscr{M})$, the diagram
\begin{equation}\label{lax1}
\xymatrix{
F(X)\otimes F(Y)\otimes F(Z)\ar[rr]^{\id\otimes\gamma_{Y,Z}}\ar[d]_{\gamma_{X,Y}\otimes\id}&&F(X)\otimes F(Y\otimes Z)\ar[d]^{\gamma_{X,Y\otimes Z}} \\
F(X\otimes Y)\otimes F(Z)\ar[rr]^{\gamma_{X\otimes Y,Z}}&&F(X\otimes Y\otimes Z)
}
\end{equation}
is commutative. The functors $\gamma_{X,Y}$ are called the {\it lax-monoidal maps},
\item[(2)]
for any $X\in\Ob\mathscr{M}$ the following two diagrams are commutative
\begin{equation}\label{lax2}
\begin{aligned}
\ &\xymatrix{
F(1_{\mathscr{M}}\otimes X)\ar[d]& F(1_{\mathscr{M}})\otimes F(X)\ar[l]_{\gamma_{1,X}}\\
F(X)&      1_{\mathscr{N}}\otimes F(X)\ar[u]_{\kappa\otimes\id}\ar[l]
} &
\xymatrix{
F(X\otimes 1_{\mathscr{M}})\ar[d]& F(X)\otimes F(1_{\mathscr{M}})\ar[l]_{\gamma_{X,1}}\\
F(X)&     F(X)\otimes 1_{\mathscr{N}}\ar[u]_{\id\otimes \kappa}\ar[l]
}
\end{aligned}
\end{equation}

\end{itemize}

The corresponding definitions for $\mathscr{V}$-enriched monoidal categories, where $\mathscr{V}$ is a symmetric monoidal category, follow easily from the definitions given above, and are left to the reader.

\subsection{\sc The categories $\Delta$ and $\Delta_0$}\label{deltanot}
Here we recall the simplicial category $\Delta$ and its subcategory $\Delta_0$, and introduce notations which will be used throughout the paper. 

The category $\Delta$ has finite non-empty ordinals $[0],[1],[2],\dots$ as its objects, $$[n]=\{0<1<\dots<n\}$$
and the order-preserving maps as the morphisms. That is, a morphism $f\colon [m]\to [n]$ is a map
$$
f\colon \{0<1<\dots<m\}\to\{0<1<\dots<n\}
$$
such that
$$
f(i)\le f(j)\text{    whenever    }i\le j
$$
Denote by $f_i\colon [n]\to [n+1]$, $0\le i\le n+1$, the {\it elementary face maps}:
$$
f_i(j)=\begin{cases}j&\text{if  }j<i\\ j+1&\text{if  }j\ge i\end{cases}
$$
The elementary face maps $f_0,f_{n+1}\colon [n]\to[n+1]$ are called {\it extreme face maps}.

Denote by $\delta_i\colon [n]\to [n-1]$, $0\le i\le n-1$, the {\it elementary degeneracy maps}:
$$
\delta_i(j)=\begin{cases}j&\text{if  }j\le i\\j-1&\text{if  }j>i+1\end{cases}
$$
They satisfy the relations the reader can find e.g. in [W, Section 8.1]. 

For a functor $X\colon \Delta^\opp\to\mathscr{C}$, we denote
$$
F_i=X(f_i),\ D_i=X(\delta_i)
$$

Recall, for convenience of the reader, the standard simplicial identities among $f_i$ and $\delta_j$, and, dually, among $F_i$ and $D_j$, see e.g. loc.cit.:
\begin{equation}\label{simplid}
\begin{aligned}
\ &F_iF_j=F_{j-1}F_i\ \text{    if   }i<j\\
&D_iD_j=D_{j+1}D_i\ \text{     if   }i\le j\\
&F_iD_j=\begin{cases}
D_{j-1}F_i&\text{   if   }i<j\\
\id&\text{   if   }i=j\text{  or  }i=j+1\\
D_jF_{i-1}&\text{   if   }i>j+1
\end{cases}
\end{aligned}
\end{equation}

The category $\Delta_0$ a sub-category of the simplicial category $\Delta$ having the same objects, and generated by all simplicial maps except the extreme face maps. 

More directly, 
a morphism $[m]\to [n]$ in $\Delta_0$ is a map of ordinals $f\colon \{0<1<\dots<m\}\to \{0<1<\dots <n\}$ such that:
\begin{itemize}
\item[(1)] $f(i)\le f(j)$ whenever $i\le j$ (that is, $f$ is a morphism in $\Delta$),
\item[(2)] $f(0)=0$ and $f(m)=n$.
\end{itemize}

Whenever a confusion is possible, we denote the oobject $[n]$ of $\Delta_0$ by $[n]_*$. 
The category $\Delta_0$ is monoidal (unlike the simplicial category $\Delta$). The monoidal product in $\Delta_0$ is defined on objects as $[a]_*\otimes [b]_*=[a+b]_*$ on objects. It can be thought of as the identification of the rightmost element $a$ in $\{0<1<\dots<a\}$  with the leftmost element 0 in $\{0<1<\dots<b\}$. Due to condition (2) above, the product is extended to morphisms. In particular, this construction doesn't work for the entire category $\Delta$, for which condition (2) is dropped. 

\subsubsection{\sc The category $\Delta_0$ and the Joyal duality}
The category $\Delta_0$ which we have defined turns out to be equal to the category $\Delta_\varnothing^\opp$ where $\Delta_\varnothing$ is the simplicial category $\Delta$ with one more object thought of as the empty ordinal. This empty ordinal $[\varnothing]$ is added as the initial object to $\Delta$, and there are no morphisms in $\Hom_{\Delta_\varnothing}([n],[\varnothing])$ except for the case $n=\varnothing$. This duality appears as the simplest case of a more general Joyal duality, see [J].

Let us construct an equivalence $\Delta_0\simeq \Delta_\varnothing^\opp$. 
Let $I$ be the full subcategory of $\Delta_0$ containing all objects $[n]$ with $n\ge 1$. For a better readability, we re-denote here the objects $[n]$ of $I$ or $\Delta_0$ by $[n]_*$. First of all, we construct an equivalence $I\simeq \Delta^\opp$.

Consider the functor $F\colon \Delta^\opp\to I$, $F([n])=\Hom_\Delta([n],[1])$. The set $F([n])$ has $n+2$ different elements, which we identify with $[n+1]_*=\{0_*<1_*<\dots<(n+1)_*\}$. Here $i_*$ is corresponded to the map $q_i\colon [n]\to [1]$ in $\Delta$, such that $q_i(j)=0$ for $j<i$, and $q_i(j)=1$ for $j\ge i$. In particular, for $i=0$, $q_i(j)=1$ for any $j$, and $q_{n+1}(j)=0$ for any $j$. These two maps $q_0$ and $q_{n+1}$ are thought of as the end-points of the interval $[n+1]_*$. We denote them also $q_-$ and $q_+$, correspondingly. One easily checks that, for any map $\theta\colon [m]\to [n]$ in $\Delta$ one has: $F(\theta)(q_-)=q_-$ and $F(\theta)(q_+)=q_+$. One shows that the functor $F$ is indeed an equivalence $\Delta^\opp\simeq I$.

Turn back to construction of an equivalence $\Delta_\varnothing^\opp\simeq \Delta_0$. A ``drawback'' of the functor $F$ is that its image does not contain $[0]_*$, as the minimal $F([0])=[1]_*$. We can fix the problem by adding the empty ordinal $[\varnothing]$ to $\Delta$, and extending $F$ by setting $F([\varnothing])=[0]_*$. This extension is correct because $\varnothing$ is by definition the initial object in $\Delta_\varnothing$, and $[0]_*$ is the final object of $\Delta_0$. By abuse of notation, we denote this extended equivalence also by $F$. 

As the category $\Delta_0$ is monoidal, one should have a monoidal structure on $\Delta_\varnothing$ such that the corresponding monoidal structure on $\Delta_\varnothing^\opp$ is respected by $F$. One defines a monoidal product on $\Delta_\varnothing$ by $[m]\otimes [n]=[m+n+1]$ when $[\varnothing]$ is considered as $[-1]$. That is, $[\varnothing]$ is the monoidal unit of this structure. The reader will easily check that the functor $F$ is indeed strictly monoidal.

\subsection{\sc Leinster monoids}\label{lmintro}
Here we recall the definition of a {\it Leinster monoid} in a monoidal category $\mathscr{M}$, see [Le], [Sh1, Sect. 2]. We warn the reader that, by abuse of terminology, we use the terms ``Leinster algebras'' and ``Leinster monoids'' as synonymous.

As a motivation, consider what happens with the {\it nerve} of a small category in the $\k$-linear case.
Recall, that for a small set-enriched category $\mathscr{C}$ there is a nerve of $\mathscr{C}$. It is a simplicial set $X$, whose $n$-simplices $X_n$ are the sequences of $n$ composable morphisms. The simplest example is the case when $\mathscr{C}$ is a category with a single object, that is a monoid $M$.
In this case,
$$
X_n=M^{\times n}
$$
The simplicial face maps but the two extreme ones are defined via the monoidal product, and are well defined for categories with any enrichment.
The two extreme face maps $M^{\times n}\to M^{\times (n-1)}$ are defined as the projections along the first (respectively, along the last) factor.

It is instructive to consider the three simplicial face maps $F_0,F_1,F_2\colon M^{\times 2}\to M$. They are defined as follows:
\begin{equation}\label{threef}
\begin{aligned}
\ &F_0(x,y)=y\\
&F_1(x,y)=x*y\\
&F_2(x,y)=x
\end{aligned}
\end{equation}

Suppose now that $M$ is a monoid in a $\k$-linear category, that is, an associative algebra over $\k$. It is natural to replace the cartesian product in the definition of the nerve by the tensor product, and set
$$
X_n=M^{\otimes n}
$$
We refer to this ``nerve'' as a {\it $\k$-linear nerve}.
We claim that $X_\ldot$ fails to be a simplicial vector space.
Indeed, for $n=2$ the there maps  $F_0,F_1,F_2$ in \eqref{threef} should be maps $M\otimes M\to M$.
The maps $F_0$ and $F_2$ are not defined, as there are no projections $V\otimes W\to V$ and $V\otimes W\to W$ of vector spaces.
By the same reason, the two extreme face maps are not defined for any $n$.
As a consequence, the $\k$-linear nerve fails to be a simplicial vector space. The matter is that the tensor product of vector spaces is not their cartesian product.

The concept of a Leinster monoid serves as a suitable replacement of Segal monoids in a {\it non-cartesian-enriched} monoidal category $\mathscr{M}$. The goal of the concept of a Segal monoid is to define a ``weak monoid'', that is, a monoid with not strictly associative product, in an appropriate sense, see [Seg]. The concept of Leinster monoids provides a counter-part of Segal monoids for not necessarily cartesian-enriched case, for instance, for $\k$-linear categories. 

Here is the definition.

\begin{defn}\label{catwe}{\rm
\begin{itemize}
\item[(i)]
A category $\mathscr{C}$ is called {\it a category with weak equivalences} if there is a class $\mathscr{W}$ of morphisms in $\mathscr{C}$ closed under the composition and such that all identity morphisms belong to $\mathscr{W}$. A $\k$-linear category with weak equivalences should obey the following property: for a morphism $f\in\mathscr{W}$ and for $\lambda\in \k^*$, the morphism $\lambda\cdot f$ belongs to $\mathscr{W}$,
\item[(ii)]
a {\it monoidal category with weak equivalences} is a monoidal category with a structure of a category with weak equivalences, with the following additional property: for $f\in \Hom(X,Y)$ and $g\in\Hom(X^\prime,Y^\prime)$ in $\mathscr{W}$, their product $f\otimes g\in\Hom(X\otimes X^\prime,Y\otimes Y^\prime)$ also belongs to $\mathscr{W}$.
\end{itemize}
}
\end{defn}

\begin{defn}\label{lmon}{\rm
Let $\mathscr{M}$ be a symmetric monoidal category with weak equivalences, see Definition \ref{catwe}.
{\it A Leinster monoid} in $\mathscr{M}$ is a colax-monoidal functor $X_L\colon \Delta_0^\opp\to \mathscr{M}$ such that the colax-maps
$$
\beta_{m,n}\colon X_{m+n}\to X_m\otimes X_n
$$
and the map
$$
\alpha\colon X_0\to e_\mathscr{M}
$$
are {\it weak equivalences}. Here we use notation $X_n=X_L([n])$.
}
\end{defn}
If the contrary is not explicitly indicated, we assume that $X_0=e_\mathscr{M}$ and $\alpha\colon X_0\to e_\mathscr{M}$ is the identity isomorphism.

It is clear from the discussion above that a honest monoid $M$ in $\mathscr{M}$ defines a Leinster monoid $X(M)_L$ by
$$
X(M)_n=M^{\otimes n}
$$
with the natural action of $\Delta_0^\opp$, and with $\beta_{m,n}\colon M^{\otimes (m+n)}\to M^{\otimes m}\otimes M^{\otimes n}$ the identity map.

The category of Leinster monoids in $\mathscr{M}$ is monoidal. Indeed, for two Leinster monoids $X_L$ and $Y_L$ in $\mathscr{M}$ their product
$X_L\otimes Y_L$ is a Leinster monoid $Z_L$ with
\begin{equation}
Z_m=X_m\otimes_{\mathscr{M}}Y_m
\end{equation}
with the diagonal action of $\Delta_0^\opp$, and with
\begin{equation}
\beta_{m,n}^Z=\beta_{m,n}^X\otimes \beta_{m,n}^Y
\end{equation}
This monoidal product is symmetric.

The monoidal category of Leinster monoids in $\mathscr{M}$ is denoted by $\mathscr{L}(\mathscr{M})$.
The category $\mathscr{L}(\mathscr{M})$ is again a category with weak equivalences, defined as the component-wise weak equivalences in $\mathscr{M}$.
The latter observation makes it possible to iterate the construction.

\section{\sc The bar-cobar duality for associative dg (co)algebras}
Denote by $\Alg(\k)$ the category of all dg associative algebras over our ground field $\k$, and by $\Alg_u(\k)$ its subcategory of unital dg associative algebras.

As well, denote by $\Coalg(\k)$ the category of dg coalgebras over $\k$.
Let $C\in \Coalg(\k)$. Denote by $F_\ell(C)\subset C$ the subspace of all $x\in C$ such that $\Delta^\ell(x)=0$, where $\Delta^\ell=\Delta\circ\Delta\dots\circ \Delta$ ($\ell$ times). One has an ascending filtration
$$
0\subset F_1(C)\subset F_2(C)\subset F_3(C)\subset \dots
$$
A dg coalgebra $C$ is called {\it conilpotent} if $$
\bigcup_{\ell\ge 0}F_\ell(C)=C
$$

Denote by $\Coalg_\nilp(\k)$ the full subcategory of $\Coalg(\k)$ whose objects are conilpotent dg coalgebras.
Note that if $C\in\Coalg_\nilp(\k)$, $C$ can not have a counit.

Define the functors $\Bar\colon \Alg(\k)\to \Coalg_\nilp(\k)$ and $\Cobar\colon\Coalg_\nilp(\k)\to\Alg(\k)$ (the bar and the cobar complexes) as follows. Note that our signs in the bar and the cobar differentials are different from the standard ones, see Remark \ref{ncsigns} below for the motivation of our choice of signs.

Let $A\in\Alg(\k)$. Its bar-complex $\Bar(A)$ is the total complex of the bicomplex whose rows are $\mathbb{Z}_{\le -1}$-graded and look like (each term is a column-complex if $A$ has several non-zero graded components):
\begin{equation}
\dots\rightarrow \underset{\text{degree -3}}{A^{\otimes 3}}\xrightarrow{d_3} \underset{\text{degree -2}}{A^{\otimes 2}}\xrightarrow{d_2}\underset{\text{degree -1}}A\rightarrow 0
\end{equation}
and the differential $d_n\colon A^{\otimes n}\to A^{\otimes (n-1)}$ is
\begin{equation}\label{diffbarclass}
d_n(a_1\otimes a_2\otimes \dots \otimes a_n)=\sum_{i=1}^{n-1}(-1)^{i+\sum_{j=1}^{i}\deg a_j}a_1\otimes \dots \otimes \ a_i\cdot a_{i+1}\ \otimes \dots \otimes a_n
\end{equation}

The standard fact about bar-complex is:
\begin{lemma}\label{lemmabc1}
Let $A$ be an associative dg algebra over $\k$. Then:
\begin{itemize}
\item[(i)] $\Bar(A)\in \Coalg_\nilp(\k)$ is a (conilpotent) dg coalgebra, whose underlying coalgebra is the cofree non-counital coalgebra cogenerated by $A[1]$,
\item[(ii)] assume that the dg algebra $A$ is unital (that is, $A\in\Alg_u(\k)\subset\Alg(\k)$), then the complex $\Bar(A)$ is acyclic in all degrees.
\end{itemize}
\end{lemma}
\begin{proof}
It is fairly standard. For (ii), recall that the maps $h_n\colon A^{\otimes n}\to A^{\otimes (n+1)}$ defined as
$$
h_n(a_1\otimes\dots\otimes a_n)=1\otimes a_1\otimes\dots \otimes a_n
$$
give a homotopy of $\Bar(A)$ such that $dh\pm hd=\id$.
\end{proof}

Dually, let $C\in\Coalg(\k)$. Its cobar-complex $\Cobar(C)$ is the total complex of the bicomplex, whose rows are $\mathbb{Z}_{\ge 1}$-graded and look like:
\begin{equation}
0\rightarrow\underset{\text{degree 1}}C\xrightarrow{\delta_1}\underset{\text{degree 2}}{C^{\otimes 2}}\xrightarrow{\delta_2}\underset{\text{degree 3}}{C^{\otimes 3}}\rightarrow\dots
\end{equation}
where the differential $\delta_n\colon C^{\otimes n}\to C^{\otimes (n+1)}$ is given as
\begin{equation}\label{diffcobarclass}
\delta_n(c_1\otimes c_2\otimes\dots\otimes c_n)=\sum_{i=1}^n(-1)^{i-1+\deg c_i^\prime+\sum_{j=1}^{i-1}\deg c_j}c_1\otimes\dots\otimes c_i^\prime\otimes c_i^{\pprime}\otimes\dots\otimes c_n
\end{equation}
where $\Delta\colon C\to C^{\otimes 2}$ is the coproduct, and $\Delta(c)=c^\prime\otimes c^\pprime$ is the coproduct written down in ``Sweedler notations'' (which means that there is a sum of such terms, and in our case all $c^\prime$ and $c^\pprime$ are homogeneous, for homogenous $c$).

The standard fact on the cobar-complex, dual to Lemma \ref{lemmabc1}, is
\begin{lemma}\label{lemmabc2}
Let $C$ be a dg coalgebra, $C\in\Coalg(\k)$. Then:
\begin{itemize}
\item[(i)] the cobar-complex $\Cobar(C)\in \Alg(\k)$ is an associative non-unital dg algebra, whose underlying algebra is the free non-unital associative algebra generated by $C[-1]$,
\item[(ii)] if $C\in\Coalg_u(\k)$, the cobar-complex $\Cobar(C)$ is acyclic in all degrees.
\end{itemize}
\end{lemma}
\qed

\begin{remark}\label{ncsigns}{\rm
Note that in our definition of the bar and of the cobar complexes, the signs in the differential are slightly (but essentially) different from the standard definitions. In the standard definitions, the sign in \eqref{diffbarclass} would be $i-1+\sum_{j=1}^{i-1}\deg a_j$, and the standard sign in \eqref{diffcobarclass} is $i-1+\sum_{j=1}^{i-1}\deg c_j$. Our choice is for the sign in the bar-complex is motivated by the Eilenberg-MacLane map in Theorem \ref{emmnewtheorem}. More precisely, the claim (I) of this Theorem fails with the standard signs (more specifically, the case (iii) of the proof of (I) in Lemma \ref{lemma44proof} fails). Then the sign in the cobar-complex should also be ``adjusted''.
}
\end{remark}

The first manifestation of the Quillen bar-cobar duality is the following:
\begin{prop}\label{propbc1}
The functors of bar and cobar complexes form an adjoint pair of functors
\begin{equation}\label{adjf}
\Cobar\colon\ \Coalg_\nilp(\k)\rightleftarrows \Alg(\k)\ \colon \Bar
\end{equation}
with $\Cobar$ the left adjoint. That is, one has an isomorphism of bifunctors:
\begin{equation}
\Hom_{\Alg(\k)}(\Cobar(C),A)\simeq \Hom_{\Coalg_\nilp(\k)}(C,\Bar(A))
\end{equation}
\end{prop}
\qed

We essentially use the following result:
\begin{theorem}\label{theorembc1}
For any $A\in \Alg(\k)$ (not necessarily unital) and $C\in\Coalg_\nilp(\k)$ the adjunction morphisms of the adjoint pair \eqref{adjf}
\begin{equation}
\Phi_A\colon \Cobar(\Bar(A))\to A
\end{equation}
and
\begin{equation}
\Psi_C\colon  C\to \Bar(\Cobar(C))
\end{equation}
are quasi-isomorphisms of dg algebras and of dg coalgebras, correspondingly.
\end{theorem}
It is a consequence of the Quillen bar-coabar duality, see e.g. [LH], but can also be proven directly. Below in this paper, we will need a generalization of this result for the bar-complex of Leinster monoids, see Theorem \ref{theorem27}. 
To make our paper more reader-friendly, we provide here a complete proof of the statement for $\Phi_A$ (the case we use here), 
in the less technical setting of associative dg (co)algebras.
\begin{remarks}\label{remarkbc1}{\rm
\begin{itemize}
\item[1.] When $A\in\Alg_u(\k)\subset\Alg(\k)$ is an algebra with 1, the map $\Phi_A$ is a quasi-isomorphism despite of the acyclicity of $\Bar(A)$, see Lemma \ref{lemmabc1}(i).
\item[2.] On the other hand, $\Bar(\Cobar(C))$ is acyclic when $\Cobar(C)$ is acyclic, which is always the case when $C$ is counital.
Recall that a nilpotent coalgebra can never be counital.
\end{itemize}
}
\end{remarks}
\begin{proof}
We prove the statement for $\Phi_A$, as this is the part of Theorem we use in the paper.

The bicomplex whose total complex computes the cohomology of $\Cobar(\Bar(A))$, looks like:
\begin{equation}\label{eqbc5}
\rightarrow\underset{\text{degree -3}}{L_3}\rightarrow\underset{\text{degree -2}}{L_2}\rightarrow\underset{\text{degree -1}}{L_1}\rightarrow\underset{\text{degree 0}}{L_0}\rightarrow 0
\end{equation}
where each particular complex $L_n$ looks like:
\begin{equation}\label{eqjv}
      0\rightarrow\underset{\text{degree 0}}{X_0}(n)\rightarrow\dots \rightarrow \underset{\text{degree $n-2$}}{X_{n-2}}(n)\rightarrow \underset{\text{degree ${n-1}$}}{X_{n-1}}(n)\rightarrow \underset{\text{degree $n$}}{X_n}(n)\rightarrow 0
\end{equation}
where
\begin{equation}
X_p(n)=\bigoplus_{k_1+\dots+k_p=n+1}A^{\otimes k_1}\otimes A^{\otimes k_2}\otimes\dots \otimes A^{\otimes k_p}
\end{equation}
In particular, the rightmost component is $X_n(n)=A\otimes A\otimes\dots \otimes A$ ($n+1$ factors), and the leftmost one is $A^{\otimes (n+1)}$.

We refer to the differential in \eqref{eqjv} as ``vertical'', and to the differential in \eqref{eqbc5} as ``horizontal''.

If $A$ itself has several non-zero grading components, the grading of an element in $X_p(n)$ in the complex $L_n$ is the sum $n-p+\deg_A$, where $\deg_A$ is the total $A$-degree of the factors.

\comment
The ``horizontal'' differential in the complex in \eqref{eqbc5} is the bar-differential extended by the Leibniz rule.
Therefore, the cohomology of this differential are 0 in all degrees. However, the spectral sequence whose differential in the $E_0$-term is this bar-differential, may diverge by dimensional reasons.
\endcomment

We compute the cohomology of the total complex $\Cobar(\Bar(A))$ by the spectral sequence whose differential in $E_0$-term is the ``vertical'' differential \eqref{eqjv}. This spectral sequence converges by reasons of grading, see e.g. [W, Th.5.5.1]. 
(Note that another spectral sequence, computing firstly the cohomology of the ``horizontal'' differential \eqref{eqbc5}, in general diverges).

\begin{lemma}\label{lemmaln}
The complexes $L_n$ are acyclic for $n\ge 1$, while the complex $L_0$ is quasi-isomorphic (and isomorphic) to $A$.
\end{lemma}
\begin{proof}
The differentials in $L_n$ have two components $d_A\colon X_p(n)\to X_p(n)$ and $d_S\colon X_p(n)\to X_{p+1}(n)$ of degree +1, where $d_A$ is the differential in $A$, and $d_S$ is defined as follows.

The restriction of $d_S$ to $A^{\otimes k_1}\otimes A^{\otimes k_2}\otimes\dots\otimes A^{\otimes k_p}\subset K_p(n)$ is the sum
$\sum_{j=1}^pd_{S,j}$, and $d_{S,j}$ acts only on the factor $A^{\otimes k_j}$ (and is the identity on the remaining factors).
For any $\alpha\in \{1,\dots k_j-1\}$ denote by $d_{S,j,\alpha}(A^{\otimes k_j})\subset A^{\otimes \alpha}\otimes A^{\otimes k_j-\alpha}$ the same monomial subdivided into two pieces of consecutive elements. We have then $d_{S,j}=\sum_{\alpha=1}^{k_j-1}d_{S,j,\alpha}$. Finally, we set
\begin{equation}
d_S|_{K_p(n)}=\sum_{j=1}^p\pm d_{S,j}
\end{equation}
where the signs $\pm$ are found immediately from those in the bar and cobar differentials.

To prove the claim, note that $\oplus_{n\ge 0}L_n$ with the differential $d_S$ is just the cobar-differential of the cofree coalgebra $T^+(A[1])$ without counit cogenerated by the underlying graded space of $A$. The cohomology of this cobar-complex is known to be equal to the Koszul dual to $T_+(A[1])$ algebra, which is $A$.
\end{proof}
We turn back to the spectral sequence of the bicomplex $\Cobar(\Bar(A))$, with the differential in $E_0$ equal to the one in the complexes $L_n$ (the ``vertical'' one). By Lemma, the first term is
$$
E_1^{ij}=H^j(A) \text{ for $i=0$ and 0 otherwise}
$$

Then the cohomology of the bicomplex is isomorphic to that of $A$.

The map $\Phi_A\colon \Cobar(\Bar(A))\to A$ can be explicitly described as follows. Its restriction to the underlying graded spaces $L_n$ is non-zero only
the rightmost component $A\otimes A\otimes \dots\otimes A$, where it is $a_1\otimes a_2\otimes\dots \otimes a_{n+1}\mapsto a_1a_2\dots a_{n+1}$.
This map is not a map of bicomplexes, so we can not apply our spectral sequence argument immediately.

The map $\Phi_A$ is a map of algebras, and we only needs to prove that it is a quasi-isomorphism of complexes. To this end, consider another map
$\imath\colon A\to \Cobar(\Bar(A))$ which maps $A$ by identity to $L_0$. It is not a map of algebras, but it is a map of bicomplexes, and our spectral sequence argument shows that $\imath$ is a quasi-isomorphism of the total complexes.

On the other hand, the composition $\Phi_A\circ \imath=\id_A$, therefore $\Phi_A$ is also a quasi-isomorphism.
\end{proof}
An important property of the bar-complex is that it preserves the quasi-isomorphisms.
\begin{prop}\label{barquis}
Let $f\colon A\to B$ be a quasi-isomorphism of dg algebras (or, more generally, an $A_\infty$-quasi-isomorphism).
Then the corresponding map $\Bar(f)\colon \Bar(A)\to \Bar(B)$ is a quasi-isomorphism of dg coalgebras.
\end{prop}
\begin{proof}
We give a proof which works in the $A_\infty$ case as well.
At first, the statement is equivalent to the acyclicity of $\Cone(\Bar(f))$ (for the $A_\infty$ case it is an abuse of notations to denote the morphism of the bar-complexes by $\Bar(f)$). When $f$ is a map of dg algebras, $\Cone(\Bar(f))$ is a bicomplex. We can use the spectral sequence computing the cohomology of the (total) complexes $L_k=\Cone(A^{\otimes k}\to B^{\otimes k})$ at first. This spectral sequence converges, and collapses at the term $E_1$, as by the assumption (that $f$ is a quasi-isomorphism) cohomology of the complexes $L_k$ vanish in all degrees.

Consider now the case when $f$ is an $A_\infty$ map. Consider the ascending filtration of $\Cone(\Bar(f))$, setting
$$
F_k=\Cone(\oplus_{i\le k}A^{\otimes i}\to\oplus_{i\le k}B^{\otimes i})
$$
Then $\{F_k\}_{k\ge 1}$ is a filtration of $\Cone(\Bar(f))$ by {\it subcomplexes}.
Consider the spectral sequence corresponding to this filtration. It converges by dimensional reasons, and its term $E_1$ depends only on the linear component of the $A_\infty$ map $f$. That is the cohomology in the term $E_1$ vanish.
\end{proof}

\section{\sc Graded Leinster monoids}
In this Section, we introduce the concepts of a {\it graded Leinster monoid} in $\Vect(\k)$ and in $\Alg(\k)$.
The concept of a graded Leinster monoid in $\Vect(\k)$ provides a weaker and relaxed context for the concept of a dg associative algebra over $\k$. We define a {\it bar-complex} $\Bar(X_L)$ of a graded Leinster monoid $X_L$, and show that it is a dg coalgebra. We prove that $\Cobar(\Bar(X_L))$ is a dg algebra whose underlying dg vector space is quasi-isomorphic to $X_1$, the first component of $X_L$, in Theorem \ref{theorem27}. 
\subsection{\sc }
Here we define the concepts of a graded Leinster monoid in $\Vect(\k)$ and of a graded Leinster monoid in $\Alg(\k)$.
In Section \ref{sectionapp1}, we define graded Leinster monoids in $\Cat^\dg(\k)$.

Let $X_L\colon\Delta_0^\opp\to\Vect(\k)$ be a Leinster monoid, with the colax maps $\beta_{m,n}\colon X_{m+n}\to X_m\otimes X_n$, see Definition \ref{lmon}. Recall that $\beta_{m,n}$ are quasi-isomorphisms of complexes (the weak equivalences in $\Vect(\k)$). We denote by $X_\ell$ the components $X_L([\ell])$.

\begin{defn}\label{grl}{\rm
A Leinster monoid $X_L$ in $\Vect(\k)$ is called {\it graded} if each $X_\ell$ is a $\mathbb{Z}^\ell$-graded vector space, 
with the graded components $X_\ell^{a_1,\dots,a_\ell}$, such that:
\begin{itemize}
\item[(i)] for any $\ell$, the colax-map $\beta_{a,b}\colon X_\ell\to X_a\otimes X_b$, $a+b=\ell$, is $\mathbb{Z}^\ell$-graded, where the right-hand side is endowed with natural $\mathbb{Z}^a\oplus\mathbb{Z}^b=\mathbb{Z}^\ell$-grading, obtained by concatenation of the gradings on the factors;
\item[(ii)] the total grading on $X_\ell^{a_1,\dots,a_\ell}$, defined as $a_1+\dots+a_\ell$, coincides with the homological grading on the complex $X_\ell$.
\end{itemize}
}
\end{defn}
We have:
\begin{equation}
X_\ell^n=\oplus_{a_1+\dots+a_\ell=n}X_\ell^{a_1,\dots,a_\ell}
\end{equation}

The category of graded Leinster monoids in $\Vect(\k)$ is denoted by $\mathscr{L}^\gr(\Vect(\k))$.

\begin{example}{\rm
Our basic example of a graded Leinster monoid in $\Vect(\k)$ is the following. Let $A$ be an associative dg algebra.
The formula $X_n=A^{\otimes n}$ gives a Leinster monoid in $\Vect(\k)$, with $\beta_{a,b}\colon A^{\otimes (a+b)}\to A^{\otimes a}\otimes A^{\otimes b}$ the identity maps. Let 
$$
A=\oplus_{a\in\mathbb{Z}}A^a
$$
be the underlying graded vector space. Define the grading on $X_L$ by setting
$$
X_\ell^{a_1,a_2,\dots,a_\ell}=A^{a_1}\otimes A^{a_2}\otimes\dots\otimes A^{a_\ell}
$$
One easily checks that this definition makes $X_L$ a graded Leinster monoid in $\Vect(\k)$.
}
\end{example}

For Leinster monoids in $\Alg(\k)$, we need one extra condition:
\begin{defn}\label{grla}{\rm
A Leinster monoid $X_L$ in $\Alg(\k)$ is called {\it graded} the underlying Leinster monoid in $\Vect(\k)$ is graded, and for each $\ell$, the product in $X_\ell$ agrees with $\mathbb{Z}^\ell$-grading:
\begin{equation}
X_\ell^{a_1,\dots,a_\ell}*X_\ell^{b_1,\dots,b_\ell}\subset X_\ell^{a_1+b_1,\dots,a_\ell+b_\ell}
\end{equation}
}
\end{defn}
The category of graded Leinster monoids in $\Alg(\k)$ is denoted by $\mathscr{L}^\gr(\Alg(\k))$.

\subsection{\sc The symmetric monoidal category of graded Leinster monoids}
Here we endow the category $\mathscr{L}^\gr(\Vect(\k))$  with a symmetric monoidal structure.
The construction is straightforward, but one should be careful on signs.

Let $X_L$ and $Y_L$ be two graded Leinster monoids in $\Vect(\k)$. Define the multi-grading on $X_n\otimes Y_n$ be setting
\begin{equation}\label{grmon1}
(X_n\otimes Y_n)^{c_1,\dots,c_n}=\bigoplus_{\substack{{a_1,\dots,a_n}\\{b_1,\dots,b_n}\\{a_i+b_i=c_i}}}
X_n^{a_1,\dots,a_n}\otimes Y_n^{b_1,\dots,b_n}
\end{equation}

Endow the collection of spaces $\{Z_n=X_n\otimes Y_n\}$, ${n\ge 0}$, with a structure of a functor $Z_L\colon \Delta_0^\opp\to \Vect(\k)$. 

For $x\in X_n^{a_1,\dots,a_n}, y\in Y_n^{b_1,\dots,b_n}$, we set
\begin{equation}\label{grmon2}
F_i(x\otimes y)=F_i(x)\otimes F_i(y)
\end{equation}
and
\begin{equation}\label{grmon3}
D_j(x\otimes y)=D_j(x)\otimes D_j(y)
\end{equation}
It follows that \eqref{grmon2} and \eqref{grmon3} define a structure of  functor 
$Z_L\colon \Delta^\opp_0\to \Vect(\k)$ with $Z_n=X_n\otimes Y_n$.

Finally, define the colax-monoidal structure $\beta(Z)$ on the functor $Z_L\colon \Delta_0^\opp\to \Vect(\k)$, by
\begin{equation}\label{grmon4}
\begin{aligned}
\ &\beta(Z)(x_n\otimes y_n)=\\
&\sum_{\substack{{a_1^\prime,\dots,a_n^\prime}\\{b_1^\prime,b_n^\prime}}}(-1)^{(a_1^\pprime+\dots+a_n^\pprime)(b_1^\prime+\dots+b_n^\prime)}
\beta(X)^{(1)}_{a_1^\prime,\dots,a_n^\prime}(x_n)\cdot \beta(Y)^{(1)}_{b_1^\prime,\dots,b_n^\prime}(y_n)\ \otimes\ \beta(X)^{(2)}_{a_1^\pprime,\dots,a_n^\pprime}(x_n)\cdot \beta(Y)^{(2)}_{b_1^\pprime,\dots,b_n^\pprime}(y_n)
\end{aligned}
\end{equation}
where
$$
\beta(X)(x_n)=\sum_{\substack{a_i^\prime+a_i^{\pprime}=a_i}}\beta(X)^{(1)}_{a_1^\prime,\dots,a_n^{\prime}}(x_n)\otimes
\beta(X)^{(2)}_{a_1^\pprime,\dots,a_n^{\pprime}}(x_n)
$$
and
$$
\beta(Y)(y_n)=\sum_{\substack{b_i^\prime+b_i^{\pprime}=b_i}}\beta(Y)^{(1)}_{b_1^\prime,\dots,b_n^{\prime}}(y_n)\otimes
\beta(Y)^{(2)}_{b_1^\pprime,\dots,b_n^{\pprime}}(y_n)
$$
It is checked that $\beta(Z)$ indeed defines a colax-monoidal structure.

In this way, $\mathscr{L}^\dg(\Vect(\k))$ is a symmetric monoidal category.

\subsection{\sc The bar-complex of a graded Leinster monoid in $\Vect(\k)$}\label{sectiongraded}
For a graded Leinster monoid $X_L$ in $\Vect(\k)$, we define a dg coalgebra $\Bar(X_L)\in\Coalg_\nilp(\k)$ which we call {\it the bar-complex of $X_L$}. When the Leinster monoid $X_L$ is defined by an associative dg algebra $A\in\Alg(\k)$, by setting $X_\ell=A^{\otimes \ell}$, the dg coalgebra $\Bar(X_L)$ coincides with $\Bar(A)\in\Coalg_\nilp(\k)$.

\begin{remark}{\rm
Consider the bar-complex of a dg associative algebra $A$. The summands in the formula \eqref{diffbarclass} for the bar-differential contain the sign corrections $(-1)^{i-1+\sum_{j=1}^{i-1}\deg a_j}$ depending on degrees of $a_i$'s. It makes impossible
to consider the bar-differential just as the conventional chain differential $F_1-F_2+F_3-\dots +(-1)^nF_{n-1}$ of the Leinster monoid with $X_\ell=A^{\otimes \ell}$; the ``sign corrections'' should be inserted in the chain differential. Our concept of the bar-complex of a {\it graded} Leinster monoid $X_L$ provides a more general setting for the ``bar-differential with the correct signs''. The grading on $X_L$ is used to determine these signs.
}
\end{remark}

The construction is as follows. The underlying complex $\Bar(X_L)$ is the total complex of the bicomplex
\begin{equation}\label{diffbars}
\rightarrow\underset{\text{degree -3}}{X_3}\xrightarrow{d_3^s}\underset{\text{degree -2}}{X_2}\xrightarrow{d_2^s}\underset{\text{degree -1}}{X_1}\rightarrow 0
\end{equation}
whose horizontal differential, the
``signed'' chain differential $d^s$, is defined as follows.

The component of the differential $d_\ell^s\colon X_\ell^{a_1,\dots,a_\ell}\to X_{\ell-1}$ is the alternated sum of the face maps in $\Delta_0$:
\begin{equation}\label{diffbarlb}
d_\ell^s=(-1)^{A_1}F_1+(-1)^{A_2}F_2+\dots+(-1)^{A_{\ell-1}}F_{\ell-1}
\end{equation}
where $F_i$ is corresponded to the face map $f_i\colon [\ell-1]=\{0<\dots <\ell-1\}\to [\ell]=\{0<\dots\ell\}$, see Section \ref{deltanot}, and
\begin{equation}\label{diffbarbis}
A_i=a_1+\dots+a_{i}+i
\end{equation}

\begin{lemma}\label{lemmaqz}
Let $X_L$ be a graded Leinster monoid in $\Vect(\k)$. Then the following statements are true:
\begin{itemize}
\item[(i)] $F_i\colon X_n\to X_{n-1}$ maps $X_n^{a_1,\dots,a_n}$ to $X_{n-1}^{a_1,\dots,a_{i-1},a_i+a_{i+1},a_{i+2},\dots,a_n}$,
$1\le i\le n-1$;
\item[(ii)] $D_j\colon X_n\to X_{n+1}$ maps $X_n^{a_1,\dots,a_n}$ to $X_{n+1}^{a_1,\dots,a_{j-1},0,a_j,\dots,a_n}$, $0\le j\le n$;
\item[(iii)] $(d^s)^2=0$.
\end{itemize}
\end{lemma}
\begin{proof}
(i): For $n=2$ it follows from the part (ii) of the Definition \ref{grl}, because all morphisms from $\Delta_0^\opp$ preserve the cohomological grading. For general $n$ and some $1\le i\le n-1$, we represent $[n]$ as $[i-1]\otimes [2]\otimes [n-i-1]$, with respect to the monoidal structure in $\Delta_0^\opp$, and $F_i=\id_{[i-1]}\otimes (F_1)_{[2]}\otimes \id_{[n-i-1]}$.
Now we use the part (i) of the Definition \ref{grl}, which implies that for a homogeneous $x\in X_n^{a_1,\dots,a_n}$, 
the multi-degree of an element $x$ is the same that of the element $\beta_{a,b}(x)$.
As the colax-monoidal structure $\beta$ is natural with respect to the morphisms of $\Delta_0^\opp$, one has:
\begin{equation}
\begin{aligned}
\ &\beta_{i-1,2,n-i-1}(F_i(x))=\beta_{i-1,2,n-i-1}(\id_{[i-1]}\otimes (F_1)_{[2]}\otimes \id_{[n-i-1]}(x))=\\
&(\id_{[i-1]}\otimes 
(F_1)\otimes\id_{[n-i-1]})(\beta_{i-1,2,n-i-1}(x))
\end{aligned}
\end{equation}
what gives the result.

(ii): is proven similarly, taking to the account that by our assumption $X_0=\k$, and  has the $\mathbb{Z}^0$-grading 0.

(iii): it follows from (i), in a standard way.
\end{proof}

\hspace{1mm}

As the next step, we endow the bar-complex $\Bar(X_L)$ with a structure of a dg coalgebra.

The restriction of the coproduct to $X_\ell$ is given as the direct sum:
\begin{equation}\label{barcoalg}
X_\ell\xrightarrow{\oplus \beta_{a,b}}\bigoplus_{\substack{{a+b=\ell}\\{a,b>0}}}X_a\otimes X_b\subset \Bar(X_\ell)\otimes\Bar(X_\ell)
\end{equation}
\begin{lemma}\label{lemmabarcoalg}
The coproduct \eqref{barcoalg} is coassociative, and is compatible with the differential, which makes $\Bar(X_L)$, for a graded Leinster monoid $X_L$, a co-nilpotent dg coalgebra. 
\end{lemma}
\begin{proof}
The coassociativity follows directly from the colax-monoidal property for $\beta_{m,n}$, see diagram \eqref{colax1}.
The co-nilpotency is clear. We need to check that the differential $d^s$ and the coproduct are compatible, by
\begin{equation}
d^s(\Delta(x))=\sum_id^s(y_i)\otimes z_i+(-1)^{\deg y_i}y_i\otimes d^s(z_i)
\end{equation}
with
$$
\Delta(x)=\sum_iy_i\otimes z_i
$$
and homogeneous $x,y_i,z_i$. We provide a detailed proof, due to a possible confusion explained in Remark \ref{remconf} below.

Let $x\in X_n^{t_1,\dots,t_n}$, then
\begin{equation}\label{eqz1}
d^s(x)=\sum_{j=1}^{n-1}(-1)^{j+t_1+\dots+t_{j}}F_jx
\end{equation}
Note that 
\begin{equation}\label{eqz2}
F_jx\in X_{n-1}^{t_1,\dots,t_{j-1},t_j+t_{j+1},t_{j+2},\dots,t_n}
\end{equation}
by Lemma \ref{lemmaqz}(i).

On the other hand,
\begin{equation}\label{eqz3}
\Delta(x)=\sum_{a+b=n}x_{t_1,\dots,t_a}\otimes x_{t_{a+1},\dots,t_n}
\end{equation}
with
\begin{equation}\label{eqz4}
x_{t_1,\dots,t_a}\in X_a^{t_1,\dots,t_a},\ \ \ x_{t_{a+1},\dots,t_n}\in X_b^{t_{a+1},\dots,t_n}
\end{equation}
by condition (i) of Definition \ref{grl}
(we use so-called ``Sweedler notations'' assuming another summation $\sum_i x_{t_1,\dots,t_a}^i\otimes x_{t_{a+1},\dots,t_n}^i$, but omitting it, to make the formulas shorter and more visible).

Then we have, for $1\le j\le n-1$:
\begin{equation}\label{eqz5}
\Delta(F_j x)=\sum_{\substack{a+b=n\\j\le a-1}}F_jx_{t_1,\dots,t_a}\otimes x_{t_{a+1},\dots,t_n}+\sum_{\substack{{a+b=n}\\j\le a+1}} x_{t_1,\dots,t_a}\otimes F_jx_{t_{a+1},\dots,t_n}  
\end{equation}
because the colax-map $\beta$, used in the definition of the coproduct, is a morphism of bifunctors. 

Then
\begin{equation}\label{eqz6}
\begin{aligned}
\ &\Delta(d^s(x))=\Delta(\sum_{j=1}^{n-1}(-1)^{j+t_1+\dots+t_{j}}F_jx)=\\
&\sum_{j=1}^{n-1}(-1)^{j+t_1+\dots+t_{j}}\Bigl(\sum_{\substack{a+b=n\\j\le a-1}}F_jx_{t_1,\dots,t_a}\otimes x_{t_{a+1},\dots,t_n}+\sum_{\substack{{a+b=n}\\{j\ge a+1}}} x_{t_1,\dots,t_a}\otimes F_jx_{t_{a+1},\dots,t_n}
\Bigr)=\\
&\sum_{a+b=n}\sum_{j=1}^{a-1}(-1)^{j+t_1+\dots+t_{j}}F_jx_{t_1,\dots,t_a}\otimes x_{t_{a+1},\dots,t_n}+\\
&\sum_{a+b=n}(-1)^{a+t_1+\dots+t_a}
\sum_{k=1}^{b-1}(-1)^{k+t_{a+1}+\dots+t_{a+k}} x_{t_1,\dots,t_a}\otimes   F_kx_{t_{a+1},\dots,t_n}
\end{aligned}
\end{equation}
On the other hand,
\begin{equation}\label{eqz7}
\begin{aligned}
\ &d^s(\Delta(x))=d^s(\sum_{a+b=n}x_{t_1,\dots,t_a}\otimes x_{t_{a+1},\dots,t_n})=\\
&\sum_{a+b=n}\Bigl(d^s(x_{t_1,\dots,t_a})\otimes x_{t_{a+1},\dots,t_n}+(-1)^{\deg x_{t_1,\dots,t_a}}x_{t_1,\dots,t_a}\otimes 
d^s(x_{t_{a+1},\dots,t_n})\Bigr)
\end{aligned}
\end{equation}
Here $\deg x_{t_1,\dots,t_a}$ is the total degree in $\Bar(X_L)$, as $x_{t_1,\dots,t_a}\in X_a^{t_a,\dots,t_a}$, it is equal
to $-a+t_1+\dots+t_a$, by condition (ii) in Definition \ref{grl}. We see that the r.h.s. of \eqref{eqz7} is equal to the r.h.s. of \eqref{eqz6}. We are done.
\end{proof}
\begin{remark}\label{remconf}{\rm
Endowed with the ordinary chain differential $d_n=\sum_{i=1}^{n-1}(-1)^{i+1}F_i$, and with the coproduct \eqref{barcoalg}, the chain complex of a Leinster monoid $X_L$ fails to be a dg coalgebra. Indeed, for $x\in X_\ell$, consider the $(a,b)$-component $\Delta_{a,b}(x)=\sum_ix_i^a\otimes x_i^b\in X_a\otimes X_b$, $a+b=\ell$. Then $$d(\Delta_{a,b}(x))=
\sum_i(dx_i^a)\otimes x_i^b+\sum_i(-1)^{\deg x_i^a-a}x_i^a\otimes (dx_i^b)$$
where $\deg x_i^a-a$ is the total degree of $x_i^a$ in $\Bar(X_L)$. On the other hand, the ordinary chain differential does not depend on the cohomological degree of elements in $X_j$. Thus $d(\Delta(x))$ (depending on the cohomological grading in $X_j$'s) has no chances to be equal to $\Delta(dx)$ (which does not depend on the cohomological grading in $X_j$'s),
except for the case when all $X_j$ are complexes concentrated in cohomological degree 0.
}
\end{remark}

\subsection{\sc The bar-cobar duality for graded Leinster monoids}
By so far, we have not made use the assumption that all $\beta_{m,n}$ are quasi-isomorphisms. It is essential in the proof of the following result, which we call, by abuse of terminology, the ``bar-cobar duality for graded Leinster monoids''.
\begin{theorem}\label{theorem27}
Suppose $X_L\colon \Delta_0^\opp\to\Vect(\k)$ is a graded Leinster monoid (in particular, all $\beta_{m,n}\colon X_{m+n}\to X_m\otimes X_n$ are quasi-isomorphisms).
Then the underlying complex of the dg algebra $\Cobar(\Bar(X_L))$ is quasi-isomorphic to the complex $X_1$.
\end{theorem}
\begin{proof}
Consider the spectral sequence of the bicomplex  $\Cobar(\Bar(X_L))$ computing the cohomology of the complexes (analogous to the complexes) $L_i$ in \eqref{eqbc5} at first (the ``vertical'' differential \eqref{eqjv} in the bicomplex). This spectral sequence converges by dimensional reasons. Moreover, the inclusion
\begin{equation}
i\colon X_1\to L_0\subset \Cobar(\Bar(X_L))
\end{equation}
is a map of bicomplexes, and therefore induces a map of the corresponding spectral sequences.

So it is enough to prove that $i$ induces an isomorphism in the term $E_1$. For this end, we have
\begin{lemma}\label{lemmalnbis}
For a Leinster monoid $X_L$ in $\Vect(\k)$, the complexes $L_n$ given in \eqref{eqbc5} are acyclic for $n\ge 1$, whence the complex $L_0$ is isomorphism to $X_1$ (and thus is quasi-isomorphic to it).
\end{lemma}
\begin{proof}
The difference with the case of the bar-complex of a dg algebra, dealt with in Lemma \ref{lemmaln}, is that the maps $\beta_{mn}\colon X_{m+n}\to X_m\otimes X_n$ are not isomorphisms (as in the latter case) but only are quasi-isomorphisms. As we show now, it does not change much in the proof.

Each complex $L_n$ is a bicomplex by its own, where one differential $d_1$ comes from  the colax-maps $\beta_{ab}$s, and another differential $d_2$ comes from the underlying differentials
in $X_i$s. Consider the spectral sequence of the bicomplex, computing the cohomology of $d_2$ at first.
This spectral sequence converges by the reasons of grading, see [W, Th.5.5.1].
In its $E_1$ term, we the complex $L_n(H^\udot(X_L))$.
The Leinster monoid $H^\udot(X_L)$ has components $H^\udot(X_n)=(H^\udot(X_L))^{\otimes n}$, as $\beta_{ab}$s are {\it isomorphisms} on the cohomology level.
Thus, $H^\udot(X_L)$ is isomorphic to the Leinster monoid of a strict dg algebra $H^\udot(X_1)$ with 0 differential.
Then the acyclicity of $L_n(H^\udot(X_L))$ for $n\ge 1$ is be proven by the same argument as in Lemma \ref{lemmaln}, what gives the result.
\end{proof}

\end{proof}

\subsection{\sc The operation $\sharp$}\label{sectionsharp}
Let $X_L$ be a graded Leinster monoid in $\Vect(\k)$. Here we define a new functor $X_\ldot^\sharp\colon \Delta_0^\opp\to\Vect(\k)$, which differs from the underlying functor $X_\ldot\colon \Delta_0^\opp\to\Vect(\k)$ by signs (and coincides with $X_L$ on the objects).
Note that the functor $X^\sharp_\ldot$ can {\it not} be extended to a Leinster monoid structure, that is, this functor is not colax-monoidal.

The aim is to define $X^\sharp_\ldot$ such that the conventional chain differential $d_n^\sharp\colon X^\sharp_n\to X^\sharp_{n-1}$,
\begin{equation}\label{dsharp}
d_n^\sharp=-F_1+F_2+\dots+(-1)^{n-1}F_{n-1}
\end{equation}
coincides with the ``signed'' chain differential \eqref{diffbarlb} for $X_\ldot$.

The components $X_n^\sharp=X_n$. The operators $F_i^\sharp$ and $D_i^\sharp$ are obtained by the corresponding operators $F_i$ and $D_i$, by multiplication with $\pm 1$.

Let $x\in X_n^{a_1,\dots,a_n}$ be a homogeneous element, with respect to the $\mathbb{Z}^n$-grading on $X_n$.
Define $F_i^\sharp\colon X_n\to X_{n-1}$, $1\le i\le n-1$, by
\begin{equation}\label{newf1}
F_i^\sharp(x)=(-1)^{a_1+\dots+a_{i}}F_i(x)
\end{equation}
As well, define $D_j^\sharp\colon X_n\to X_{n+1}$, $0\le j\le n$, by
\begin{equation}\label{newf2}
D_j^\sharp(x)=(-1)^{a_1+\dots +a_j}D_j(x)
\end{equation}
\begin{lemma}
Let $X_L$ be a graded Leinster monoid in $\Vect(\k)$, see Definition \ref{grl}.
\begin{itemize}
\item[(i)] The formulas \eqref{newf1}, \eqref{newf2} define a functor $X^\sharp_\ldot\colon \Delta_0^\opp\to\Vect(\k)$,
\item[(ii)] the conventional chain differential $d^\sharp$ \eqref{dsharp} of $X_\ldot^\sharp$ coincides with the ``signed'' chain differential 
$d^s$ \eqref{diffbarlb} of $X_L$.
\end{itemize}
\end{lemma}
\begin{proof}
(ii) follows directly from (i). For (i), we need to check the standard simplicial identities, see e.g. [W, Section 8.1], for $X_\ldot^\sharp$, except for the extreme ones (that is, except those which contain non-defined $F_0,F_n\colon X_n\to X_{n-1}$). It is a direct check. 
\end{proof}

\section{\sc The Eilenberg-MacLane map for graded Leinster monoids}
Our goal in this Section is to construct a lax-monoidal structure on the functor 
$$
\Bar\colon \mathscr{L}^\gr(\Vect(\k))\to \Vect(\k),\ \ X_L\mapsto\Bar(X_L)
$$
assigning to a {\it graded} Leinster monoid in $\Vect(\k)$ (Definition \ref{grl}) its bar-complex (see \eqref{diffbars}, \eqref{diffbarlb}), which mimics in our situation the classical Eilenberg-MacLane (shuufle) map. We recall this classical counter-part in Section \ref{emremainder}, and construct its  modified version for graded Leinster monoids in Section \ref{sectionemmnew}. 
Then we prove a crucial for our paper result that the modified Eilenberg-MacLane map 
$$
\nabla^\Bar\colon \Bar(X_L)\otimes \Bar(Y_L)\to\Bar(X_L\otimes Y_L)
$$
(for graded Leinster monoids $X_L$ and $Y_L$ in $\Vect(\k)$) is a map of {\it dg coalgebras}, in Theorem \ref{theoremcoalg} (recall that the bar-complex of a graded Leinster monoid is a dg coalgebra, by Lemma \ref{lemmabarcoalg}). Finally, we prove that $\Bar(X_L)$, for a graded Leinster monoid in $\Alg(\k)$ (see Definition \ref{grla}), is a dg bialgebra, in Corollary \ref{corcoalg}.

\subsection{\sc Reminder on the Eilenberg-MacLane map}\label{emremainder}
Let $\mathscr{A}$ be an abelian category. We denote by $\mathscr{A}^{\simpl}$ the category of simplicial abelian groups. For $X_\ldot\in \mathscr{A}^\simpl$ there is the {\it chain complex}
$C(X_\ldot)$ and the {\it normalized chain complex} $N(X_\ldot)$, both are objects of the category $\mathscr{C}^{\le 0}(\mathscr{A})$ of $\mathbb{Z}_{\le 0}$-graded complexes in $\mathscr{A}$. See [W, Section 8.4] for more detail.

There is an equivalence of categories, for any $\mathscr{A}$:
\begin{equation}
N\colon \mathscr{A}^\simpl \rightleftarrows \mathscr{C}^{\le 0}(\mathscr{A})\colon \Gamma
\end{equation}
called the {\it Dold-Kan correspondence}. See e.g. [W, Th.8.4.1].

Below we consider the cases $\mathscr{A}=\Mod(\mathbb{Z})$ and $\mathscr{A}=\Vect(\k)$, $\k$ a field.
We use notation $\mathscr{A}_*$ assuming either of these two cases.
Both categories $\mathscr{A}_*^\simpl$ and $\mathscr{C}^{\le 0}(\mathscr{A}_*)$ are {\it symmetric monoidal}.

The non-normalized chain complex $C(X_\ldot)$ is the complex
\begin{equation}\label{chain1}
\dots\rightarrow \underset{\text{degree -3}}{X_3}\rightarrow\underset{\text{degree -2}}{X_2}\rightarrow\underset{\text{degree -1}}{X_1}\rightarrow\underset{\text{degree 0}}{X_0}\rightarrow 0
\end{equation}
with the differential $d_n\colon X_n\to X_{n-1}$ equal to
\begin{equation}\label{chain2}
d_n=F_0-F_1+F_2-\dots +(-1)^nF_n
\end{equation}

The Eileberg-MacLane map (a.k.a. the Eilenberg-Zilber map) endows both functors $C$ and $N$ with a lax monoidal structure, whose lax-maps are symmetric and are quasi-isomorphisms of complexes. Below we consider only the case of the functor $C$ (of the non-normalized chain complex):
\begin{equation}\label{emmap}
\nabla\colon C(X_\ldot)\otimes C(Y_\ldot)\to C(X_\ldot\otimes Y_\ldot)
\end{equation}
It is defined as
\begin{equation}\label{emform}
\nabla(x_m\otimes y_n)=\sum_{\substack{(m,n)\text{-shuffles}\\ \sigma\in\Sigma_{m+n}}}(-1)^{\sharp(\sigma)}
\Bigl(D_{\sigma({m+n-1})}\circ\dots\circ D_{\sigma(m)}(x_m)\Bigr)\otimes \Bigl(D_{\sigma(m-1)}\circ\dots\circ D_{\sigma(0)}(y_n)\Bigr)
\end{equation}
for $x_m\in X_m, y_n\in Y_n$. Note that the expressions in big parentheses in the r.h.s.
belong to $X_{m+n}$ and to $Y_{m+n}$, correspondingly.
See [EM, Section 5] for more detail.

The main properties of the Eilenberg-MacLane map are summarized in the following Proposition.
\begin{prop}\label{emprop}
The following statements are true:
\begin{itemize}
\item[(i)]
The Eileberg-MacLane map defines a map of chain complexes
$$
\nabla\colon C(X_\ldot)\otimes C(Y_\ldot)\to C(X_\ldot\otimes Y_\ldot)
$$
that is, for $x\in X_m$ and $y\in Y_n$ one has:
\begin{equation}
d\nabla(x,y)=\nabla(dx,y)+(-1)^m\nabla(x,dy)
\end{equation}

\item[(ii)] the Eilenberg-MacLane map is lax monoidal, that is for any three $X_\ldot,Y_\ldot,Z_\ldot\in\mathscr{A}_*^\simpl$ the diagrams \eqref{lax1} and \eqref{lax2} commute,

\item[(iii)] the lax-monoidal structure $\nabla$ is symmetric, that is, for any $X_\ldot,Y_\ldot$ the diagram below commutes:
\begin{equation}
\xymatrix{
C(X_\ldot)\otimes C(Y_\ldot)\ar[d]_{\kappa_{\mathscr{C}^{\le 0}(\mathscr{A}_*)}}\ar[rr]^{\nabla}&&C(X_\ldot\otimes Y_\ldot)\ar[d]^{\kappa_{\mathscr{A}_*^\simpl}}\\
C(Y_\ldot)\otimes C(X_\ldot)\ar[rr]^{\nabla}&&C(Y_\ldot\otimes X_\ldot)
}
\end{equation}
where $\kappa_{\mathscr{C}^{\le 0}(\mathscr{A}_*)}$ and $\kappa_{\mathscr{A}_*^\simpl}$ are the symmetric braidings (satisfying $\kappa^2=\id$) in the corresponding symmetric monoidal categories,
\item[(iv)] the map $\nabla$ is a quasi-isomorphism of complexes for any $X_\ldot,Y_\ldot$.
\end{itemize}
\end{prop}

\subsection{\sc The Eilenberg-MacLane map for graded Leinster monoids}\label{sectionemmnew}
Let $X_L$  and $Y_L$ be {\it graded} Leinster monoids in $\Vect(\k)$, see Definition \ref{grl}. We are going to construct a map
\begin{equation}
\nabla^\Bar\colon \Bar(X_L)\otimes\Bar(Y_L)\to\Bar(X_L\otimes Y_L)
\end{equation}
bi-functorial in $X_L$ and $Y_L$, and which enjoys the analogues of the properties listed in the case of chain complexes in Proposition \eqref{emprop}.

Note that the statement that $\nabla^\Bar$ is a map of complexes means the following:
for $x_m\in X_m^{a_1,\dots,a_m}$ and $y_n\in Y_n^{b_1,\dots,b_n}$ one has:
\begin{equation}\label{emmnew1}
d^s_{X_L\otimes Y_L}(\nabla^\Bar(x_m\otimes y_n))=\nabla^\Bar(d^s_{X_L}(x_m)\otimes y_n)+(-1)^{-m+a_1+\dots+a_m}\nabla^\Bar(x_m\otimes d^s_{Y_L}(y_n))
\end{equation}
where $-m+a_1+\dots+a_m$ is the cohomological degree of $x_m$ considered as an element in $\Bar(X_L)$. Here $d^s$ is the differential in the bar-complex of a graded Leinster monoid, see \eqref{diffbarlb}, \eqref{diffbarbis}.

For $x_m\in X_m^{a_1,\dots,a_m}, y_n\in Y_n^{b_1,\dots,b_n}$ as above, we define

\begin{equation}\label{emmnew2}
\nabla^\Bar(x_m\otimes y_n)=\sum_{\substack{{(m,n)\text{-shuffles}}\\
{\sigma\in\Sigma_{m+n}}}}(-1)^{S_\EM(\sigma,x_m,y_n)+S_\sharp(\sigma,x_m,y_n)}D_{\sigma(m+n-1)}\circ\dots\circ D_{\sigma(m)}(x_m)\ \otimes\ 
D_{\sigma(m-1)}\circ\dots\circ D_{\sigma(0)}(y_n)
\end{equation}
where
\begin{equation}\label{emmnew3}
S_\EM(\sigma,x_m,y_n)=\sum_{\substack{{0\le i\le m-1}\\ {m\le j\le m+n-1}\\{\sigma(i)>\sigma(j)}}}(a_{i+1}+1)(b_{j-m+1}+1)
\end{equation}
and
\begin{equation}\label{emmnewsharpsign}
S_\sharp(\sigma,x_m,y_n)=\sum_{\substack{{0\le i\le m-1}\\{m\le j\le m+n-1}\\{\sigma(i)>\sigma(j)}}}a_{i+1}b_{j-m+1}
\end{equation}

\begin{remark}\label{emremnew}{\rm
1. Consider the case when all $a_i$ and $b_j$ are equal to 0, the sign $(-1)^{S_\EM(\sigma,x_m,y_n)}$, see \eqref{emmnew3}, is the conventional sign of the transposition $\sigma$. 2. A Leinster monoid $X_L$ in $\Vect(\k)$, whose all terms $X_n$ are concentrated in cohomological degree 0, can be regarded as a graded Leinster monoid such that $X_n=X_n^{0,\dots,0}$, $n\ge 0$.
Then \eqref{emmnew2} is applied, and gives a modification of the classical Eilenberg-MacLane map \eqref{emform}, for the simplicial chain differential(s) replaced by the differential(s) \eqref{dsharp} without the two extreme terms. 3. Note that \eqref{emmnew1} is {\it not} the same as the modified Eileberg-MacLane map from (2.) applied to $X_\ldot^\sharp,Y_\ldot^\sharp$, see Section \ref{sectionsharp}, with forgotten cohomological degree.  Indeed, for the latter case the sign corresponded to a shuffle $\sigma$ is just the classical sign of $\sigma$ (as all $a_i=b_j=0$), instead of the ``correct'' sign defined in \eqref{emmnew3}.
}
\end{remark}

\begin{theorem}\label{emmnewtheorem}
Consider the map $\nabla^\Bar$ defined in \eqref{emmnew2}. The following statements are true:
\begin{itemize}
\item[(I)] $\nabla^\Bar\colon \Bar(X_L)\otimes \Bar(Y_L)\to \Bar(X_L\otimes Y_L)$ is a map of complexes, that is \eqref{emmnew1} holds,
\item[(II)] the map $\nabla^\Bar(-,-)$ is a lax-monoidal map of bifunctors, that is, it is bi-functorial, and for any three graded Leinster monoids in $\Vect(\k)$ the diagrams \eqref{lax1} and \eqref{lax2} commute,
\item[(III)] the map $\nabla^\Bar$ is symmetric, that is, for any two graded Leinster monoids $X_L$ and $Y_L$ in $\Vect(\k)$, the diagram below commutes:
\begin{equation}\label{symmkappas}
\xymatrix{
\Bar(X_L)\otimes\Bar(Y_L)\ar[rr]^{\nabla^\Bar}\ar[d]_{\kappa_1}&&\Bar(X_L\otimes Y_L)\ar[d]^{\kappa_2}\\
\Bar(Y_L)\otimes\Bar(X_L)\ar[rr]^{\nabla^\Bar}&&\Bar(Y_L\otimes X_L)
}
\end{equation}
where, for $x_m\in X_m^{a_1,\dots,a_m}$ and $y_n\in Y_n^{b_1,\dots,b_n}$, one has
\begin{equation}
\kappa_1(x_m\otimes y_n)=(-1)^{(m+a_1+\dots+a_m)(n+b_1+\dots+b_n)}y_n\otimes x_m
\end{equation}
where $x_m$ and $y_n$ are considered as elements of $\Bar(X_L)$ and $\Bar(Y_L)$, correspondingly; and
\begin{equation}
\kappa_2(x_m\otimes y_n)=(-1)^{(a_1+\dots+a_m)(b_1+\dots+b_n)}y_n\otimes x_m
\end{equation}
where $x_m$ and $y_n$ are considered as elements of $X_m$ and $Y_n$.
\end{itemize}
\end{theorem}
The only prove we know relies on some rather long and routine computation, see below. 
For the particular case of Remark \ref{emremnew}, we provide an alternative and more conceptual proof in Appendix.

\hspace{1mm}

The proof is given in Section \ref{proeff}.

\subsection{\sc Application to algebras over operads}
Let $\mathcal{O}$ be a (dg) operad in $\Vect(\k)$, and denote by $\Alg(\mathcal{O})$ the category of (dg) $\mathcal{O}$-algebras.
Denote by $\Alg(\mathcal{O})^{\Delta^\opp_0}$ the category of functors $\Delta_0^\opp\to \Alg(\mathcal{O})$.

Let $X_\ldot\in\Alg(\mathcal{O})^{\Delta^\opp_0}$. By abuse of notations, we denote by the same symbol the underlying functor in $\Fun(\Delta_0^\opp,\Vect(\k))$.
Consider the bar-complex $\Bar(X_\ldot)$ of the underlying $\Delta_0^\opp$-vector space. It is natural to ask whether one can ``extend'' the $\mathcal{O}$-algebra structure from the terms $X_\ell$ to the entire bar-complex $\Bar(X_\ldot)$. The answer is affirmative, and the construction relies on the modified Eilenberg-MacLane map  and on the fact that it is symmetric monoidal and lax-monoidal map of bi-functors, see Theorem \ref{emmnewtheorem}.

\begin{coroll}\label{corollop}
Let $\mathcal{O}$ be an operad in $\Vect(\k)$.
For $X_\ldot\in\Alg(\mathcal{O})^{\Delta^\opp_0}$, there is a natural $\mathcal{O}$-algebra structure on the total complex of the bicomplex $\Bar(X_\ldot)$, functorial in maps in $\Alg(\mathcal{O})^{\Delta_0^\opp}$.
\end{coroll}
\begin{proof}
To endow $\Bar(X_\ldot)$ with a structure of an $\mathcal{O}$-algebra, we need to define maps
\begin{equation}\label{barop1}
\mathcal{O}(n)\otimes (\Bar(X_\ldot))^{\otimes n}\to \Bar(X_\ldot)
\end{equation}
for $n\ge 1$, which are
\begin{itemize}
\item[(i)] compatible with the compositions,
\item[(ii)] equivariant with respect to the actions of symmetric group $\Sigma_n$.
\end{itemize}

We construct \eqref{barop1} as the composition
\begin{equation}\label{barop2}
\mathcal{O}(n)\otimes (\Bar(X_\ldot))^{\otimes n}\xrightarrow{\id\otimes (\nabla^\Bar)^{n-1}}\mathcal{O}(n)\otimes \Bar(X_\ldot^{\otimes n})\xrightarrow{\mathcal{O}}
\Bar(X_\ldot)
\end{equation}
where the second arrow is the term-wise application of the maps
$$
\mathcal{O}(n)\otimes X_\ell^{\otimes n}\to X_\ell
$$
giving the $\mathcal{O}$-algebra structures on the components $X_\ell$.

Now (i) follows from the lax-monoidal property of $\nabla^\Bar$, and (ii) follows from the commutation of $\nabla^\Bar$ with the symmetric braidings, see Theorem \ref{emmnewtheorem}(iii).

\end{proof}

\subsection{\sc The Eilenberg-MacLane map $\nabla^\Bar$ is a map of dg coalgebras}
Here we prove the following
\begin{theorem}\label{theoremcoalg}
Let $X_L,Y_L$ be graded Leinster monoids in $\Vect(\k)$. Consider the bar-complexes $\Bar(X_L)$ and $\Bar(Y_L)$ as dg coalgebras, see Section \ref{sectiongraded}.
Then the Eilenberg-MacLane map for the bar-complexes (defined in Section \ref{sectionemmnew})
\begin{equation}
\nabla^\Bar\colon \Bar(X_L)\otimes \Bar(Y_L)\to\Bar(X_L\otimes Y_L)
\end{equation}
is a map of dg coalgebras.
\end{theorem}

The proof is given in Section \ref{proofcoalg}.

\subsection{\sc The case of Leinster monoids in $\mathcal{O}$-algebras}
The operads in a symmetric monoidal category $(\mathscr{M},\otimes)$ form, by their own, a symmetric monoidal category.
The product is
$$
(\mathcal{O}\otimes \mathcal{O}^\prime)(n)=\mathcal{O}(n)\otimes\mathcal{O}^\prime(n)
$$
with the natural action of the symmetric group $\Sigma_n$.

We consider the case $(\mathscr{M},\otimes)=(\Vect(\k),\otimes_k)$.
Let $\mathcal{O}$ be {\it a coalgebra object} in the symmetric monoidal category of operads in $\Vect(\k)$. That is, for any $n\ge 0$ there is a map
\begin{equation}
\Delta_n\colon \mathcal{O}(n)\to\mathcal{O}(n)\otimes\mathcal{O}(n)
\end{equation}
compatible with the operadic composition and with the actions of the symmetric group.

In this case, the $\mathcal{O}$-algebras form a monoidal category. Indeed, let $A,B$ be two $\mathcal{O}$-algebras.
To define an $\mathcal{O}$-algebra structure 
\begin{equation}
\mathcal{O}(n)\otimes (A\otimes B)^{\otimes n}\to A\otimes B
\end{equation}
consider the composition
$$
\mathcal{O}(n)\otimes(A\otimes B)^{\otimes n}\xrightarrow{\Delta_n\otimes\id}\mathcal{O}(n)\otimes \mathcal{O}(n)\otimes (A\otimes B)^{\otimes n}\xrightarrow{}(\mathcal{O}(n)\otimes A^{\otimes n})\otimes (\mathcal{O}(n)\otimes B^{\otimes n})\rightarrow A\otimes B
$$
It endows the category $\Alg(\mathcal{O})$ with a monoidal structure. This monoidal structure is {\it symmetric}, if the operadic coproduct
$\Delta\colon \mathcal{O}\to\mathcal{O}\otimes\mathcal{O}$ is {\it cocommutative}.

For operads $\mathcal{O}$, which are coalgebra objects in operads in $\Vect(\k)$, the category $\Alg(\mathcal{O})$ is monoidal, and we can talk about {\it Leinster monoids in $\Alg(\mathcal{O})$}.
As usual, {\it a Leinster monoid} in $\Alg(\mathcal{O})$ is defined as a colax-monoidal functor $X_L\colon \Delta_0^\opp\to \Alg(\mathcal{O})$ such that the components of the colax-monoidal structure $\beta$, $$\beta_{m,n}\colon X_{m+n}\to X_m\otimes X_n$$ 
are quasi-isomorphisms of $\mathscr{O}$-algebras.

\begin{theorem}\label{theorema}
Let $\mathcal{O}$ be a coalgebra object operad in the symmetric monoidal category $\Vect(\k)$, so that $\Alg(\mathcal{O})$ is a monoidal category. Let $X_L$ be a Leinster monoid in $\Alg(\mathcal{O})$, and let $\Bar(X_L)$ be the bar-complex of the underlying Leinster monoid in $\Vect(\k)$. It is an $\mathcal{O}$-algebra, by Corollary \ref{corollop}, and it is as well a dg coalgebra in $\Vect(\k)$, with the coproduct given in \eqref{barcoalg}. Then the bar-complex $\Bar(X_L)$ is a coalgebra object in $\mathcal{O}$-algebras, that is the coproduct in $\Bar(X_L)$ given by in \eqref{barcoalg} is a map of $\mathcal{O}$-algebras.
\end{theorem}
\begin{coroll}\label{corcoalg}
Consider the operad $\mathcal{O}=\Assoc$ in $\Vect(k)$. It is a coalgebra object in operads in $\Vect(\k)$. Consequently, for any Leinster monoid $X_L\colon\Delta_0^\opp\to\Alg(\k)$, the bar-complex $\Bar(X_L)$ is a dg bialgebra.
\end{coroll}
\qed

Now we prove Theorem \ref{theorema}.
\begin{proof}
We need to show that the coproduct
\begin{equation}
\Delta_\Bar\colon \Bar(X_L)\to\Bar(X_L)\otimes \Bar(X_L)
\end{equation}
is a map of $\mathcal{O}$-algebras, that is, the diagram below commutes:
\begin{equation}\label{yui1}
\xymatrix{
\mathcal{O}(n)\otimes (\Bar(X_L))^{\otimes n}\ar[d]\ar[rr]^{\id_\mathcal{O}\otimes\Delta_\Bar^{\otimes n}\hspace{30pt}}&& \mathcal{O}(n)\otimes (\Bar(X_L)\otimes\Bar(X_L))^{\otimes n}\ar[d]\\
\Bar(X_L)\ar[rr]^{\Delta_\Bar\hspace{10pt}}&&\Bar(X_L)\otimes\Bar(X_L)
}
\end{equation}
By so far, we made use only a part of the information encoded in the structure of a Leinster monoid in $X_L$, namely, that the simplicial operators in $X_\ldot$ are maps of $\mathcal{O}$-algebras. Now we use the remaining part of the structure, namely, that all colax-maps $\beta_{a,b}\colon X_{a+b}\to X_a\otimes X_b$ are maps of $\mathcal{O}$-algebras. It results in the commutativity of the following diagram:
\begin{equation}\label{yui2}
\xymatrix{
\mathcal{O}(n)\otimes X_{a+b}^{\otimes n}\ar[d]_{m_n}\ar[rr]^{\Delta_n\otimes \beta_{a,b}^n\hspace{35pt}}&&\mathcal{O}(n)\otimes \mathcal{O}(n)\otimes X_a^{\otimes n}\otimes X_b^{\otimes n}\ar[d]^{m_n\otimes m_n}\\
X_{a+b}\ar[rr]^{\beta_{a,b}\hspace{20pt}}&&X_a\otimes X_b
}
\end{equation}
Consider the following diagram:
\begin{equation}\label{yui3}
\xymatrix{
\mathcal{O}(n)\otimes (\Bar(X_L))^{\otimes n}\ar[rr]^{\Delta_\Bar^{\otimes n}\hspace{40pt}}\ar[d]_{(\nabla^\Bar)^{n-1}}&&\mathcal{O}(n)\otimes (\Bar(X_L))^{\otimes n}\otimes (\Bar(X_L))^{\otimes n}\ar[d]^{\Delta_n\otimes ((\nabla^\Bar)^{n-1})^{\otimes 2}}\\
\mathcal{O}(n)\otimes\Bar(X_L^{\otimes n})\ar[rr]^{\Delta_n\otimes \Delta_{\Bar}\hspace{40pt}}\ar[d]^{\mathcal{O}}&&\mathcal{O}(n)\otimes\mathcal{O}(n)\otimes\Bar(X_L^{\otimes n})\otimes\Bar(X_L^{\otimes n})\ar[d]^{\mathcal{O}\otimes\mathcal{O}}\\
\Bar(X_L)\ar[rr]^{\Delta_\Bar\hspace{30pt}}&&\Bar(X_L)\otimes \Bar(X_L)
}
\end{equation}
It is enough to prove the commutativity of each of two sub-diagrams in \eqref{yui3}, as the the total diagram is \eqref{yui1}. The commutativity of the upper sub-diagram follows from Theorem \ref{theoremcoalg}, whence the commutativity of lower sub-diagram follows from \eqref{yui2}.
In fact, for the lower sub-diagram, $X_L^{\otimes L}$ is regarded as a Leinster monoid in $\Vect(\k)$, and we interpret \eqref{yui2} as the statement that $X_L$ is an $\mathcal{O}$-algebra in $\Alg(\mathcal{O})$.
\end{proof}

\section{\sc Theorem \ref{theorem2intro}}
\subsection{\sc Formulation}
One of the main results of the paper is:
\begin{theorem}\label{mtheorem}
Let $X_L$ be a graded Leinster monoid in the category $\Alg(\k)$ of dg algebras over $\k$. Then there is a  $C_\ldot(E_2,\k)$-algebra structure on a space $\mathscr{Y}(X_L)$ quasi-isomorphic to the first component $X_{1}$ of $X_L$ such that the underlying $C_\ldot(E_1,\k)$-algebra on $\mathscr{Y}(X_L)$ is $A_\infty$-quasi-isomorphic to the associative algebra $X_{1}\in \Alg(\k)$. The correspondence
\begin{equation}
X_L\rightsquigarrow \mathscr{Y}(X_L)
\end{equation}
gives rise to a functor
\begin{equation}
\mathscr{Y}\colon \mathscr{L}^\gr(\Alg(\k))\to \Alg(C_\ldot(E_2,\k))
\end{equation}
This functor maps weakly equivalent graded Leinster monoids in $\Alg(\k)$ to quasi-isomorphic $C_\ldot(E_2,\k)$-algebras.
\end{theorem}
Some explanations are in order. There is a standard map of topological operads $i\colon E_1\to E_2$ inducing the map of chain operads $i_*\colon \Alg(C_\ldot(E_2,\k))\to\Alg(C_\ldot(E_1,\k))$. Now the Stasheff operad $A_\infty$ is quasi-isomorphic to the operad $C_\ldot(E_1,\k)$, which makes possible to consider $\mathscr{Y}(X_L)$ as an $A_\infty$-algebra. 

The proof of this Theorem occupies Section \ref{sectionoccupy1} below.

The strategy is as follows. We set
$$
\mathscr{Y}(X_L)=\Cobar(\Bar(X_L))
$$
By Theorem \ref{theorem27}, $\mathscr{Y}(X_L)$ is quasi-isomorphic to $X_1$.
On the other hand, $\Bar(X_L)$ is a dg coalgebra, which is also a dg bialgebra, by Theorem \ref{theorema}.
For any bialgebra $B$, the cobar-complex $\Cobar(B)$ of $B$ with the underlying dg coalgebra structure is an algebra over the Getzler-Jones $B_\infty$ operad.
The construction is referred to as the Tamarkin-Kadeishvili construction [T4],[Kad]. We recall the basic facts on the Getzler-Jones operad in Section \ref{sectionbraces} and recall the Tamarkin-Kadeishvili construction in Section \ref{kadconst}.

There is a map of operads $C_\ldot(E_2,\k)\to B_\infty(\k)$, constructed by McClure-Smith [MS1,2], what makes possible to consider any $B_\infty$-algebra as a $C_\ldot(E_2,\k)$-algebra. Then it will remain to check the statement on the underlying $A_\infty$-algebras, which is done in Section \ref{occupy1.2}.

\comment
For the case $n>1$ we make use of the monoidal structure on $\mathsf{e}_n$-algebras and some ideas of [T3], see Section \ref{sectionoccupy2}.
Recall, that a homotopy $n$-algebra $Y$ is given by a differential in the cofree $n$-coalgebra $\mathsf{e}_n(Y)$ cogenerated by $Y[n]$.
Recall from [T3, Section 3.1] that the strict $n$-algebras form a symmetric monoidal category.
Finally, for a homotopy $n$-algebra $Y$ (that is, for a cofree dg $n$-coalgebra $\mathsf{e}_n(Y)$) there is the dual $n$-algebraic cobar-construction $\mathsf{e}_n^\vee(\mathsf{e}_n(Y))$ which is a strict dg $n$-algebra. It is a dg $\mathsf{e}_n$-algebra, whose underlying $\mathsf{e}_n$-algebra is
the free $n$-algebra generated by $(\mathsf{e}_n(X))[-n]$. To perform the induction step from $n$ to $n+1$ in Theorem \ref{mtheorem}, we prove simultaneously with it the following claim.
\begin{theorem}\label{mtheorem2}
Let $N\ge 1$. In the notations of Theorem \ref{mtheorem}, the functor $\mathscr{Y}$ from Leinster $N$-monoids in $\Alg(\k)$ to homotopy $(N+1)$-algebras in $\k$ enjoys additionally the following colax-monoidal property: for any two Leinster $N$-monoids $X_L,X_L^\prime$ in $\Alg(\k)$ there is a map of strict $(N+1)$-algebras
\begin{equation}\label{mt2eq}
B_{X,Y}\colon   \mathsf{e}_{N+1}^\vee\mathsf{e}_{N+1}(\mathscr{Y}(X_L\otimes X_L^\prime))\to \mathsf{e}_{N+1}^\vee\mathsf{e}_{N+1}(\mathscr{Y}(X_L))\otimes \mathsf{e}_{N+1}^\vee\mathsf{e}_{N+1}(\mathscr{Y}(X_L^\prime))
\end{equation}
which satisfy the colax-properties \eqref{colax1}, \eqref{colax2} and which are quasi-isomorphisms of strict $(N+1)$-algebras.
In other words, the functor $$\mathsf{e}_{N+1}^\vee\circ \mathsf{e}_{N+1}\circ\mathscr{Y}\colon X_L\rightsquigarrow \mathsf{e}_{N+1}^\vee\mathsf{e}_{N+1}(\mathscr{Y}(X_L))$$
from Leinster $N$-monoids in $\Alg(\k)$ to strict dg $(N+1)$-algebras is colax-monoidal, with quasi-isomorphisms as the colax-maps.
\end{theorem}
Here in the l.h.s. of \eqref{mt2eq} there stands the monoidal product of strict dg $(N+1)$-algebras.

The case $N=1$ of this Theorem is proven in Section \ref{mth2n=1}, and the case $N>1$ in.

\endcomment

In the rest of this Section, we recall the operad $B_\infty$ (Section \ref{sectionbraces}) and the Tamarkin-Kadeishvili construction (Section \ref{kadconst}).

\subsection{\sc The Getzler-Jones $B_\infty$ operad}\label{sectionbraces}
\begin{defn}{\rm
Let $V\in\Vect(\k)$. We say that $V$ is a $B_\infty$ algebra if there is a structure of a {\it dg bialgebra} on the cofree coalgebra $T^\vee(V[1])$ cogenerated by $V[1]$, whose underlying coalgebra is the cofree one.
}
\end{defn}
By this definition, the data we need to endow $V$ with a $B_\infty$ structure is (i) a differential on $T^\vee(V[1])$ (which endows $V$ with an $A_\infty$ structure), (ii) a structure of algebra on $T^\vee(V[1])$ (which encodes the ``higher structure'' like a Lie bracket of degree -1 on $V$) such that (A) the differential and the algebra structure on $T^\vee(V[1])$ are compatible and define a dg algebra structure on $T^\vee(V[1])$, and (B) the compatibility axiom $\Delta(a*b)=\Delta(a)*\Delta(b)$ is fulfilled.

There are three notable examples of $B_\infty$ algebras, which are:
\begin{itemize}
\item[(i)] $V=C^\udot_{\mathrm{sing}}(X,A)$ is a singular cochain complex of a simply-connected topological space $X$ with coefficients in any ring $A$, it was constructed by Baues [B]. When $A=\mathbb{Z}$
it refers to the famous ``problem of commutative cochains" in topology and the Steenrod operations,
\item[(ii)] $V=\Hoch^\udot(A)$, the cohomological complex of a dg algebra $A$, the construction firstly appeared in [GJ], and was used in Tamarkin's proof of the Kontsevich formality [T1], [H]; see also [T3, Introduction] for a formula-free more conceptual explanation,
\item[(iii)] $V=\Cobar(B_\coalg)$ where $B$ is an associative bialgebra, $B_\coalg$ is its underlying coalgebra. Then $V$ is a dg algebra, and the differential part of the dg bialgebra structure on $T^\vee(V[1])$ is the bar-differential of this dg algebra. Tamarkin [T4] and Kadeishvili [Kad] have shown by an explicit construction that $T^\vee(V[1])$ has an associative algebra structure, which makes $V$ a $B_\infty$ algebra.
\end{itemize}
We recall the construction of the $B_\infty$ algebra of Example (iii) (which we use in the paper) in Section \ref{kadconst} below.

Explicitly, to define a structure of a $B_\infty$ algebra on a complex of $\k$-vector spaces $X$, one needs to specify the differential and the product on $T^\vee(X[1])$.
As the underlying coalgebra is cofree, they are defined by the projections to the cogenerators, which are subject to some identities. These projections are:
\begin{equation}
m_n\colon X[1]^{\otimes n}\to X[1],\ \ n\ge 2
\end{equation}
and
\begin{equation}
m_{k,\ell}\colon X[1]^{\otimes k}\otimes X[1]^{\otimes \ell}\to X[1],\ \ k,\ell\ge 0
\end{equation}

\comment
\begin{defn}{\rm
Let $V\in\Vect(\k)$. We say that $V$ is a $B_\Lie$ algebra if there is a dg Lie bialgebra structure on the cofree Lie coalgebra $\Lie^\vee(V[1])$ cogenerated by $V[1]$ whose underlying Lie coalgebra structure is the cofree one.
}
\end{defn}

Unlike the operad $B_\infty$, the operad $B_\Lie$ is very simple, due to the following Lemma:
\begin{lemma}\label{ll2}
Let $V\in\Vect(\k)$ be an algebra over the operad $B_\Lie$. Then $V$ is endowed with a homotopy 2-algebra structure. Moreover, there is a map of operads $\hoe_2\to B_\Lie$ which is a quasi-isomorphism of dg operads.
\end{lemma}
\begin{proof}
Assume that $\frak{g}=\Lie^\vee(X[1])$ is a Lie bialgebra. Consider the underlying Lie algebra, and take the Chevalley-Eilenberg chain complex $$Y=C_{\mathrm{CE}}(\frak{g}_\Lie;\k)=S^\vee(\Lie^\vee(X[1])[1])$$
Then $Y$ is a cofree $\mathsf{e}_2$-coalgebra cogenerated by $X[2]$; the structue of $B_\Lie$-algebra on $X$ is transformed to the total differential on $Y$, compatible with the $\mathsf{e}_2$-coalgebra structure. By definition, it is a homotopy 2-algebra structure on $X$. It is well-known (and can be easily seen) that it is a quasi-isomorphism of dg operads.
\end{proof}

The crucial is the following fact, see [H, Theorem 7.3]. It uses the Etingof-Kazhdan (de)quantization which by its own relies on the theory of Drinfeld associators.
\begin{theorem}\label{theoremek}
There exists an isomorphism dg operads
\begin{equation}
D\colon B_\Lie\to B_\infty
\end{equation}
which enjoys the following property.
For any $B_\infty$-algebra $X$, the underlying complex of the cofree dg bialgebra on $T^\vee(X[1])$ is isomorphic to the symmetric algebra of the underlying complex of the dg Lie bialgebra on $\Lie^\vee(D(X)[1])$.
The isomorphism $D$ is canonically constructed by a chosen Drinfeld associator.
\end{theorem}
\begin{proof}
The claim that the dg operads $B_\Lie$ and $B_\infty$ are isomorphic, is proven in [H, Theorem 7.3], by an application of the Etingof-Kazhdan (de)quantization. The mentioned property on the underlying complexes follows from the corresponding properties of the Etingof-Kazhdan (de)quantization. See also [EG, Theorem 4.2] for a more direct than [H, Section 7] approach.
\end{proof}
\endcomment

Let $X$ be a $B_\infty$-algebra. When we forget the product on $T^\vee(X[1])$, we get a differential on $T^\vee(X[1])$ which is coderivation of the cofree coalgebra structure. The latter structure is an $A_\infty$ algebra structure on $X$. We call it {\it the underlying $A_\infty$ structure of the $B_\infty$ structure}. 

\begin{remark}{\rm
It is true that the underlying $A_\infty$ algebra of a $B_\infty$ algebra $X$ is $A_\infty$ quasi-isomorphic to a homotopy commutative algebra. However, it does not imply that the differential on $T^\vee(X[1])$ is given by a Harrison cochain, cf. [GJ], Prop. 5.3 and Prop. 5.5.
}
\end{remark}

First of all, we prove the following $B_\infty$ version of Theorem \ref{mtheorem}:
\begin{theorem}\label{mtheorembis}
Let $X_L$ be a graded Leinster monoid in the category $\Alg(\k)$ of dg algebras over $\k$. Then there is a  $B_\infty$-algebra structure on the space $\Cobar(\Bar((X_L)))$, quasi-isomorphic to the first component $X_{1}$ of $X_L$, such that the underlying $A_\infty$-algebra on $\Cobar(\Bar((X_L)))$ is $A_\infty$-quasi-isomorphic to the dg associative $X_{1}\in \Alg(\k)$. The correspondence
\begin{equation}
X_L\rightsquigarrow \Cobar(\Bar((X_L)))
\end{equation}
gives rise to a functor
\begin{equation}
\mathscr{Y}\colon \mathscr{L}^\gr(\Alg(\k))\to B_\infty(\k)
\end{equation}
This functor maps weakly equivalent graded Leinster monoids in $\Alg(\k)$ to quasi-isomorphic $B_\infty$-algebras.
\end{theorem}
After that, we show how the results of [MS1,2] help to deduce Theorem \ref{mtheorem} from Theorem \ref{mtheorembis}.

\subsection{\sc The Tamarkin-Kadeishvili construction}\label{kadconst}
Here we follow [Kad].

Let $B$ be a dg (co)associative bialgebra. Here we endow the cobar-complex $\Cobar(B_\coalg)$ of the underlying coalgebra $B$ with a $B_\infty$ algebra structure.
The cobar-complex $\Cobar(C)$ of any dg coalgebra $C$ has a structure of a dg associative algebra, see Lemma \ref{lemmabc2}.
Then $T^\vee(\Cobar(C)[1])\simeq \Bar(\Cobar(C))$ inherits the bar-differential of the dg algebra $\Cobar(C)$.

One has:
\begin{theorem}\label{theoremkad}
Let $B$ be a dg (co)associative bialgebra. Then there exists a $B_\infty$ algebra structure on $Y=\Cobar(B)$ such that the differential on the cofree coalgebra $T^\vee(Y[1])\simeq\Bar(Y)$ is the bar-differential.
\end{theorem}

As the differential on $T^\vee(Y[1])$ is already fixed by the assumption, one only needs to construct the associative product Taylor components
\begin{equation}\label{kad0}
m_{k,\ell}\colon Y[1]^{\otimes k}\otimes Y[1]^{\otimes \ell}\to Y[1]
\end{equation}
Recall the construction. The maps $m_{k,\ell}=0$ for $k>1$ (that is, $\Cobar(B)$ is a brace algebra).

For $a\in B$ and $y\in B^{\otimes n}$ introduce the operation $a\lozenge y\in B^{\otimes n}$ as
$$
a\lozenge y=(\Delta^{n-1}a)*y
$$
where in the r.h.s. there is the factor-wise product in $B$ of two elements in $B^{\otimes n}$.

The operation $a\lozenge y$ is well-defined for homogeneous $y$ of any degree $n$, then the power of the coproduct is defined accordingly to $n$.

Denote by $[a_1\otimes\dots\otimes a_n]$ an element in $\Cobar(B)=T^\udot(B[-1])$, $a_i\in B$, and let $y_1,\dots,y_\ell\in\Cobar(B)$
be homogeneous elements.

The formula for $m_{1,\ell}$ reads:
\begin{equation}\label{kad1}
\begin{aligned}
\ &m_{1,\ell}([a_1\otimes\dots\otimes a_n], y_1\otimes\dots \otimes y_\ell)=\\
&\sum_{\substack{1\le i_1<i_2<\dots<i_\ell\le n}}\pm[a_1\otimes\dots\otimes a_{i_1-1}\otimes (a_{i_1}\lozenge y_1)\otimes a_{i_1+1}\otimes\dots \otimes a_{i_\ell-1}\otimes (a_{i_\ell}\lozenge y_\ell)\otimes a_{i_\ell+1}\otimes\dots\otimes a_n]
\end{aligned}
\end{equation}
for $\ell\le n$, and is set to be zero for $\ell>n$.

Tamarkin [T4] and Kadeishvili [Kad, Section 5] proved that these operations define a dg (co)associative bialgebra structure on $T^\vee(\Cobar(B)[1])$.

\qed

\section{\sc A proof of Theorem \ref{mtheorem}}\label{sectionoccupy1}
Let $X_L$ be a graded Leinster monoid in $\Alg(\k)$.
Set $$Y(X_L)=\Cobar(\Bar(X_L))$$
First of all, we prove Theorem \ref{mtheorembis}. After that, we deduce Theorem \ref{mtheorem} from Theorem \ref{mtheorembis}, and McClure-Smith results [MS1,2].

We know that $Y(X_L)=\Cobar(\Bar(X_L))$ is a $B_\infty$ algebra, as $\Bar(X_L)$ is a dg bialgebra by Theorem \ref{theorema}, and then the Tamarkin-Kadeishvili construction from Section \ref{kadconst} endows $\Cobar(\Bar(X_L))$ with a $B_\infty$-algebra structure.

\subsection{\sc The underlying $A_\infty$-algebra of the $B_\infty$-algebra $Y(X_L)$ is $A_\infty$
quasi-isomorphic to $X_1$}\label{occupy1.2}

The underlying $A_\infty$ algebra of a $B_\infty$ algebra $X$ is given by the underlying dg coalgebra of the dg bialgebra 
$T^\vee(X[1])$. When $X=\Cobar(B_\coalg)$, for a dg bialgebra $B$, and the $B_\infty$ structure on $X$ is given by the Tamarkin-Kadeishvili construction (see Section \ref{kadconst}), the differential of the underlying dg coalgebra $T^\vee(X[1])$ is 
the bar-differential of the dg algebra $\Cobar(B_\coalg)$. 

We need to construct an $A_\infty$ quasi-isomorphism 
$$
\phi\colon X_1\to (\Cobar(\Bar(X_L)_\coalg)_\alg
$$

It is constructed in Theorem \ref{mainyon} below.

Introduce some notations.

Denote by $w_1^{(n)},\dots, w_n^{(n)}$ the morphisms $[0,1,\dots,n]\to [0,1]$ in $\Delta_0$, defined as
\begin{equation}\label{yon0}
w_k^{(n)}(j)=\begin{cases}0& \text{if }j\le k-1\\ 1& \text{if }j\ge k
\end{cases}
\end{equation}
($1\le k\le n$). We drop the upper index out from the notation $w_i^{(j)}$ whenever it is clear from the context.

By abuse of notations,  we denote by the same symbols the corresponding maps $w_k\colon X_1\to X_n$, for a Leinster monoid $X_L$.

We have 
\begin{theorem}\label{mainyon}
Let $X_L$ be a Leinster monoid in $\Alg(\k)$. Consider its component $A_1=X_1$ with its underlying associative dg algebra structure, and $A_2=\Cobar(\Bar(X_L))$ with its associative dg algebra structure as on a cobar-complex of a dg coalgebra. Define maps
$$
\phi_n\colon (A_1[1])^{\otimes n}\to A_2[1]
$$
as
\begin{equation}\label{yonbasic}
\begin{aligned}
\ &\phi_n(x_1\otimes\dots\otimes x_n)=(-1)^{(n-1)\deg x_1+(n-2)\deg x_2+\dots+\deg x_{n-1}}\cdot w_1(x_1)*\dots*w_n(x_n)\\
&\in X_n[-n+1]\subset\Bar(X_L)[-1]\subset\Cobar(\Bar(X_L))
\end{aligned}
\end{equation}
as the product in $X_n$ of the corresponding elements.
Then $\{\phi_n\}_{n\ge 1}$ are the Taylor components of an $A_\infty$ morphism $\Phi\colon A_1\to A_2$, which is also an $A_\infty$ quasi-isomorphism.
\end{theorem}

We prove Theorem \ref{mainyon} in Section \ref{proefmy}.

To complete the proof of Main Theorem \ref{mtheorembis}, it remains to show that the assignment $X_L\rightsquigarrow Y(X_L)$ preserves quasi-isomorphisms. It follows from the corresponding preservation for each step of our construction.

Theorem \ref{mtheorembis} is proven.

\qed

\subsection{\sc McClure-Smith theory [MS1,2] and Theorem \ref{mtheorem}}\label{sms}
Here we briefly show how the results of [MS1,2] make possible to deduce Theorem \ref{mtheorem} from Theorem \ref{mtheorembis}.

In [MS2], the authors construct a filtered operads $S_*$ in abelian groups,
$$
S_1\subset S_2\subset \dots 
$$
$$
\cup_{i\ge 1} S_i=S_*
$$
such that $S_*$ is a model for $E_\infty$ operad, and $S_n$ is quasi-isomorphic (as a dg operad) to the normalized chain operad $C_*(E_n,\mathbb{Z})$ of the topological little $n$-disk operad.

The topological operad $E_n$ is naturally filtered,
$$
E_1\subset E_2\subset\dots \subset E_n
$$
and this filtration gives rise to a filtration on $C_\ldot(E_n,\mathbb{Z})$.

For $n=2$ it is shown [MS2, Section 4] that there is a natural map of operads $S_2\to B_\infty$, which is also a map of filtered operads.
Here the operad $B_\infty$ is considered with its two-term filtration:
$$
A_\infty \subset B_\infty
$$
(the $A_\infty$ structure on a $B_\infty$ algebra $X$ is given by the underlying dg coalgebra structure on the dg bialgebra $T^\vee(X[1])$; that is, is given by the {\it differential} in the dg bialgebra $T^\vee(X[1])$).

The zig-zag of quasi-isomorphisms 
$$
C_\ldot(E_2,\mathbb{Z})\rightarrow \dots \leftarrow S_2
$$
constructed in [MS2], end of Section 5, is formed by filtered operads and filtered maps between them. 

By Theorem \ref{mtheorembis}, there is an action of the operad $B_\infty$ on $Y(X_L)=\Cobar(\Bar(X_L))$, for a Leinster monoid $X_L$ in $\Alg(\k)$, such that the restriction to the sub-operad $A_\infty\subset B_\infty$ is equivalent to the Yoneda product on $X_1\simeq Y(X_L)$. Then the above zig-zag of filtered quasi-isomorphisms of operads gives Theorem \ref{mtheorem}.

\qed

\section{\sc The Leinster monoids constructed in [Sh1] are graded}
In this Section, we show that the Leinster $n$-monoids that appear in our approach to Deligne conjecture [Sh1] are {\it graded}, see Definition \eqref{grl}. To make the exposition shorter, we consider only the case $n=1$ (which is the only case relevant for this paper), though the case of general $n$ is treated without any essential changes.

We start with definition of a Leinster monoid in $\Cat^\dg(\k)$, the category of dg categories. Then we show that the ``refined Drinfeld dg quotient" [Sh1, Section 4] preserves the graded Leinster monoids in $\Cat^\dg(\k)$. It is the key-point, and then we deduce the result directly.

\subsection{\sc Graded Leinster monoids in $\Cat^\dg(\k)$}\label{sectionapp1}
The category $\Cat^\dg(\k)$ is (symmetric) monoidal, with $\mathscr{C}\otimes\mathscr{D}$ as the monoidal product. We consider $\Cat^\dg(\k)$ as a category with weak equivalences, taking for weak equivalences the quasi-equivalences of dg categories.
Therefore, we can speak on Leinster monoids in $\Cat^\dg(\k)$, see Definition \ref{lmon}.
\begin{defn}\label{grldg}{\rm
A Leinster monoid $X_L$ in $\Cat^\dg(\k)$ is called {\it graded} if, for each $n\ge 0$, and for any two objects $x,y\in X_n$, the vector space $\Hom_{X_n}(x,y)$ is $\mathbb{Z}^n$-graded, such that
\begin{itemize}
\item[(i)] the composition $\Hom_{X_n}(x,y)\otimes \Hom_{X_n}(y,z)\to \Hom_{X_n}(x,z)$ sends 
$\Hom_{X_n}(x,y)^{a_1,\dots,a_n}\otimes \Hom_{X_n}(y,z)^{b_1,\dots,b_n}$ to $\Hom_{X_n}(x,z)^{a_1+b_1,\dots,a_n+b_n}$;
\item[(ii)] the colax-maps $\beta_{a,b}\colon X_n\to X_a\otimes X_b$ are $\mathbb{Z}^n$-graded on morphisms, where 
the $\mathbb{Z}^n$-grading on the r.h.s. is obtained by concatenation of the $\mathbb{Z}^a$ and $\mathbb{Z}^b$-gradings on the factors;
\item[(iii)] the total grading $$\Hom_{X_n}(x,y)^A:=\bigoplus_{a_1+\dots+a_n=A}\Hom_{X_n}(x,y)^{a_1,\dots,a_n}$$
coincides with the cohomological $\mathbb{Z}$-grading on the complex $\Hom_{X_n}(x,y)^\udot$.
\end{itemize}

}
\end{defn}

We have directly:
\begin{lemma}\label{lemmaapp}
Let $X_L$ be a graded Leinster monoid in $\Cat^\dg(\k)$, such that each dg category $X_n$ has a distinguished object $*_n$, and the sequence $*_0,*_1,*_2,\dots$ is invariant by the action of all morphisms in $\Delta_0^\opp$, and $\beta_{a,b}(*_{a+b})=*_a\times *_b$. Then the formula
$$
Y_n:=\Hom_{X_n}(*_n,*_n)
$$
defines a graded Leinster monoid in $\Alg(\k)$, see Definition \ref{grla}.
\end{lemma}

\qed

Note that Definition \ref{grla} is a particular case of Definition \ref{grldg}, for dg categories with a single object.

\subsection{\sc Refined Drinfeld dg quotient}\label{sectionapp2}
Recall our definition of the {\it refined Drinfeld dg quotient}, [Sh1, Sections 4.3, 4.4].

Consider the category of tuples $(\mathscr{C};\ \mathscr{C}_1,\dots,\mathscr{C}_n)$ where $\mathscr{C}$ is a dg category over $\k$, and $\mathscr{C}_1,\dots,\mathscr{C}_n$ its small full dg subcategories. A morphism 
$$
f\colon (\mathscr{C};\ \mathscr{C}_1,\dots,\mathscr{C}_n)\to (\mathscr{D};\ \mathscr{D}_1,\dots,\mathscr{D}_n)
$$
is given by a dg functor $f\colon\mathscr{C}\to\mathscr{D}$ such that $f(\mathscr{C}_i)\subset\mathscr{D}_i$.

Denote this category, for fixed $n$, by $\mathscr{P}_n\Cat^\dg(\k)$.

The refined Drinfeld dg quotient is a functor 
$$
\mathscr{D}r\colon \mathscr{P}_n\Cat^\dg(\k)\to\Cat^\dg(\k)
$$
The motivation for introducing this refinement was that the ordinary Drinfeld dg quotient does not have any good monoidal properties, whence the refined dg quotient is colax-monoidal functor, in a suitable sense, see Section \ref{sectionapp4} below.

The construction goes as follows.

For any $X=X_i\in \mathscr{C}_i$, $i=1,\dots,k$, we add formally an element $\varepsilon^i_X$ which is a morphism from $X$ to $X$ of degree -1, with the differential $d\varepsilon^i_X=\id_X$.

For any $X=X_{ij}\in\mathscr{C}_i\cap\mathscr{C}_j$, $i<j$, we introduce formally a morphism $\varepsilon^{ij}_X$ from $X$ to itself of degree -2, with $d\varepsilon^{ij}_X=\varepsilon^i_X-\varepsilon^j_X$.

For any $X=X^{ijk}\in\mathscr{C}_i\cap\mathscr{C}_j\cap\mathscr{C}_k$, $i<j<k$, we introduce formally a morphism $\varepsilon_X^{ijk}$ from $X$ to itself of degree -3, with $d\varepsilon^{ijk}_X=\varepsilon^{ij}_X-\varepsilon^{ik}_X+\varepsilon^{jk}_X$, and so on.

The sign rule is as in the \v{C}ech cohomology theory, which implies $d^2=0$.

Now the dg category $\mathscr{C}/(\mathscr{C}_1,\dots,\mathscr{C}_k)$ has the same objects as the dg category $\mathscr{C}$, and the morphisms are freely generated by the morphisms in $\mathscr{C}$ with the given composition among them, and by the newly added morphisms $\varepsilon_X^{i_1i_2\dots i_s}$ of degree $-s$, with the differentials of $\varepsilon_X^{i_1,\dots,i_s}$ defined as above and extended to the whole morphisms by the Leibniz rule. 

Denote by $\mathscr{C}_\Sigma$ the full dg subcategory of $\mathscr{C}$ having the objects
$$
\Ob\mathscr{C}_\Sigma=\bigcup_{i=1}^k\Ob\mathscr{C}_i
$$

Consider the ordinary Drinfeld dg quotient $\mathscr{C}/\mathscr{C}_\Sigma$.
There is a natural dg functor
\begin{equation}\label{psi}
\Psi\colon\mathscr{C}/(\mathscr{C}_1,\dots,\mathscr{C}_k)\to\mathscr{C}/\mathscr{C}_\Sigma
\end{equation}
sending all $\varepsilon_X^i$ to $\varepsilon_X$, for $X\in\Ob\mathscr{C}_\Sigma$, and sending all $\varepsilon_X^{i_1\dots i_s}$ to 0, for $s>1$.

\begin{prop}
The functor $\Psi$ defined above is a quasi-equivalence of dg categories.
\end{prop}
See [Sh1], Lemma 4.3 for a proof.

\qed

\subsection{\sc The refined Drinfeld dg quotient and the grading}\label{sectionapp3}
\begin{defn}\label{dgg}{\rm
Let $\mathscr{C}$ be a dg category over $\k$. We say that $\mathscr{C}$ is $\mathbb{Z}^\ell$-graded if there is a $\mathbb{Z}^\ell$-grading on each vector space $\Hom_\mathscr{C}(x,y)$ such that:
\begin{itemize}
\item[(i)] the compositions maps $\Hom_\mathscr{C}(y,z)^{b_1,\dots,b_\ell}\otimes \Hom_\mathscr{C}(x,y)^{a_1,\dots,a_\ell}$ to
$\Hom_\mathscr{C}(x,z)^{a_1+b_1,\dots,a_\ell+b_\ell}$, for any $x,y,z\in\Ob\mathscr{C}$;
\item[(ii)] the total grading 
$$
\Hom_\mathscr{C}(x,y)^{[n]}:=\bigoplus_{a_1+\dots+a_\ell=n}\Hom_\mathscr{C}(x,y)^{a_1,\dots,a_\ell}
$$
coincides with the cohomological grading, for any $x,y\in\Ob\mathscr{C}$.
\end{itemize}
}
\end{defn}

\begin{lemma}
Let $\mathscr{C}$ be a $\mathbb{Z}^\ell$-graded dg category, $\mathscr{C}_1$ its full subcategory. Then $\mathscr{C}_1$ is also $\mathbb{Z}^\ell$-graded.
\end{lemma}
\qed

\begin{prop}\label{propapp1}
Let $\mathscr{C}$ be $\mathbb{Z}^\ell$-graded dg category, $\mathscr{C}_1,\dots,\mathscr{C}_\ell$ its $\ell$ full small dg subcategories. Then the refined Drinfeld dg quotient $\mathscr{C}\text{/}(\mathscr{C}_1,\dots,\mathscr{C}_\ell)$ is naturally $\mathbb{Z}^\ell$-graded dg category.
\end{prop}
\begin{proof}
We define $\mathbb{Z}^\ell$-grading of the new morphisms $\varepsilon_X^{i_1,\dots,i_s}$, $s\le \ell$, by
\begin{equation}
\deg\varepsilon_X^{i_1,\dots,i_s}=(a_1,\dots,a_\ell)
\end{equation}
where
\begin{equation}
a_i=\begin{cases}-1&\text{if $i=i_j$ for some $1\le j\le s$}\\
0&\text{otherwise}
\end{cases}
\end{equation}
Then we use condition (i) of Definition \ref{dgg} the $\mathbb{Z}^\ell$-grading to all morphisms (in a unique way).
Condition (ii) is also fulfilled, because the cohomological degree of $\varepsilon_X^{i_1,\dots,i_s}$ is equal to $-s$.
\end{proof}

\subsection{\sc Monoidal structure}\label{sectionapp4}
There is a monoidal structure
$$
\mathscr{P}_m\Cat^\dg(\k)\times \mathscr{P}_n\Cat^\dg(\k)\to \mathscr{P}_{m+n}\Cat^\dg(\k)
$$
defined as
$$
(\mathscr{C};\ \mathscr{C}_1,\dots,\mathscr{C}_m)\times (\mathscr{D};\ \mathscr{D}_1,\dots,\mathscr{D}_n)\mapsto (\mathscr{C}\otimes\mathscr{D};\ \mathscr{C}_1\otimes\mathscr{D},\dots,\mathscr{C}_m\otimes\mathscr{D},\mathscr{C}\otimes\mathscr{D}_1,\dots,\mathscr{C}\otimes\mathscr{D}_n)
$$
To simplify notations, we use the notation $(\mathscr{C};\underline{C})$ for a tuple 
$(\mathscr{C};\ \mathscr{C}_1,\dots,\mathscr{C}_m)$. Then we rewrite our product as
$$
(\mathscr{C};\underline{\mathscr{C}})\otimes (\mathscr{D};\underline{\mathscr{D}})=(\mathscr{C}\otimes\mathscr{D}; \underline{\mathscr{C}}\otimes\mathscr{D},\mathscr{C}\otimes\underline{\mathscr{D}})
$$
As well, we use the notation $\underline{\underline{\mathscr{C}}}$ for $(\mathscr{C},\underline{\mathscr{C}})$.

The refined Drinfeld dg quotient is colax-monoidal with respect to this structure. More precisely, there is a colax-monoidal map
$$
\beta_{\underline{\underline{\mathscr{C}}},\underline{\underline{\mathscr{D}}}}\colon \Dr(\mathscr{C}\otimes\mathscr{D}; \underline{\mathscr{C}}\otimes\mathscr{D},\mathscr{C}\otimes\underline{\mathscr{D}})\to\Dr(\mathscr{C},\underline{\mathscr{C}})\otimes \Dr(\mathscr{D},\underline{\mathscr{D}})
$$
constructed in [Sh1], Prop. 4.4. For any particular $\underline{\underline{\mathscr{C}}}$ and $\underline{\underline{\mathscr{D}}}$ it is a map of dg categories. Thus, it is uniquely defined by the images of the ``new'' morphisms $\varepsilon_{X\times Y}^{i_1,\dots,i_s,j_1,\dots, j_t}$, where the notation for the upper index assumes that 
$$
X\in \cap_{a=1}^s\mathscr{C}_{i_a},\ \ Y\in\cap_{b=1}^t\mathscr{D}_{j_b}
$$
By our construction in loc.cit., the definition reads
\begin{equation}\label{apgraded}
\beta_{\underline{\underline{\mathscr{C}}},\underline{\underline{\mathscr{D}}}}(\varepsilon_{X\times Y}^{i_1,\dots,i_s,j_1,\dots,j_t})=\varepsilon_X^{i_1,\dots,i_s}\otimes\varepsilon_Y^{j_1,\dots,j_t}
\end{equation}

It implies a crucial for the paper consequence:
\begin{prop}\label{propapp2}
Let $\mathscr{C}$ be a $\mathbb{Z}^{\ell_1}$-graded dg category, $\mathscr{D}$ be a $\mathbb{Z}^{\ell_2}$-graded category, and let $\underline{\mathscr{C}}=(\mathscr{C}_1,\dots,\mathscr{C}_{\ell_1})$, $\underline{\mathscr{D}}=(\mathscr{D}_1,\dots,\mathscr{D}_{\ell_2})$ be tuples with $\ell_1$ and $\ell_2$ small dg subcategories, correspondingly.
Then the dg category $\Dr(\mathscr{C}\otimes\mathscr{D};\ \underline{\mathscr{C}}\otimes\mathscr{D},\mathscr{C}\otimes \underline{\mathscr{D}})$ is $\mathbb{Z}^{\ell_1+\ell_2}$-graded, and
the dg functor $\beta_{\underline{\underline{\mathscr{C}}},\underline{\underline{\mathscr{D}}}}$ agrees with the grading in the following sense: for any $I\in\mathbb{Z}^{\ell_1}$, $J\in \mathbb{Z}^{\ell_2}$, it sends 
$$
\Hom_{\Dr(\mathscr{C}\otimes\mathscr{D};\ \underline{\mathscr{C}}\otimes\mathscr{D},\mathscr{C}\otimes \underline{\mathscr{D}})}(x\times y,x^\prime\times y^\prime)^{(I,J)}
$$
to
$$
\Hom_{\Dr(\mathscr{C},\underline{\mathscr{C}})}(x,x^\prime)^{I}\otimes \Hom_{\Dr(\mathscr{D},\underline{\mathscr{D}})}(y,y^\prime)^{J}
$$
where $(I,J)\in\mathbb{Z}^{\ell_1+\ell_2}$ is obtained by the concatenation of $I$ and $J$.
\end{prop}
\begin{proof}
At first, by the assumption that $\mathscr{C}$ is $\mathbb{Z}^{\ell_1}$-graded and $\mathscr{D}$ is $\mathbb{Z}^{\ell_2}$-graded, we conclude that the tensor product $\mathscr{C}\otimes\mathscr{D}$ is $\mathbb{Z}^{\ell_1+\ell_2}$-graded. Then the dg category 
$\Dr(\mathscr{C}\otimes\mathscr{D};\ \underline{\mathscr{C}}\otimes\mathscr{D},\mathscr{C}\otimes \underline{\mathscr{D}})$ is $\mathbb{Z}^{\ell_1+\ell_2}$-graded, by Proposition \ref{propapp1}. It proves the first claim. For the second claim, it is enough to check it for
$$
\varepsilon_{x\times y}^{i_1,\dots,i_s,j_1,\dots,j_t}\in \Hom_{\Dr(\mathscr{C}\otimes\mathscr{D};\ \underline{\mathscr{C}}\otimes\mathscr{D},\mathscr{C}\otimes \underline{\mathscr{D}})}(x\times y,x\times y)
$$
(as $\beta_{\underline{\underline{\mathscr{C}}},\underline{\underline{\mathscr{D}}}}$ is a map of dg categories), in which case it follows from \eqref{apgraded}.
\end{proof}

\subsection{\sc Applications to Deligne conjecture}\label{sectionapp5}
Let $\mathscr{A}$ be an abelian $\k$-linear monoidal category, such that the monoidal and the abelian structures are {\it weakly compatible}, see [Sh1, Def.5.1]. Denote by $\otimes$ the monoidal product in $\mathscr{A}$ and by $e$ the unit.  Denote by $\mathscr{A}^\dg$ the dg category of bounded from above complexes in $\mathscr{A}$, and by $\mathscr{I}\subset\mathscr{A}^\dg$ the full dg subcategory formed by the acyclic complexes. We say that the exact and the monoidal structures in $\mathscr{A}$ are {\it weakly compatible} if there is a $\k$-linear additive subcategory $\mathscr{A}_0\subset\mathscr{A}$ such that the following properties hold:
\begin{itemize}
\item[(i)] $\mathscr{A}_0$ is monoidal subcategory of $\mathscr{A}$, $e\in\mathscr{A}$ belongs to $\mathscr{A}_0$ and is a monoidal unit in $\mathscr{A}_0$,
\item[(ii)] denote $\mathscr{J}=\mathscr{A}_0^\dg\cap\mathscr{I}$, then $\mathscr{J}$ is a 2-sided ideal in $\mathscr{A}_0^\dg$ with respect to the monoidal product; that is, for an acyclic $I\in\mathscr{J}$ and any $X\in\mathscr{A}_0^\dg$ both complexes $X\otimes I$ and $I\otimes X$ are acyclic (here $\mathscr{A}_0^\dg$ is the full dg subcategory of $\mathscr{A}^\dg$, whose objects are bounded from above complexes in $\mathscr{A}_0$),
\item[(iii)] the natural embedding of pairs of dg categories
$$
(\mathscr{A}_0^\dg,\mathscr{J})\hookrightarrow (\mathscr{A}^\dg,\mathscr{I})
$$
induces a quasi-equivalence of the dg quotients 
$$
\mathscr{A}_0^\dg/\mathscr{J}\tilde{\rightarrow} \mathscr{A}^\dg/\mathscr{I}
$$
\end{itemize}

\hspace{1mm}

We pass to the proof of the Deligne conjecture.

At the first step, we define a Leinster monoid $X_\ldot$ in $\mathscr{P}_*\Cat^\dg(\k)$, with $X_n\in\mathscr{P}_n\Cat^\dg(\k)$, by setting
\begin{equation}
X_n=((\mathscr{A}_0^\dg)^{\otimes n};\ \mathscr{L}_1,\dots,\mathscr{L}_n)
\end{equation}
where 
\begin{equation}
\mathscr{L}_i=(\mathscr{A}_0^\dg)^{\otimes (i-1)}\otimes \mathscr{J}\otimes (\mathscr{A}_0^\dg)^{\otimes (n-i)}
\end{equation}
The face maps are well-defined because $\mathscr{J}$ is an ideal in $\mathscr{A}_0^\dg$. The colax-monoidal maps 
$$
\beta_{a,b}\colon X_{a+b}\to X_a\otimes X_b
$$
are tautological.
\begin{remark}\label{rmapp}
{\rm Note that $X_\ldot$ is a {\it graded} Leinster monoid in $\mathscr{P}_*\Cat^\dg(\k)$.}
\end{remark}

As the next step, we apply the refined Drinfeld dg quotient term-wise, and set 
\begin{equation}
Y_n=\Dr((\mathscr{A}_0^\dg)^{\otimes n};\ \mathscr{L}_1,\dots,\mathscr{L}_n)
\end{equation}
Then $Y_\ldot$ is a Leinster monoid in $\Cat^\dg(\k)$, by the monoidal property of the refined Drinfeld dg quotient, see Section \ref{sectionapp4}.
\begin{remark}{\rm
To avoid set-theoretical troubles, we assumed in [Sh1] that $\mathscr{A}$ is small. (We need the categories $\mathscr{L}_i$ to be small). A more reasonable assumption is that $\mathscr{A}$ is locally-presentable, and then we work with a small model of $\mathscr{A}^\dg$, formed by $\lambda$-perfect objects, for some small ordinal $\lambda$. 
}
\end{remark}
We have the following
\begin{prop}\label{propapp3}
In the notations as above, $Y_\ldot$ is a graded Leinster monoid in $\Cat^\dg(\k)$. 
\end{prop}
\begin{proof}
It follows from Propositions \ref{propapp1} and \ref{propapp2}.
\end{proof}
Finally, we pass from the Leinster monoid $Y_\ldot$ in $\Cat^\dg(\k)$ to a Leinster monoid $Z_\ldot$ in $\Alg(\k)$, as follows.

Each category $X_n$ contains a distinguished object $e_n=e^{\otimes n}$, where $e$ is the monoidal unit.
The objects $\{e_n\}$ are preserved by the action of morphisms in $\Delta_0^\opp$, and by the colax-maps. 
The subsequent application of the refined Drinfeld dg quotient sends $\{e_n\}$ to $\{\bar{e}_n\}$, which are also preserved by the action of morphisms in $\Delta_0^\opp$, and by the colax-maps. Set
$$
Z_n=\Hom_{Y_n}(\bar{e}_n,\bar{e}_n)
$$
Then $Z_\ldot$ is a Leinster monoid in $\Alg(\k)$, see also [Sh1, Section 5.2].
\begin{theorem}\label{propappf}
The Leinster monoid $Z_\ldot$ in $\Alg(\k)$ is graded, see Definition \ref{grla}.
\end{theorem}
\begin{proof}
It follows from Proposition \ref{propapp3} and Lemma \ref{lemmaapp}.
\end{proof}

\section{\sc A proof of Theorem \ref{emmnewtheorem}}\label{proeff}
Below we prove the statements (I) and (III) of the Theorem; the statement (II) is proven directly and is left to the reader. Note that due to Remark 
\ref{emremnew}(3.) we can not reduce the computation, in the case of non-trivial multi-gradings, to the known fact that the classical Eilenberg-MacLane map \eqref{emform} obeys the Leibniz rule. 

{\it Proof of (I):}

We need to prove \eqref{emmnew1}. Let us rewrite \eqref{emmnew1}, implementing the explicit expressions for the differentials:

\begin{equation}\label{emmnewbig}
\begin{aligned}
\ & \sum_{\substack{{(m,n)-}\\{\text{shuffles}\ \sigma}}}\sum_{\ell=1}^{m+n-1}(-1)^{S_\EM(\sigma,x_m,y_n)+S_\sharp(\sigma,x_m,y_n)+S_1(\sigma,\ell)}
F_\ell D_{\sigma(m+n-1)}\dots D_{\sigma(m)}(x_m)\otimes
F_\ell D_{\sigma(m-1)}\dots D_{\sigma(0)}(y_n)=\\
&\sum_{\substack{{(m-1,n)-}\\{\text{shuffles}\ \sigma_1}}}\sum_{q=1}^{m-1}(-1)^{S_\EM(\sigma_1,F_q x_m,y_n)+S_\sharp(\sigma_1,F_qx_m,y_n)+\sum_{u=1}^{q}a_u+q}\\
&D_{\sigma_1(m+n-2)}\dots D_{\sigma_1(m-1)} F_q(x_m)\otimes
 D_{\sigma_1(m-2)}\dots D_{\sigma_1(0)}(y_n)+\\
&\sum_{\substack{{(m,n-1)-}\\{\text{shuffles}\ \sigma_2}}}\sum_{p=1}^{n-1}(-1)^{S_\Sigma(x_m)+S_\EM(\sigma_2,x_m,F_p y_n)+S_\sharp(\sigma_2,x_m,F_py_n)+\sum_{u=1}^{p}b_u+p}\\
&D_{\sigma_2(m+n-2)}\dots D_{\sigma_2(m)}(x_m)\otimes
D_{\sigma_2(m-1)}\dots D_{\sigma_2(0)} F_p(y_n)
\end{aligned}
\end{equation}

where
\begin{equation}\label{emmnews1}
S_1(\sigma,\ell)=\sum_{\substack{{0\le i\le n-1}\\{\sigma(i)<\ell}}}b_{i+1}+\sum_{\substack{{n\le j\le m+n-1}\\{\sigma(j)<\ell}}}a_{j-n+1}+\ell
\end{equation}
and
\begin{equation}\label{emmnew3.2}
S_\Sigma(x_m)=a_1+\dots+a_m+m
\end{equation}
(the expression \eqref{emmnews1} for $S_1(\sigma,\ell)$ follows from the formulas for the bar-differential \eqref{diffbarlb}, \eqref{diffbarbis}, and from Lemma \ref{lemmaqz}(ii).

Fix an $(n,m)$-shuffle $\sigma\in\Sigma_{m+n}$ and $1\le \ell\le m+n-1$ in the l.h.s. of \eqref{emmnewbig}; we want to ``commute'' $F_\ell$ with the compositions of the degeneracy operators, for both factors. There are 3 possible cases:
\begin{itemize}
\item[(case 1):] both degeneracy operators $D_\ell$ and $D_{\ell-1}$ appear in the composition  
$D_{\sigma(m-1)}\circ\dots\circ D_{\sigma(0)}$;
\item[(case 2):] both degeneracy operators $D_\ell$ and $D_{\ell-1}$ appear in the composition  
$D_{\sigma(m+n-1)}\circ\dots\circ D_{\sigma(m)}$;
\item[(case 3):] the degeneracy operators $D_\ell$ and $D_{\ell-1}$ appear in different factors, say $D_\ell$ in the left-hand factor and $D_{\ell-1}$ in the right-hand factor, or vice versa. 
\end{itemize}
We use the third group of relations in \eqref{simplid} to perform the commutation.
The computation is provided below, and is rather long and routine. For convenience of the reader, we sum it up in the following lemma:
\begin{lemma}\label{lemma44proof}
In the notations as above, one has:
\begin{itemize}
\item[(i)] the sum of the summands of type (case 1), for all $(m,n)$-shuffles $\sigma$ and $1\le \ell\le m+n-1$, is equal to the second-line summand in the r.h.s. of \eqref{emmnewbig};
\item[(ii)] the sum of the summands of type (case 2), for all $(m,n)$-shuffles $\sigma$ and $1\le\ell\le m+n-1$, is equal to the first-line summand in the r.h.s.of \eqref{emmnewbig};
\item[(iii)] the summands of type (case 3), for all $(m,n)$-shuffles $\sigma$ and $1\le\ell\le m+n-1$, are mutually cancelled.
\end{itemize}
\end{lemma}
\begin{proof}
We prove (ii) and (iii), the proof of (i) is analogous to (ii) and is left to the reader.

{\it Proof of (ii)}: a summand of type (case2) looks like:
\begin{equation}\label{equk1}
(-1)^S \Bigl[F_\ell\circ D_{\sigma(m+n-1)}\circ\dots\circ D_{\sigma(t+1)}\circ D_\ell\circ D_{\ell-1}\circ D_{\sigma(t-2)}\circ\dots D_{\sigma(m)}(x_m)\Bigr]\ \otimes \ \Bigl[F_\ell\circ D_{\sigma(m-1)}\circ\dots\circ D_{\sigma(0)}(y_n)\Bigr]
\end{equation}
where 
\begin{equation}
S=S_\EM(\sigma,x_m,y_n)+S_\sharp(\sigma,x_m,y_n)+S_1(\sigma,\ell)
\end{equation}
Let
\begin{equation}\label{equk1bis}
J=\min_{\substack{{0\le i\le m-1}\\
{\sigma(i)>\ell}}}i
\end{equation}
Then \eqref{equk1} is equal to
\begin{equation}\label{equk2}
\begin{aligned}
\ &(-1)^S \Bigl[D_{\sigma(m+n-1)-1}\circ\dots\circ D_{\sigma(t+1)-1}\circ D_{\ell-1}\circ D_{\sigma(t-2)}\circ\dots D_{\sigma(m)}(x_m)\Bigr]\ \otimes \\ &\Bigl[D_{\sigma(m-1)-1}\circ\dots\circ D_{\sigma(J)-1}\circ D_{\sigma(J-1)}\circ \dots\circ D_{\sigma(0)}\circ F_{\ell-J}(y_n)\Bigr]
\end{aligned}
\end{equation}
It easily follows from \eqref{simplid}.

We see that the summand \eqref{equk1} is equal, up to a sign, to the summand in the second line of the r.h.s. of \eqref{emmnewbig}, corresponded to $p=\ell-J$ and to $\sigma_2$ defined as
\begin{equation}\label{sigma2}
\sigma_2(i)=\begin{cases}\sigma(i)&0\le i\le J-1\\
\sigma(i)-1&J\le i\le m-1\\
\sigma(i)&m\le i\le t-1\\
\sigma(i+1)-1&t\le i\le m+n-2
\end{cases}
\end{equation}
When $\sigma$ is an $(m,n)$-shuffle, $\sigma_2$ is a $(m,n-1)$-shuffle. By definition, $\sigma(t-1)=\ell-1$, and $J$ depends on $\ell$ by \eqref{equk1bis}. Note also that
\begin{equation}\label{xyzproef1}
t-m+J=\ell
\end{equation}

We need to check that the signs also are equal. 

\hspace{1mm}

The sign for the left-hand side summand is
\begin{equation}\label{lhssign}
\begin{aligned}
\ &S_\lhs=S_\EM(\sigma,x_m,y_n)+S_\sharp(\sigma,x_m,y_n)+S_1(\sigma,\ell)=\\
&\sum_{\substack{{0\le i\le m-1}\\{m\le j\le m+n-1}\\{\sigma(i)>\sigma(j)}}}(a_{i+1}+1)(b_{j-m+1}+1)+
\sum_{\substack{{0\le i\le m-1}\\{m\le j\le m+n-1}\\{\sigma(i)>\sigma(j)}}}a_{i+1}b_{j-m+1}+\\
&\sum_{\substack{{0\le i\le m-1}\\{\sigma(i)<\ell}}}a_{i+1}+\sum_{\substack{{m\le j\le m+n-1}\\{\sigma(j)<\ell}}}b_{j-m+1}+\ell
\end{aligned}
\end{equation}

The sign for the right-hand side 
summand
\begin{equation}\label{emmextra}
\begin{aligned}
\ &(-1)^{S_\Sigma(x_m)+S_\EM(\sigma_2,x_m,F_p y_n)+S_\sharp(\sigma_2,x_m,F_py_n)+\sum_{u=1}^{p}b_u+p}\\
&D_{\sigma_2(m+n-2)}\dots D_{\sigma_2(m)}(x_m)\ \otimes\ 
D_{\sigma_2(m-1)}\dots D_{\sigma_2(0)}\circ F_p(y_n)
\end{aligned}
\end{equation}
is 
\begin{equation}\label{rhssign}
S_\rhs=S_\EM(\sigma_2,x_m,F_py_n)+S_\sharp(\sigma_2,x_m,F_py_n)+a_1+\dots+a_m+m+b_1+\dots+b_{p}+p
\end{equation}
where $\sigma_2$ is given by \eqref{sigma2}, and 
\begin{equation}\label{xyzproef2}
p=\ell-J=t-m
\end{equation}
by \eqref{xyzproef1}.

One has:
$$
F_py_n=F_{\ell-J}y_n\in Y_{n-1}^{b_1,\dots,b_{\ell-J-1},b_{\ell-J}+b_{\ell-J+1},b_{\ell-J+2},\dots,b_n}
$$

Therefore,
\begin{equation}\label{emmvie2}
F_py_n=F_{t-m}y_n\in Y_{n-1}^{b_1,\dots, b_{t-m-1},b_{t-m}+b_{t-m+1},\dots,b_n}
\end{equation}
\begin{lemma}\label{lemmalc}
In the notations as above, the following statements are true:
\begin{itemize}
\item[(A)]
\begin{equation}\label{fc1}
S_\EM(\sigma,x_m,y_n)-S_\EM(\sigma_2,x_m,F_py_n)=a_{J+1}+\dots+a_{m}+m-J
\end{equation}
\item[(B)]
\begin{equation}
S_\sharp(\sigma_2,x_m,F_py_n)=S_\sharp(\sigma,x_m,y_n)
\end{equation}
\end{itemize}
\end{lemma}
\begin{proof}
(A): 
For $x_m^\prime\in X_m^{a_1^\prime,\dots,a_m^\prime}$ and $y_{n-1}^\prime\in Y_{n-1}^{b_1^\prime,\dots,b_{n-1}^\prime}$, one has:
\begin{equation}\label{eqgo1}
\begin{aligned}
\ &S_\EM(\sigma_2,x_m^\prime,y_{n-1}^\prime)=
\sum_{\substack{{0\le i\le J-1}\\{m\le j\le t-1}\\{\sigma(i)>\sigma(j)}}}(a_{i+1}^\prime+1)(b_{j-m+1}^\prime+1)+\\
&\sum_{\substack{{J\le i\le m-1}\\{t+1\le j\le m+n-1}\\{\sigma(i)>\sigma(j)}}}(a_{i+1}^\prime+1)(b_{j-m}^\prime+1)+
\sum_{\substack{{J\le i\le m-1}\\{m\le j\le t-1}}}(a_{i+1}^\prime+1)(b_{j-m+1}^\prime+1)
\end{aligned}
\end{equation}
We can ignore the terms of the first summand corresponded to $j=t-1$, as for any $0\le i\le J-1$ one has $\sigma(i)<\sigma(t-1)=\ell-1$. Then
\begin{equation}\label{eqgo1bis}
\begin{aligned}
\ &S_\EM(\sigma_2,x_m^\prime,y_{n-1}^\prime)=
\sum_{\substack{{0\le i\le J-1}\\{m\le j\le t-2}\\{\sigma(i)>\sigma(j)}}}(a_{i+1}^\prime+1)(b_{j-m+1}^\prime+1)+\\
&\sum_{\substack{{J\le i\le m-1}\\{t+1\le j\le m+n-1}\\{\sigma(i)>\sigma(j)}}}(a_{i+1}^\prime+1)(b_{j-m}^\prime+1)+
\sum_{\substack{{J\le i\le m-1}\\{m\le j\le t-1}}}(a_{i+1}^\prime+1)(b_{j-m+1}^\prime+1)
\end{aligned}
\end{equation}
so that
\begin{equation}\label{eqgo2}
\begin{aligned}
\ &S_\EM(\sigma_2,x_m,F_py_n)=
\sum_{\substack{{0\le i\le J-1}\\{m\le j\le t-2}\\{\sigma(i)>\sigma(j)}}}(a_{i+1}+1)(b_{j-m+1}+1)+
\sum_{\substack{{J\le i\le m-1}\\{t+1\le j\le m+n-1}\\{\sigma(i)>\sigma(j)}}}(a_{i+1}+1)(b_{j-m+1}+1)+\\
&\sum_{\substack{{J\le i\le m-1}\\{m\le j\le t-2}}}(a_{i+1}+1)(b_{j-m+1}+1)+
\sum_{\substack{{J\le i\le m-1}}}(a_{i+1}+1)\underline{(b_{t-m}+b_{t-m+1}+1)}
\end{aligned}
\end{equation}
On the other hand, we can rewrite
\begin{equation}\label{eqgo3}
\begin{aligned}
\ &S_\EM(\sigma,x_m,y_n)=\sum_{\substack{{0\le i\le J-1}\\{m\le j\le t}\\{\sigma(i)>\sigma(j)}}}(a_{i+1}+1)(b_{j-m+1}+1)+\\
&\sum_{\substack{{J\le i\le m-1}\\{t+1\le j\le m+n-1}\\{\sigma(i)>\sigma(j)}}}(a_{i+1}+1)(b_{j-m+1}+1)+
\sum_{\substack{{J\le i\le m-1}\\{m\le j\le t}}}(a_{i+1}+1)(b_{j-m+1}+1)
\end{aligned}
\end{equation}
We can ignore the terms of the first summand, corresponded to $j=t-1$ and to $j=t$, as $\sigma(t-1)=\ell-1$, $\sigma(t)=\ell$, and the inequality $\sigma(i)>\sigma(j)$ is impossible for $0\le i\le J-1$. Then we see that \eqref{eqgo3} is equal to 
\begin{equation}\label{eqgo3}
\begin{aligned}
\ &S_\EM(\sigma,x_m,y_n)=\sum_{\substack{{0\le i\le J-1}\\{m\le j\le t-2}\\{\sigma(i)>\sigma(j)}}}(a_{i+1}+1)(b_{j-m+1}+1)+
\sum_{\substack{{J\le i\le m-1}\\{t+1\le j\le m+n-1}\\{\sigma(i)>\sigma(j)}}}(a_{i+1}+1)(b_{j-m+1}+1)+\\
&\sum_{\substack{{J\le i\le m-1}\\{m\le j\le t-2}}}(a_{i+1}+1)(b_{j-m+1}+1)+
\sum_{\substack{J\le i\le m-1}}(a_{i+1}+1)\underline{(b_{t-m}+b_{t-m+1}+2)}
\end{aligned}
\end{equation}

The difference between the r.h.s. \eqref{eqgo2} and \eqref{eqgo3} comes from the underlined terms, what gives the result.

(B): is analogous.
\end{proof}

By \eqref{lhssign}, \eqref{rhssign}, and Lemma \ref{lemmalc}, we get (all equalities below till the end of the proof should be understood as equalities $\mod 2$):
\begin{equation}\label{fc2}
\begin{aligned}
\ &S_\lhs-S_\rhs=\\
&\Bigl(a_{J+1}+\dots+a_{m}+m-J+
\sum_{\substack{{0\le i\le m-1}\\{\sigma(i)<\ell}}}a_{i+1}+\sum_{\substack{{m\le j\le m+n-1}\\{\sigma(j)<\ell}}}b_{j-m+1}+\ell\Bigr)-\\
&\Bigl(a_1+\dots+a_m+m+b_1+\dots+b_{t-m}+t-m\Bigr)
\end{aligned}
\end{equation}

One can rewrite \eqref{fc2} as
\begin{equation}\label{fc3}
\begin{aligned}
\ &S_\lhs-S_\rhs=\\
&\Bigl((a_{J+1}+\dots+a_{m})+(a_1+\dots+a_{J})+(b_1+\dots+b_{t-m})\Bigr)+\\
& \Bigl(a_1+\dots+a_m+m+b_1+\dots+b_{t-m}\Bigr)\\
&(m-J+\ell+t-m)
\end{aligned}
\end{equation}

As $\ell-J=t-m$ by \eqref{xyzproef1}, we see that
$$S_\lhs-S_\rhs=0\mod 2$$

We have proved that a summand \eqref{equk1} of type (case2) and the corresponding summand 
\eqref{emmextra}
in the third line of \eqref{emmnewbig} {\it have equal signs}. It remains to show that there is a 1-to-1 correspondence between the summands themselves.

The first type data depends on an $(m,n)$-shuffle $\sigma$ and $1\le \ell\le m+n-1$; the second type data 
depends on an $(m,n-1)$-shuffle $\sigma_2$, and $1\le p\le n$. In the computation above, the pair $(\sigma_2,p)$ is constructed by the pair $(\sigma,\ell)$; we need to reconstruct $(\sigma,\ell)$ by $(\sigma_2,p)$. But $\sigma$ can be reconstructed by $\sigma_2$, if we know $t$, see \eqref{equk2}, \eqref{sigma2}. Then $t-m=p$, see \eqref{xyzproef1}, and finally $\ell=\sigma(t)$.

It completes the proof of the claim (ii) of Lemma \ref{lemma44proof}.

\hspace{5mm}

{\it Proof of (iii)}:

Let $\sigma$ be an $(m,n)$-shuffle, such that for some $0\le \ell\le m+n-1$ the pair $(\sigma,\ell)$ is of type (case3), see above.
Define a permutation $\sigma_1\in \Sigma_{m+n}$ as the composition 
$$
\sigma_1=(\ell-1,\ell)\circ \sigma
$$
(we denote by $(a,b)$ the transposition of $a$ and $b$). Then $\sigma_1$ is again an $(m,n)$-shuffle, and the pair $(\sigma_1,\ell)$ is again of type (case3). 

That is, the terms of type (case3) appear in pairs.

Denote by $S_\mathrm{tot}(\sigma,\ell,x_m,y_n)$ the exponent of -1 in the first line of \eqref{emmnewbig}. We claim that
\begin{equation}\label{caseiii1}
S_\mathrm{tot}(\sigma,\ell,x_m,y_n)+S_\mathrm{tot}(\sigma_1,\ell,x_m,y_n)=1 \mod 2
\end{equation}
As the corresponding to $(\sigma,\ell)$ and $(\sigma_1,\ell)$ summands corresponded to the l.h.s. of \eqref{emmnewbig} differ only by a sign, \eqref{caseiii1} implies Lemma \ref{lemma44proof}(iii).

We have:
\begin{equation}
S_\mathrm{tot}(\sigma,\ell,x_m,y_n)=S(\sigma,x_m,y_n)+S_1(\sigma,\ell)+S_\sharp(\sigma,x_m,y_n)
\end{equation}
So it is enough to prove that
\begin{equation}
[S(\sigma,x_m,y_n)+S_1(\sigma,\ell)+S_\sharp(\sigma,x_m,y_n)]+[S(\sigma_1,x_m,y_n)+S_1(\sigma_1,\ell)+S_\sharp(\sigma_1,x_m,y_n)]=1\mod 2
\end{equation}
Let $\ell-1=\sigma(s)$, $\ell=\sigma(t)$; we may assume that $0\le s\le m-1$, $m\le t\le m+n-1$ (otherwise, replace $\sigma$ by $\sigma_1$).

We have:
\begin{equation}
S(\sigma,x_m,y_n)+S(\sigma_1,x_m,y_n)=(a_{s+1}+1)(b_{t+1}+1)\mod 2
\end{equation}
Also
\begin{equation}
S_\sharp(\sigma,x_m,y_n)+S_\sharp(\sigma_1,x_m,y_n)=a_{s+1}b_{t+1}\mod 2
\end{equation}
and
\begin{equation}
S_1(\sigma,\ell)+S_1(\sigma_1,\ell)=a_{s+1}+b_{t+1}\mod 2
\end{equation}
We are done.
\end{proof}

{\it Proof of Theorem \ref{emmnewtheorem}(III):}
Denote by $\tau_{m,n}\in\Sigma_{m+n}$ the permutation exchanging the first $m$ with the last $n$ entries,
$$
\tau(i)=\begin{cases}i+n&0\le i\le m-1\\
i-m&m\le i\le m+n-1
\end{cases}
$$
For $\sigma$ is an $(m,n)$-shuffle, the composition $\sigma\circ \tau_{n,m}$ is an $(n,m)$-shuffle.
Let $x_m\in X_m^{a_1,\dots,a_m}, y_n\in Y_n^{b_1,\dots,b_n}$. The correspondence $\sigma\mapsto \sigma\circ \tau_{n,m}$ provides a 1-to-1 correspondence between the terms of $\nabla^\Bar(x_m,y_n)$ and the terms of $\nabla^\Bar(y_n,x_m)$. However, the signs with the corresponding terms are different. We compare, for a fixed $(m,n)$-shuffle $\sigma$ and for $0\le\ell\le m+n-1$, 
$$
S_\tot(\sigma,\ell,x_m,y_n)=S_\EM(\sigma,x_m,y_n)+S_\sharp(\sigma,x_m,y_n)+S_1(\sigma,\ell, x_m,y_n)
$$
with
$$
S_\tot(\sigma\circ \tau_{n,m},\ell,y_n,x_m)=S_\EM(\sigma\circ\tau,y_n,x_m)+S_\sharp(\sigma\circ\tau,y_n,x_m)+S_1(\sigma\circ\tau,\ell,y_n,x_m)
$$
One needs to prove that the difference between these integer numbers is $\mod 2$ equal to the ``sign correction'' $\kappa_1-\kappa_2$ in the diagram \eqref{symmkappas}:
\begin{equation}\label{eqkappas1}
S_\tot(\sigma,\ell)-S_\tot(\sigma\circ\tau_{n,m},\ell)=(a_1+\dots+a_m+m)(b_1+\dots+b_n+n)-(a_1+\dots+a_m)(b_1+\dots+b_n) \mod 2
\end{equation}

One has:
$$
S_1(\sigma,\ell,x_m,y_n)=S_1(\sigma\circ\tau_{n,m},\ell,y_n,x_m)
$$
and thus \eqref{eqkappas1} follows from the two equalities below, which are straightforward:
\begin{equation}\label{eqkappas2}
S_\EM(\sigma,x_m,y_n)-S_\EM(\sigma\circ\tau_{n,m},y_n,x_m)=(a_1+\dots+a_m+m)(b_1+\dots+b_n+n)\mod 2
\end{equation}
and
\begin{equation}\label{eqkappas3}
S_\sharp(\sigma,x_m,y_n)-S_\sharp(\sigma\circ\tau,y_n,x_m)=(a_1+\dots+a_m)(b_1+\dots+b_n)\mod 2
\end{equation}

\hspace{1mm}

Theorem \ref{emmnewtheorem} is proven.

\qed

\section{\sc Proof of Theorem \ref{theoremcoalg}}\label{proofcoalg}
We already know from Theorem \ref{emmnewtheorem}(I) that $\nabla^\Bar$ is a map of complexes. Thus we only need to prove that is agrees with the coproduct \eqref{barcoalg}.

We need to prove that for any $x_m\in X_m^{a_1,\dots,a_m}$ and $y_n\in Y_n^{b_1,\dots,b_n}$ one has
\begin{equation}\label{comp1}
(\nabla^\Bar\otimes\nabla^\Bar)(\Delta(x_m),\Delta(y_n))=\Delta(\nabla^\Bar(x_m,y_n))
\end{equation}
where, as in \eqref{barcoalg},
\begin{equation}\label{comp2}
\Delta(x_m)=\oplus_{m_1+m_2=m}\beta_{m_1,m_2}(x_m)
\end{equation}
and
\begin{equation}\label{comp3}
\Delta(y_n)=\oplus_{n_1+n_2=n}\beta_{n_1,n_2}(y_n)
\end{equation}
Note that
\begin{equation}\label{comp2bis}
\beta_{m_1,m_2}(x_m)\in X_{m_1}^{a_1,\dots,a_{m_1}}\otimes X_{m_2}^{a_{m_1+1},\dots,a_m}
\end{equation}
and similarly
\begin{equation}\label{comp3bis}
\beta_{n_1,n_2}(y_n)\in Y_{n_1}^{b_1,\dots,b_{n_1}}\otimes Y_{n_2}^{b_{n_1+1},\dots,b_n}
\end{equation}
For the right-hand side of \eqref{comp1}, one has:
\begin{equation}\label{comp4}
\Delta(\nabla^\Bar(x_m,y_n))=\sum_{a+b=m+n}(-1)^{S_\beta(\sigma,a,b,x_m,y_n)}\beta_{a,b}(\nabla_1^\Bar(x_m))\otimes \beta_{a,b}(\nabla_2^\Bar(y_n))
\end{equation}
with the ``Sweedler-like notation'' $\nabla^\Bar(x_m,y_n)= \nabla_1^\Bar(x_m)\otimes \nabla_2^\Bar(y_n)\in\Bar(X_L\otimes Y_L)$ (which assumes a sum of such monomials), where $S_\beta(\sigma,a,b,x_m,y_n)$ comes from \eqref{grmon4}.

Take any $a,b$ such that $a+b=m+n$ and any particular $(n,m)$-shuffle $\sigma$ (which gives a summand in $\nabla(x,y)$, see \eqref{emform}).
The corresponding summand in the r.h.s. of \eqref{comp4} is:
\begin{equation}\label{comp5}
(-1)^{S}\beta_{a,b}\Bigl(D_{\sigma({m+n-1})}\circ\dots\circ D_{\sigma(m)}(x_m)\Bigr)\otimes\beta_{a,b}\Bigl(D_{\sigma(m-1)}\circ\dots\circ D_{\sigma(0)}(y_n)\Bigr)
\end{equation}
where 
\begin{equation}\label{znaks1}
S=S_\EM(\sigma,x_m,y_n)+S_\sharp(\sigma,x_m,y_n)+S_\beta(\sigma,a,b,x_m,y_n)
\end{equation}
where the first two summands come from \eqref{emmnew2}, and $S_\beta(\sigma,a,b,x_m,y_n)$ comes from \eqref{grmon4}.

Now we make use the fact that $\beta$ {\it is a map of bifunctors} $\Delta_0^\opp\times \Delta_0^\opp\to \Vect(\k)$, to move  commute the simplicial degeneracy operators  $D_i$ with $\beta_{a,b}$. It follows from this functoriality that
\begin{equation}\label{grootmak}
\beta_{a,b}(D_sz)=\begin{cases}
(D_{s}\otimes\id)\beta_{a-1,b}(z)&0\le s\le a-1\\
(\id\otimes D_{s-a})\beta_{a,b-1}(z)&a\le s\le a+b-1
\end{cases}
\end{equation}
for $z\in Z_{a+b-1}$.

Denote
\begin{equation}\label{comp6}
\begin{aligned}
\ &m_1=\sharp\{\ell\in[0,1,\dots,m-1]|\sigma(\ell)\le a-1\}\\       
&n_1=\sharp\{\ell\in[m,m+1,\dots,m+n-1]|\sigma(\ell)\le a-1\}\\
\end{aligned}
\end{equation}
then
\begin{equation}\label{comp7}
m_1+n_1=a,\ m_2=m-m_1,\ n_2=n-n_1,\ m_2+n_2=b
\end{equation}

Then \eqref{grootmak} gives:
\begin{equation}\label{comp8}
\begin{aligned}
\ &\beta_{a,b}\Bigl(D_{\sigma({m-1})}\dots D_{\sigma(0)}(x_m)\Bigr)=\\
&\Bigl( ((D_{\sigma(m_1-1)} \dots D_{\sigma(0)})\otimes\id)(\beta_{m_2,m_1}(x_m)_1)\Bigr)\otimes \Bigl(\id\otimes(D_{\sigma(m-1)-a}\dots D_{\sigma(m_1)-a})(\beta_{m_2,m_1}(x_m)_2)\Bigr)
\end{aligned}
\end{equation}
where we use the ``Sweedler notation'' $\beta_{m_2,m_1}(x_m)=(\beta_{m_2,m_1}(x_m))_1\otimes (\beta_{m_2,m_1}(x_m))_2$, with $$\beta_{m_2,m_1}(x_m)_1\in X_{m_2}^{a_1,\dots,a_{m_2}}\text{   and   } \beta_{m_2,m_1}(x_m)_2\in X_{m_1}^{a_{m_2+1},\dots,a_m}$$ 

Similarly
\begin{equation}\label{comp9}
\begin{aligned}
\ &\beta_{a,b}\Bigl(D_{\sigma(m+n-1)}\circ\dots\circ D_{\sigma(m)}(y_n)\Bigr)=\\
&\Bigl(((D_{\sigma(m+n_1)}\circ  \dots\circ D_{\sigma(m)})\otimes\id)(\beta_{n_2,n_1}(y_n)_1) \Bigr)\otimes \Bigl((\id\otimes (D_{\sigma(m+n-1)-a}\circ  \dots\circ D_{\sigma(m+n_1)-a}))(\beta_{n_2,n_1}(y_n)_2) \Bigr)
\end{aligned}
\end{equation}
with
$$
\beta_{n_2,n_1}(y_n)_1\in Y_{n_2}^{b_1,\dots,b_{n_2}}\text{   and   }\beta_{n_2,n_1}(y_n)_2\in Y_{n_2}^{b_{n_2+1},\dots,b_n}
$$
Then \eqref{grmon4} gives for the third summand in \eqref{znaks1}:
\begin{equation}
S_\beta(\sigma,a,b,x_m,y_n)=(a_{m_2+1}+\dots+a_{m})(b_{1}+\dots+b_{n_1})
\end{equation}

The summand in the l.h.s. of \eqref{comp1}, corresponded to  $\sigma,m,n,a,b$ for our summand in the r.h.s. of \eqref{comp1}, is
\begin{equation}
(\nabla^\Bar_{\sigma_2}\otimes \nabla^\Bar_{\sigma_1})(\beta_{m_2,m_1}(x_m),\beta_{n_2,n_1}(y_n))
\end{equation}
(we denote by $\nabla^\Bar_\sigma(x_m,y_n)$ the summand corresponded to a $(m,n)$-shuffle $\sigma$), 
where
\begin{equation}\label{ver1}
\sigma_1(i)=\begin{cases}
\sigma(i)&0\le i\le m_1-1\\
\sigma(i+m_2)&m_1\le i\le m_1+n_1-1
\end{cases}
\end{equation}
\begin{equation}\label{ver2}
\sigma_2(j)= \begin{cases}\sigma(j+m_1)-a  &0\le j\le m_2-1\\
\sigma(j+m_1+n_1)-a       &m_2\le j\le m_2+n_2-1
\end{cases}
\end{equation}
It is clear that $\sigma_1$ is an $(m_1,n_1)$-shuffle, and $\sigma_2$ is an $(m_2,n_2)$-shuffle.

We get a correspondence, assigning to a tuple 
$$(\sigma,m,n,a,b)$$ 
where $\sigma$ is an $(n,m)$-shuffle, and $a,b$ are non-negative integers such that $a+b=m+n$, a tuple
$$
(\sigma_1,\sigma_2,m_1,m_2,n_1,n_2)
$$
where $m_1+m_2=m$, $n_1+n_2=n$, $m_1+n_1=a$, $m_2+n_2=b$, $\sigma_1$ is $(n_1,m_1)$-shuffle, and $\sigma_2$ is $(n_2,m_2)$-shuffle. 

It is clear that this correspondence is 1-to-1, for fixed $m$ and $n$.

It remains to prove that the signs with the corresponding terms in the l.h.s. and the r.h.s. of \eqref{comp1} are equal.

In Sweedler notations, one has:
\begin{equation}\label{compx1}
\begin{aligned}
\ &(-1)^{S_1(\sigma,m_1,m_2,n_1,n_2,x_m,y_n)}(\nabla^\Bar_{\sigma_2}\otimes \nabla^\Bar_{\sigma_1})(\beta_{m_2,m_1}(x_m),\beta_{n_2,n_1}(y_n))=\\
&\nabla^{\Bar}_{\sigma_2}(\beta_{m_2,m_1}(x_m)_1,\beta_{n_2,n_1}(y_n)_1)\otimes
\nabla^{\Bar}_{\sigma_2}(\beta_{m_2,m_1}(x_m)_2,\beta_{n_2,n_1}(y_n)_2)
\end{aligned}
\end{equation}
where $S_1(\sigma,m_1,m_2,n_1,n_2,x_m,y_n)$ gives the Koszul sign of commutation of $\beta_{n_2,n_1}(y_n)_1$ and $\beta_{m_2,m_1}(x_m)_2$ in $\Bar(X_L\otimes Y_L)$. That is,
\begin{equation}
S_1(\sigma,m_1,m_2,n_1,n_2,x_m,y_n)=(a_{m_2+1}+\dots+a_m+m_1)(b_1+\dots+b_{n_2}+n_2)
\end{equation}
The sign $(-1)^{S_1(\sigma,m_1,m_2,n_1,n_2,x_m,y_n)}$ appears when we pass from the l.h.s. of \eqref{comp1} to the suitable for computations form of the r.h.s. of \eqref{compx1}.

Finally, Theorem \ref{theoremcoalg} is reduced to the following identity $\mod 2$:
\begin{equation}\label{compx2}
\begin{aligned}
\ &S_\EM(\sigma,x_m,y_n)+S_\sharp(\sigma,x_m,y_n)+S_\beta(\sigma,a,b,x_m,y_n)+S_1(\sigma,m_1,m_2,n_1,n_2,x_m,y_n)=\\
&S_\EM(\sigma_2,\beta_{m_2,m_1}(x_m)_1,\beta_{m_2,m_1}(y_n)_1)+S_\sharp(\sigma_2,\beta_{m_2,m_1}(x_m)_1,\beta_{m_2,m_1}(y_n)_1)+\\
&S_\EM(\sigma_1,\beta_{m_2,m_1}(x_m)_2,\beta_{m_2,m_1}(y_n)_2)+S_\sharp(\sigma_1,\beta_{m_2,m_1}(x_m)_2,\beta_{m_2,m_1}(y_n)_2)
\mod 2
\end{aligned}
\end{equation}
More directly, \eqref{compx2} is equivalent to:
\begin{equation}\label{compx3}
\begin{aligned}
\ &\sum_{\substack{{0\le i\le m-1}\\{m\le j\le m+n-1}\\{\sigma(i)>\sigma(j)}}}(a_{i+1}+1)(b_{j-m+1}+1)+
\sum_{\substack{{0\le i\le m-1}\\{m\le j\le m+n-1}\\{\sigma(i)>\sigma(j)}}}a_{i+1}b_{j-m+1}+\\
&(a_{m_2+1}+\dots+a_{m})(b_{1}+\dots+b_{n_2})+(a_{m_2+1}+\dots+a_m+m_1)(b_1+\dots+b_{n_2}+n_2)+\\
&\sum_{\substack{{0\le i\le m_1-1}\\{0\le j\le n_1-1}\\{\sigma_1(i)>\sigma_1(j+m)}}}(a_{i+1}+1)(b_{j+1}+1)+
\sum_{\substack{{0\le i\le m_1-1}\\{0\le j\le n_1-1}\\{\sigma_1(i)>\sigma_1(j+m)}}}a_{i+1}b_{j+1}+\\
&\sum_{\substack{{0\le i\le m_2-1}\\{0\le j\le n_2-1}\\{\sigma_2(i+m_1)>\sigma_2(j+m+n_1)}}}(a_{i+m_1+1}+1)(b_{j+n_1+1}+1)+
\sum_{\substack{{0\le i\le m_2-1}\\{0\le j\le n_2-1}\\{\sigma_2(i+m_1)>\sigma_2(j+m+n_1)}}}a_{i+m_1+1}b_{j+n_1+1}=0\mod 2
\end{aligned}
\end{equation}
It follows from \eqref{comp6} and \eqref{ver1}, \eqref{ver2} that
\begin{equation}\label{compx4}
\begin{aligned}
\ &\sum_{\substack{{0\le i\le m-1}\\{m\le j\le m+n-1}\\{\sigma(i)>\sigma(j)}}}(a_{i+1}+1)(b_{j-m+1}+1)=\\
&\sum_{\substack{{0\le i\le m_1-1}\\{0\le j\le n_1-1}\\{\sigma_1(i)>\sigma_1(j+m)}}}(a_{i+1}+1)(b_{j+1}+1)+
\sum_{\substack{{0\le i\le m_2-1}\\{0\le j\le n_2-1}\\{\sigma_2(i+m_1)>\sigma_2(j+m+n_1)}}}(a_{i+m_1+1}+1)(b_{j+n_1+1}+1)+\\
&(a_{m_2+1}+\dots+a_m+m_1)(b_1+\dots+b_{n_2}+n_2)
\end{aligned}
\end{equation}
and
\begin{equation}\label{compx5}
\begin{aligned}
\ &\sum_{\substack{{0\le i\le m-1}\\{m\le j\le m+n-1}\\{\sigma(i)>\sigma(j)}}}a_{i+1}b_{j-m+1}=
\sum_{\substack{{0\le i\le m_1-1}\\{0\le j\le n_1-1}\\{\sigma_1(i)>\sigma_1(j+m)}}}a_{i+1}b_{j+1}+
\sum_{\substack{{0\le i\le m_2-1}\\{0\le j\le n_2-1}\\{\sigma_2(i+m_1)>\sigma_2(j+m+n_1)}}}a_{i+m_1+1}b_{j+n_1+1}+\\
&(a_{m_2+1}+\dots+a_m)(b_1+\dots+b_{n_2})
\end{aligned}
\end{equation}
Now we see that \eqref{compx3} is $\mod 2$ sum of \eqref{compx4} and \eqref{compx5}, and \eqref{compx3} follows.

\qed

\section{\sc A proof of Theorem \ref{mainyon}}\label{proefmy}
\subsection{\sc The signs in $A_\infty$ algebras and $A_\infty$ morphisms}
To make the discussion on signs below more transparent, we start with recalling some basic principles on signs in $A_\infty$-algebra like formulas.
It is well-known that an $A_\infty$ algebra structure on a graded vector space $V$ can be defined as a coderivation $D$ of degree +1 of the cofree coalgebra $T^\vee(V[1])$ cogenerated by the shifted by 1 vector space $V$, such that $D^2=0$.

Note that $V$ is considered as the object of the category of $\mathbb{Z}$-graded {\it complexes}, we denote by $\Vect^\mb_d(\k)$. The Koszul sign braiding is given by 
$$
\beta(v_k\otimes w_\ell)=(-1)^{k\ell}w_\ell\otimes v_k
$$
for $v_k\in V^k, w_\ell\in W^\ell$. The differential on $V\otimes W$ is given by the Leibniz rule:
\begin{equation}\label{leibniz}
d(v_k\otimes w_\ell)=(dv_k)\otimes w_\ell+(-1)^kv_k\otimes (dw_\ell)
\end{equation}

Note also that the functor
$$
T\colon V\mapsto V[1]
$$
is {\it not} monoidal. As a result, it affects the Leibniz rule for computations with differentials, when we deal with multi-linear cochains.

More precisely, consider a cochain 
$$
\Psi\colon V_1\otimes\dots\otimes V_k\to W
$$
It defines $\Psi^T$ by
$$
(\Psi^T)(v_1,\dots,v_k)=T((\Psi)(T^{-1}v_1,\dots,T^{-1}(v_k)))
$$
with $v_k\in V_k$.

We want the commutativity with the differential hold:
\begin{equation}\label{groot1}
d(\Psi^T)\overset{?}{=}(d\Psi)^T
\end{equation}
which {\it fails}.

Note that 
\begin{equation}\label{groot2}
(d\Psi)(v_1,\dots,v_k):=-\Psi(dv_1,v_2,\dots,v_k)-(-1)^{\deg v_1}\Psi(v_1,dv_2,\dots)-\dots+(-1)^{\deg \Psi}(d(\Psi(v_1,\dots,v_k)
\end{equation}

Note that in \eqref{groot1}, the degrees $\deg v_i$ are different for the l.h.s. and the r.h.s., due to the shift functor $T$. Namely, the original degrees $\deg v_i$ are used in the computation $d\Psi$ in the r.h.s. of \eqref{groot1}, whence they are replaced by $\deg v_i+1$ for the computation of $d(\Psi^T)$ in the l.h.s. of \eqref{groot1}.

It is possible, however, to re-scale $\Psi^T$ by a sign (depending on the degrees of the arguments) to make \eqref{groot1} hold.
The solution can be found, e.g., in [GJ, Section 5.1].

Define 
\begin{equation}\label{groot3}
\Psi^{T,d}(v_1,\dots,v_k):=(-1)^{(k-1)\deg v_1+(k-2)\deg v_2+\dots+\deg v_{k-1}}\Psi^T(v_1,\dots,v_k)
\end{equation}
Then one has
\begin{lemma}
$d(\Psi^{T,d})=(d\Psi)^{T,d}$
\end{lemma}
\qed

We turn back to 	a coderivation $D$ of degree +1 defined on $T^\vee(V[1])$, such that $D^2=0$, giving an $A_\infty$ structure on $V$. It is given by its Taylor components
\begin{equation}\label{groot4}
\begin{aligned}
\ &\overline{m}_1\colon V[1]\to V[1][1]\\
&\overline{m}_2\colon (V[1])^{\otimes 2}\to V[1][1]\\
&\overline{m}_3\colon (V[1])^{\otimes 3}\to V[1][1]\\
&\dots
\end{aligned}
\end{equation}
For a cochain $c\in\Hom(V^{\otimes n},V)$, we define cochain $\overline{c}\in\Hom(V[1]^{\otimes n},V[1])[n-1]$ by
\begin{equation}\label{groot4demi}
\overline{c}=c^{T,d}
\end{equation}
That is, we start with cochains $\{m_i\colon V^{\otimes i}\to V[2-i]\}$, and define $\{\overline{m}_i\colon V[1]^{\otimes i}\to V[2]\}$, and consider them as the components of $D$, see \eqref{groot4}.

Recall the basic formula for the $\circ$-operation on the cochains in $\Hom(V[1]^{\otimes m},V[1])$.  Here $d_1$ and $d_2$ are the degrees of $\overline{c}_1\in\Hom(V[1]^{\otimes n},V[1])$ and $\overline{c}_2\in\Hom(V[1]^{\otimes n},V[1])$, for a graded vector space $V[1]$:
\begin{equation}\label{groot5}
\begin{aligned}
\ &\overline{c}_1\circ \overline{c}_2(x_1,\dots,x_{m+n-1})=\\
&\sum_{0\le i\le m-1}(-1)^{d_2(\deg x_1+\dots+\deg x_i+i)}
\overline{c}_1(x_1,\dots,x_i,\overline{c}_2(x_{i+1},\dots,x_{i+n}),x_{i+n+1},\dots,x_{m+n-1})
\end{aligned}
\end{equation}
Note that  $\deg x_i$ mean their degrees as elements of $V$ (not of $V[1]$), but $d_i$ mean the degrees of $\overline{c}_i$ as cochains in $\Hom(TV[1],V[1])$.

It is well-known that the equation $D^2=0$ is translated as 
$$
(\overline{m}_1+\overline{m}_2+\overline{m}_3+\dots)\circ (\overline{m}_1+\overline{m}_2+\overline{m}_3+\dots)=0
$$
which gives rise to the following identities:

For any fixed $n$, one has:
\begin{equation}\label{groot6}
\sum_{a+b=n}\sum_{0\le i\le a-1}(-1)^{\deg x_1+\dots+\deg x_i+i}\overline{m}_a(x_1,\dots,x_i,\overline{m}_b(x_{i+1},\dots,x_{i+b}),x_{i+b+1},\dots,x_{a+b-1})=0
\end{equation}

As the last step, we rewrite \eqref{groot6} in terms of $\{m_i\}$ instead of $\{\overline{m}_i\}$, where $\overline{m}_i=m_i^{T,d}$, see \eqref{groot3}:

\begin{equation}\label{groot7}
\sum_{a+b=n}\sum_{0\le i\le a-1}(-1)^{\deg x_1+\dots+\deg x_i+i+S}{m}_a(x_1,\dots,x_i,{m}_b(x_{i+1},\dots,x_{i+b}),x_{i+b+1},\dots,x_{a+b-1})=0
\end{equation}
where
\begin{equation}
\begin{aligned}
\ &S=((b-1)\deg x_{i+1}+(b-2)\deg x_{i+2}+\dots+\deg x_{i+b-1})+\\
&((a-1)\deg x_1+\dots+(a-i)\deg x_i+\\
&(a-i-1)\deg m_b(x_{i+1},\dots,x_{i+b})+\\
&(a-i-2)\deg x_{i+b+1}+\dots+\deg x_{a+b-2})
\end{aligned}
\end{equation}
Here we just plug in the sign \eqref{groot3} for both $m_1$ and $m_2$. We can finally plug in
$$
\deg m_b(x_{i+1},\dots,x_{i+b})=\deg x_{i+1}+\dots+\deg x_{i+b}+2-b
$$

\begin{remark}{\rm
The matter why the sign \eqref{groot3} should be always successively plugged in in the formulas of this type is the following. On the level of $V$, we want $m_1$ to be the differential on $V$, given as a part of structure of an object of the category $\Vect_d(\k)$. In particular, we want to use \eqref{groot2} in all computations with multi-linear forms like $m_i$ (and it is used, indeed, e.g. in the Leibniz rule for a dg associative algebra). More generally, we want to have an isomorphism of {\it dg Lie algebras} $\mathrm{Coder}(T^\vee(V[1]))$ (with the {\it natural} Lie bracket of coderivations) and $\oplus_{n\ge 0}\Hom(V^{\otimes n},V)$ with {\it some} Lie bracket on the r.h.s. (called the Gerstenhaber bracket), but this bracket is already fixed for $[\Psi_1,\Psi_2]$ where $\Psi_1\in \Hom(V,V)[1]$ is the underlying differential $d$ of $V$. It makes necessary to implement the sign \eqref{groot3} everywhere.
}
\end{remark}

\begin{equation}\label{groot8}
\begin{aligned}
\ &\sum_{r+s+t=n}(-1)^{\deg x_1+\dots+\deg x_r+r}\overline{f}_{r+t+1}(x_1,\dots, x_r,\overline{m}_s(x_{r+1},\dots,x_{r+s}),x_{r+s+1},\dots,x_n)=\\
&\sum_{i_1+\dots+i_r=n}\overline{m}_r(\overline{f}_{i_1}(x_1,\dots,x_{i_r}),\overline{f}_{i_2}(x_{i_1+1},\dots,x_{i_1+i_2}),\dots,\overline{f}_{i_r}(x_{i_1+\dots+i_{r-1}+1},\dots,x_n))
\end{aligned}
\end{equation}
(recall that the bar-degrees $\deg \overline{f}_i=0$ and $\deg \overline{m}_j=1$, for any $i,j$).

In the case of our interest, $m_n=0$ for $n\ge 3$. We rewrite \eqref{groot8} as the corresponding identities on ${f}_n$ only in this case:
\begin{equation}\label{groot9}
\begin{aligned}
\ &\sum_{1\le i\le n}(-1)^{\deg x_1+\dots+\deg x_{i-1}+S_1}f_n(x_1,\dots,x_{i-1},dx_i,x_{i+1},\dots,x_n)+\\
&\sum_{1\le i\le n-1}(-1)^{\deg x_1+\dots+\deg x_{i-1}+i-1+S_2}f_{n-1}(x_1,\dots,x_{i-1},m_2(x_i,x_{i+1}),x_{i+2},\dots,x_n)=\\
&(-1)^{S_3}df_n(x_1,\dots,x_n)+
\sum_{a+b=n}(-1)^{S_4}m_2(f_a(x_1,\dots,x_a),f_b(x_{a+1},\dots,x_n))
\end{aligned}
\end{equation}
where
\begin{equation}\label{groots1}
S_1=S_3=(n-1)\deg x_1+(n-2)\deg x_2+\dots+\deg x_{n-1}
\end{equation}
\begin{equation}\label{groots2}
S_2=(n-2)\deg x_1+\dots+(n-i)\deg x_{i-1}+(n-i)\deg x_i+(n-i-1)\deg x_{i+1}+(n-i-2)\deg x_{n-i-2}+\dots+\deg x_{n-1}
\end{equation}
\begin{equation}\label{groots4}
\begin{aligned}
\ &S_4=(\deg x_1+\dots+\deg x_a-a+1)+\\
&((a-1)\deg x_1+(a-2)\deg x_2+\dots+\deg x_{a-1})+((b-1)\deg x_{a+1}+(b-2)\deg x_{a+2}+\dots+\deg x_{n-1})
\end{aligned}
\end{equation}

\subsection{\sc A proof of Theorem \ref{mainyon}}
We need to prove \eqref{groot9} for general $n$, where 
$$
f_n(x_1,\dots,x_n)=(-1)^{(n-1)\deg x_1+(n-2)\deg x_2+\dots+\deg x_{n-1}}w_1(x_1)*w_2(x_2)*\dots*w_n(x_n)
$$
with the notations as in Section \ref{occupy1.2}.

Note that the total differential $d$ in $df_n(x_1,\dots,x_n)$ is the sum of three differentials
$$
d=d_\Cobar+d_\Bar^\sim+d_0
$$
where $d_0$ comes from the underlying differential in each component $X_s$, and $$d_\Bar^\sim|_{\Bar(X_L)[-1]^{\otimes k}}=(-1)^k d_\Bar$$ The differential $d$ in the first summand of \eqref{groot9} is nothing but the restriction of $d_0$ to $X_1$. Each $f_n$ is $d_0$-equivariant, so the result will follow from the following Proposition:

\begin{prop}\label{propainfty}
In the notations as above, the following identities hold:
\begin{itemize}
\item[(i)]
\begin{equation}\label{yon1}
(-1)^{S_3}(-d_\Bar)(f_n(x_1,\dots,x_n))=\sum_{i=1}^{n-1}(-1)^{\deg x_1+\dots+\deg x_{i-1}+i-1+S_2}f_{n-1}(x_1,\dots, m_2(x_i,x_{i+1}),\dots, x_n)
\end{equation}
\item[(ii)]
\begin{equation}\label{yon2}
(-1)^{S_3}d_\Cobar(f_n(x_1,\dots,x_n))=\sum_{a+b=n}(-1)^{S_4+1}m_2(f_a(x_1,\dots, x_a),f_b(x_{a+1},\dots, x_n))
\end{equation}
\end{itemize}
where the signs $S_i$ are as in \eqref{groots1},\eqref{groots2},\eqref{groots4}.
\end{prop}

\begin{proof}
Introduce
$$
\theta_n(x_1,\dots,x_n)=w_1(x_1)*w_2(x_2)*\dots*w_n(x_n)
$$
so that
$$
f_n(x_1,\dots,x_n)=(-1)^{(n-1)\deg x_1+(n-2)\deg x_2+\dots+\deg x_{n-1}}\theta_n(x_1,\dots,x_n)
$$
Then \eqref{yon1} and \eqref{yon2} are rewritten, in terms of $\theta$'s, as
\begin{equation}\label{yon1bis}
d_\Bar\theta_n(x_1,\dots,x_n)=\sum_{i=1}^{n-1}(-1)^{\deg x_1+\dots+\deg x_i+i}\theta_{n-1}(x_1,\dots,x_{i-1},m_2(x_i,x_{i+1}),\dots,x_n)
\end{equation}

\begin{equation}\label{yon2bis}
d_\Cobar\theta_n(x_1,\dots,x_n)=\sum_{a+b=n}(-1)^{\deg x_1+\dots+\deg x_a+a}m_2(\theta_a(x_1,\dots,x_a),\theta_b(x_{a+1},\dots,x_n))
\end{equation}
correspondingly.

Prove (i), in the form \eqref{yon1bis}. The bar-differential  is equal to
\begin{equation}\label{yon3}
d_\Bar(\theta_n(x_1\otimes\dots\otimes x_n))=\sum_{i=1}^{n-1}(-1)^{\deg x_1+\dots+\deg x_i+i} F_i(\theta_n(x_1\otimes\dots\otimes x_n))
\end{equation}
by \eqref{diffbarlb}. Then what need to compute 
\begin{equation}\label{yon4}
F_i(w_1(x_1)*_{X_n}w_2(x_2)*_{X_n}\dots*_{X_n}w_n(x_n))
\end{equation}

It is equal to
\begin{equation}\label{yon5}
F_i(w_1(x_1))*_{X_{n-1}}F_i(w_2(x_2))*_{X_{n-1}}\dots*_{X_{n-1}}F_i(w_n(x_n))
\end{equation}
Then we have a Lemma:
\begin{lemma}
Let $w_i^{(n)}$ be the maps $[n]\to [1]$ in $\Delta_0$ given by \eqref{yon0}, and let $F_1,\dots,F_{n-1}\colon [n-1]\to [n]$ be the face maps in $\Delta_0$.
Then one has in $\Delta_0^\opp$:
\begin{equation}\label{yon6}
F_i\circ w_k^{(n)}=\begin{cases}w_k^{(n-1)}&\text{ if  }k\le i\\
w_{k-1}^{(n-1)}&\text{ if  }k\ge i+1
\end{cases}
\end{equation}
\end{lemma}
The proof is a direct computation with the simplicial identities \eqref{simplid}.

\qed

Note that $w_k^{(n-1)}$ occurs twice (for a fixed $n$, and for a fixed $F_k$): as $F_k\circ w_k^{(n)}$ and as $F_{k}\circ w_{k+1}^{(n)}$.

We can now rewrite \eqref{yon5} as
\begin{equation}
w_1^{(n-1)}(x_1)*w_2^{(n-1)}(x_2)*\dots*[w_i^{(n-1)}(x_i)*w_i^{(n-1)}(x_{i+1})]*\dots*w_{n-1}^{(n-1)}(x_n)
\end{equation}
Now making use that all $w_i^{(n)}$ are maps of algebras (as $X_L$ is a Leinster monoid in $\Alg(k)$), we can rewrite
$$
w_i^{(n-1)}(x_i)*_{X_{n-1}}w_i^{(n-1)}(x_{i+1})=w_i^{(n-1)}(x_i*_{X_1}x_{i+1})
$$

Now \eqref{yon1bis} follows directly from our formula for the bar-differential, see \eqref{diffbarlb}, \eqref{diffbarbis}. It proves the statement (i) of Proposition.

\hspace{1mm}

Prove statement (ii) of the Proposition, in the form \eqref{yon2bis}. Recall that the coproduct $\Delta$ in $\Bar(X_L)$ at $X_n$ is
\begin{equation}\label{yon7}
\Delta=\beta_{1,n-1}+\beta_{2,n-2}+\dots +\beta_{n-1,1}
\end{equation}
Use the ``Sweedler notations'' 
$$
\beta_{i,j}(x)=\beta_{i,j}^\prime(x)\otimes \beta_{i,j}^\pprime(x)
$$
where, for a homogeneous $x$, both factors are homogeneous.

The cobar-differential in $\Cobar(\Bar(X_L))$ applied to $x\in X_n\subset \Bar(X_L)$ is 
\begin{equation}\label{yon7demi}
d_\Cobar(x)=\sum_{a+b=n}(-1)^{\deg\beta_{a,b}^\prime(x)}\beta^\prime_{a,b}(x)\otimes\beta_{a,b}^\pprime(x)\in X_a\otimes X_b\subset \Bar(X_L)^{\otimes 2}
\end{equation}
(see \eqref{diffcobarclass} for the sign).

Let $a+b=n$, $a,b\ge 1$. We claim that
\begin{equation}\label{yon8}
\beta_{a,b}\theta_n(x_1,\dots, x_n)=\theta_a(x_1,\dots, x_a)\otimes \theta_b(x_{a+1},\dots,x_n)
\end{equation}
where the two factors in the r.h.s. are $\beta_{a,b}^\prime$ and $\beta_{a,b}^\pprime$, correspondingly.

The l.h.s. of \eqref{yon8} is
\begin{equation}\label{yon9}
\beta_{a,b}\Bigl(w_1(x_1)*_{X_n}\dots*_{X_n}w_n(x_n)\Bigr)=\beta_{a,b}(w_1(x_1))*_{X_a\otimes X_b}\dots*_{X_a\otimes X_b}\beta_{a,b}(w_n(x_n))
\end{equation}
as all $\beta_{m,n}$ are maps in $\Alg(\k)$, for a Leinster monoid in $\Alg(\k)$.

\begin{lemma}\label{lemmaay}
Denote by $1^{(c)}\in X_c$ the element $D^{(c)*}(1)$, where $D^{(c)}$ the (only) degeneracy map $D^{(c)}\colon [c]\to [0]$.
One has:
\begin{equation}\label{yon10}
\beta_{a,b}(w_i^{(n)}(x))=\begin{cases}
w_i^{(a)}(x)\otimes 1^{(b)}&\text{ for  }i\le a\\
1^{(a)}\otimes w_{i-a}^{(b)}(x)&\text{ for  }i\ge a+1
\end{cases}
\end{equation}
\end{lemma}
\begin{proof}
We use that $\beta$ is a morphism of bi-functors $\Delta^\opp_0\times\Delta^\opp_0\to\Alg(\k)$.
After we have identified $[a]\otimes [b]=[a+b]=[n]$ in $\Delta_0$, we have for morphisms:
\begin{equation}\label{yon11}
D^{(a)}\otimes w_j^{(b)}=w_{a+j}^{(n)}\text{ for $1\le j\le b$,  and  }w_i^{(a)}\otimes D^{(b)}=w_{a}^{(n)}\text{ for $1\le i\le a$}
\end{equation}
where $w_i^{(a)}\colon [a]\to [1]$, $w_i^{(b)}\colon [b]\to [1]$ are the corresponding morphisms in $\Delta_0$.

The identities \eqref{yon11} follow directly from the definition of the monoidal product (of morphisms) in $\Delta_0$ given just above the Definition \ref{catwe}.

Now the identities \eqref{yon11} give rise to the commutative diagrams (in the corresponding cases):
\begin{equation}
\begin{aligned}
\ &\xymatrix{
X_1\ar[rr]^{\beta_{0,1}}\ar[d]_{w_j^{(n)}}&&X_0\otimes X_1\ar[d]^{D^{(a)}\otimes w_j^{(b)}}\\
X_n\ar[rr]^{\beta_{a,b}}&&X_a\otimes X_b
}
&
\xymatrix{
X_1\ar[rr]^{\beta_{1,0}}\ar[d]_{w_i^{(n)}}&&X_1\otimes X_0\ar[d]^{w_i^{(a)}\otimes D^{(b)}}\\
X_n\ar[rr]^{\beta_{a,b}}&&X_a\otimes X_b
}
\end{aligned}
\end{equation}
Recall our assumption that $X_0=\k$, which implies by \eqref{colax2} that the upper vertical arrows of both diagrams are the identity maps. 

\end{proof}

It follows from \eqref{yon7demi} and \eqref{yon8} that 
\begin{equation}\label{yon100}
d_\Cobar(\theta_n(x_1,\dots,x_n))=\sum_{a+b=n}(-1)^{\deg x_1+\dots+\deg x_a+a-1}\theta_a(x_1,\dots,x_a)\otimes \theta_b(x_{a+1},\dots,x_n)
\end{equation}
as $\deg \theta_a(x_1,\dots,x_a)$, as of an element of $\Cobar(\Bar(X_L))$, is $\deg x_1+\dots+\deg x_a-a+1$.

The statement (ii) of Proposition \ref{propainfty} is proven.
\end{proof}

We have proved the assertion of Theorem \ref{mainyon} that the maps $\phi_n$ given in \eqref{yonbasic} are the Taylor components of an $A_\infty$ morphism
$X_1\to \Cobar(\Bar(X_L))$. It remains to show that it is an $A_\infty$ quasi-isomorphism.

For the latter claim we need to show that the map $\phi_1\colon X_1\to\Cobar(\Bar(X_L))$ is a quasi-isomorphism of complexes.
It it the claim of Theorem \ref{theorem27} (which holds in a greater generality of Leinster monoids in $\Vect(\k)$), in the proof of which the map $\phi_1$ appears under the name $i$.

\qed

\appendix

\section{\sc A more conceptual derivation of the Leibniz rule for the Eilenberg-MacLane map for graded Leinster monoids in the case of 0-grading}\label{sectionapp}
In Appendix, we provide an alternative and a more conceptual proof of Theorem \ref{emmnewtheorem}(i), replacing the previous long routine computation, for the particular case when each component $X_n$ of a graded Leinster monoid $X_L$ has the only zero graded components:
$$
X_n=X_n^{0,\dots,0}
$$
$(n\ge 1)$.

In fact, in this Appendix we assume $X_L$ to be just a functor $X_L:\Delta_0^\opp\to\Vect^0(\k)$, where 
$$
\Vect^0(\k)\subset\Vect(\k)
$$
is the full sub-category of complexes concentrated in homological degree 0. We do not assume any colax-monoidal map $\beta$, so we do not assume $X_L$ is a Leinster monoid. Our treatment here is entirely independent on Section \ref{sectionemmnew}; the idea is to deduce the statement from the classical Eilenberg-MacLane map result [EM], saying that the Eilenberg-MacLane map \eqref{emform} is a map of complexes, that is, that the Leibniz rule holds.

For a functor $X_\ldot\colon \Delta_0^\opp\to\Vect^0(\k)$, we denote by $c(X_\ldot)$ its {\it chain complex}.
It is defined as the complex
$$
\dots\rightarrow \underset{\text{degree -3}}{X_3}\xrightarrow{d_3}\underset{\text{degree -2}}{X_2}\xrightarrow{d_2}\underset{\text{degree -1}}{X_1}\to 0
$$
where $d_n\colon X_n\to X_{n-1}$ is defined as
\begin{equation}\label{chainc}
d_n=F_1-F_2+F_3-\dots+(-1)^{n}F_{n-1}
\end{equation}
It is clear that 
\begin{equation}
d^2=0
\end{equation}
Our notation $c(X_\ldot)$ for this chain complex, with lower-case {\it c}, indicates that it is a chain complex of a $\Delta_0^\opp$-vector space. The chain complex of a $\Delta^\opp$-vector space $Y_\ldot$ we denote $C(Y_\ldot)$, as in before.

Note that for a graded Leinster monoid $X_L$ all whose components $X_n$ are equal to $X_n^{à,\dots,0}$, the chain complex $c(X_L)$ is the same that the bar-complex $\Bar(X_L)$, defined in Section \ref{sectiongraded}.

For a functor $X_\ldot\colon \Delta^\opp\to\Vect^0(\k)$, we denote by $\Res(X_\ldot)\colon \Delta_0^\opp\to\Vect^0(\k)$ the restriction of $X_\ldot$, associated with natural embedding 
$$
i\colon \Delta_0\to\Delta
$$

For $X_\ldot,Y_\ldot\colon \Delta_0^\opp\to\Vect(\k)$, define the map
$$
\nabla^c\colon c(X_\ldot)\otimes c(Y_\ldot)\to c(X_\ldot\otimes Y_\ldot) 
$$
by \eqref{emform}. One needs to prove that $\nabla^c$ is a map of complexes.

\begin{theorem}\label{theorememwxxx}
\begin{itemize}
\item[(i)] For $X_\ldot,Y_\ldot\colon \Delta^\opp\to \Vect^0(\k)$, the map
\begin{equation}
\nabla^c\colon c(\Res(X_\ldot))\otimes c(\Res(Y_\ldot))\to c(\Res(X_\ldot)\otimes\Res(Y_\ldot))
\end{equation}
is a map of complexes;
\item[(ii)] for any $X_\ldot,Y_\ldot\colon \Delta_0^\opp\to \Vect^0(\k)$, the map
\begin{equation}
\nabla^c\colon c(X_\ldot)\otimes c(Y_\ldot)\to c(X_\ldot\otimes Y_\ldot)
\end{equation}
is a map of complexes.
\end{itemize}
\end{theorem}
\begin{proof}
It is clear that (i) is a particular case of (ii). We prove (i) first, and then deduce the general case (ii).

\hspace{1mm}

{\it Proof of (i):}

A key point of the proof is found in [EM], in their proof that the Eilenberg-MacLane map $\nabla$  \eqref{emform} defines {\it a map of complexes} \eqref{emmap}.
In loc.cit. Theorem 5.2, Eilenberg and MacLane divide their proof that $\nabla$ is a map of complexes with the chain differential
\eqref{chain2}, into two parts. First of all, they prove that $\nabla$ is compatible by the Leibniz rule with $F_0$ alone, and after that, they prove the Leibniz rule for the ``tail'' $d^*=d-F_0=-F_1+F_2-\dots +(-1)^nF_n$.

Our claim (i) follows from this Eilenber-MacLane result, as follows. We deduce from the Leibniz rule for $F_0$ the Leibniz rule with the ``rightmost'' extreme face operator $F_n\colon X_n\to X_{n-1}$. 

To this end, consider the functor $\tau\colon\Delta\to\Delta$ which is identity on the objects, and for a morphism $f\colon [m]\to [n]$ the morphism $\tau(f)$ is defined as $\tau(f)[a]=n-\tau(m-a)$, where $a\in \{0<1<\dots<m\}$. Then consider the simplicial vector space $X_\ldot^\tau$, $Y_\ldot^\tau$, they have the same chain complexes (the differential may change the total sign), and $F_0$ for $X_\ldot^\tau$ becomes $F_{\mathrm{top}}$. Thus, the Eilenberg-MacLane claim in loc.cit. is applied to $F_{\mathrm{\top}}$ instead of $F_0$ as well. It then implies that $\nabla$ is compatible by the Leibniz rule with $d-F_0-F_{\mathrm{top}}$, which is our restricted chain differential \eqref{chainc}.

\hspace{1mm}

{\it Proof of (ii):}

Recall that for any $\k$-linear small categories $\mathfrak{a}$ and $\mathfrak{b}$, and for a $\k$-linear functor $f\colon \mathfrak{a}\to\mathfrak{b}$, there is a pair of adjoint functors
\begin{equation}
\mathrm{Ind}\colon \Mod(\mathfrak{a})\rightleftarrows \Mod(\mathfrak{b})\colon \mathrm{Res}
\end{equation}
with $\mathrm{Ind}$ the left adjoint.

Here $\mathrm{Res}$ is the restriction functor, sending a functor $M\colon \mathfrak{b}\to\Vect(\k)$ to the composition
\begin{equation}
\mathfrak{a}\xrightarrow{f}\mathfrak{b}\xrightarrow{M}\Vect(\k)
\end{equation}
and $\mathrm{Ind}$ is the induction functor, sending $N\in \Mod(\mathfrak{a})$ to the $\mathfrak{b}$-module
\begin{equation}
\mathrm{Ind}(N)=\mathfrak{b}\otimes_\mathfrak{a}N
\end{equation}
In particular, there is the adjunction map
\begin{equation}\label{eqemphi}
\Phi\colon N\to \mathrm{Res}\circ \mathrm{Ind} (N)
\end{equation}
for any $N\in \Mod(\mathfrak{a})$. The following claims are well-know:
\begin{lemma}\label{lemmaembis1}
In the above notations, consider for $X\in\mathfrak{a}$ the Yoneda module $h_X$ defined as $h_X(Y)=\Hom_{\mathfrak{a}}(X,Y)$.
Then for any $X\in\mathfrak{a}$ one has:
\begin{equation}
\mathrm{Ind}(h_X)=h_{f(X)}
\end{equation}
where in the right-hand side $h_{f(X)}$ is a Yoneda module in $\mathfrak{b}$.
\end{lemma}
\qed

\begin{lemma}\label{lemmaembis2}
Let $\mathfrak{a}$ be a small $\k$-linear category. Then the set $\{h_X\}$ of Yoneda modules for all $X\in\mathfrak{a}$ is a set of generators for $\Mod(\mathfrak{a})$.
\end{lemma}
\qed

If categories $\mathfrak{a}$ and $\mathfrak{b}$ are set-enriched, we can produce out of them $\k$-linear categories, defining the $\Hom$-spaces as $\k$-spans of the corresponding $\Hom$-sets. Then the above constructions are applicable to the small set-enriched categories as well. \\[4pt]

We consider the set-enriched categories $\Delta_0$ and $\Delta$ as $\k$-linear categories.

The claim that $\nabla^c$ is a map of complexes means that for any $X_\ldot,Y_\ldot\colon \Delta_0^\opp\to\Vect(\k)$ we have some identity, expressing the commutation relation with the differential \eqref{chainc}. To prove this identity for all $X_\ldot,Y_\ldot$, it is enough to prove it set of generators. By Lemma \ref{lemmaembis2},
one can choose the set of all Yoneda modules $h_{[n]}=\Hom_{\Delta_0}(-,[n])$ as the set of generators.

\begin{lemma}\label{lemmaembis3}
For any Yoneda module $h_{[n]}=\Hom_{\Delta_0}(-,[n])$, the adjunction \eqref{eqemphi}
\begin{equation}
\Phi_{[n]}\colon h_{[n]}(-)\to \mathrm{Res}\circ\mathrm{Ind}(h_{[n]})(-)
\end{equation}
is a point-wise embedding.
\end{lemma}
\begin{proof}
By Lemma \ref{lemmaembis1}, $\mathrm{Ind}(h_{[n]})$ is the simplicial vector space $H_{[n]}(-)=\Hom_{\Delta}(-,[n])$.
Then $\Phi_{[n]}([m])$ sends $\Hom_{\Delta_0}([m],[n])$ to $\Hom_{\Delta}([m],[n])$. It is easy to check that $\Phi_{[n]}([m])$ 
coincides with the natural embedding.
\end{proof}

\qed

Now take $X_\ldot=h_{[m]}$ and $Y_\ldot=h_{[n]}$. It is enough to prove the commutation relation of $\nabla^c$ with the differential \eqref{chainc} for these generating modules.
The claim (ii) of the Theorem the results holds for $\mathrm{Res}\circ\mathrm{Ind}(h_{[m]})$ and $\mathrm{Res}\circ\mathrm{Ind}(h_{[n]})$, by the claim (i) proven above. Then it holds for $h_{[m]}$ and $h_{[n]}$, because $\Phi_{[n]}$ and $\Phi_{[m]}$ are point-wise embeddings.

\end{proof}

\bigskip

{\small
\noindent {\sc Universiteit Antwerpen, Campus Middelheim, Wiskunde en Informatica, Gebouw G\\
Middelheimlaan 1, 2020 Antwerpen, Belgi\"{e}}}

\bigskip

\noindent{{\it e-mail}: {\tt Boris.Shoikhet@uantwerpen.be}}

\end{document}